\documentclass[12pt]{article}

\usepackage{bbm}
\usepackage{amsmath, amssymb, amsthm}
\numberwithin{equation}{section}
\allowdisplaybreaks[1]
\usepackage[breaklinks=true]{hyperref}
\hypersetup{
    colorlinks,%
    citecolor=red,%
    filecolor=blue,%
    linkcolor=blue,%
    urlcolor=red
}

\usepackage{bbm}
\usepackage[breaklinks=true]{hyperref}

\numberwithin{equation}{section}
\allowdisplaybreaks[4]

\theoremstyle{plain}
\newtheorem{theorem}{Theorem}[section]
\newtheorem{proposition}[theorem]{Proposition}
\newtheorem{lemma}[theorem]{Lemma}
\newtheorem{corollary}[theorem]{Corollary}

\theoremstyle{definition}
\newtheorem{definition}{Definition}[section]
\newtheorem{example}[definition]{Example}
\newtheorem{remark}[definition]{Remark}

\makeatletter

\def\^#1{\ifmmode {\mathaccent"705E #1} \else {\accent94 #1} \fi}
\def\~#1{\ifmmode {\mathaccent"707E #1} \else {\accent"7E #1} \fi}

\def\*#1{#1^\ast}
\edef\-#1{\noexpand\ifmmode {\noexpand\bar{#1}} \noexpand\else \-#1\noexpand\fi}
\def\>#1{\vec{#1}}
\def\.#1{\dot{#1}}

\def\atop{\@@atop}
\def\%#1{\mathcal{#1}}

\renewcommand{\leq}{\leqslant}
\renewcommand{\geq}{\geqslant}
\renewcommand{\phi}{\varphi}
\newcommand{\eps}{\varepsilon}
\newcommand{\D}{\Delta}
\newcommand{\eq}{\eqref}

\newcommand{\dtv}{\mathop{d_{\mathrm{TV}}}}
\newcommand{\dw}{\mathop{d_{\mathrm{W}}}}
\newcommand{\dk}{\mathop{d_{\mathrm{K}}}}

\newcommand{\bigo}{\mathrm{O}}
\newcommand{\lito}{\mathrm{o}}

\def\ER{Erd\H{o}s-R\'enyi}
\newcommand{\I}{\mathbb{I}}
\newcommand{\Po}{\mathop{\mathrm{Po}}}

\newcommand{\Bi}{\mathop{\mathrm{Bi}}}

\newcommand{\Exp}{\mathop{\mathrm{Exp}}}
\newcommand{\Ge}{\mathop{\mathrm{Ge}}\nolimits}

\newcommand{\IE}{\mathbbm{E}}
\newcommand{\IP}{\mathbbm{P}}
\newcommand{\Var}{\mathop{\mathrm{Var}}\nolimits}
\newcommand{\Cov}{\mathop{\mathrm{Cov}}}

\newcommand{\sgn}{\mathop{\mathrm{sgn}}}

\newcommand{\IN}{\mathbbm{N}}

\newcommand{\IR}{\mathbbm{R}}

\def\be#1\ee{\begin{equation*}#1\end{equation*}}
\def\ben#1\ee{\begin{equation}#1\end{equation}}

\def\bes#1\ee{\begin{equation*}\begin{split}#1\end{split}\end{equation*}}
\def\besn#1\ee{\begin{equation}\begin{split}#1\end{split}\end{equation}}

\def\bg#1\ee{\begin{gather*}#1\end{gather*}}
\def\bgn#1\ee{\begin{gather}#1\end{gather}}

\def\bm#1\ee{\begin{multline*}#1\end{multline*}}
\def\bmn#1\ee{\begin{multline}#1\end{multline}}

\def\ba#1\ee{\begin{align*}#1\end{align*}}
\def\ban#1\ee{\begin{align}#1\end{align}}

\def\bbbklr#1{\biggl(#1\biggr)}

\def\bbbkle#1{\biggl[#1\biggr]}

\def\bbbklg#1{\biggl\{#1\biggr\}}

\def\norm#1{\Vert#1\Vert}

\def\abs#1{\vert#1\vert}

\newcount\minute
\newcount\hour
\newcount\hourMins
\def\now{%
\minute=\time%
\hour=\time \divide \hour by 60%
\hourMins=\hour \multiply\hourMins by 60%
\advance\minute by -\hourMins%
\zeroPadTwo{\the\hour}:\zeroPadTwo{\the\minute}%
}
\def\zeroPadTwo#1{\ifnum #1<10 0\fi#1}

\def\blfootnote{\xdef\@thefnmark{}\@footnotetext}
\def\blfootnote{\xdef\@thefnmark{}\@footnotetext}

\def\var{\Var}

\newcommand{\ed}{\stackrel{d}{=}}
\def\cov{\Cov}

\def\dh{d_\%H}

\def\P{\%P_\lambda}
\def\s{\sigma}

\def\Geo{\Ge^0}

\begin{document}

\title{Fundamentals of Stein's method}
\author{Nathan Ross}
\date{\it University of California,
Berkeley}
\maketitle

\begin{abstract} 
This survey article 
discusses the main concepts and techniques of Stein's method
for distributional approximation by the
normal, Poisson, exponential, and geometric distributions, and also 
its relation to
concentration inequalities.
The material is 
presented at a level accessible to beginning graduate students
studying probability 
with the main emphasis
on the themes that are common to these topics and also to
much of the Stein's method literature.
\end{abstract}
\maketitle

\tableofcontents

\section{Introduction}

The fundamental example of the type of result we will deal with in this article is
the following version of the classical Berry-Esseen bound for the central limit theorem.
\begin{theorem}\label{thm1}\cite{ber41}
Let $X_1, X_2,\ldots$ be i.i.d. random variables with $\IE\abs{X_1}^3<\infty$, $\IE[X_1]=0$,
and $\var(X_1)=1$.  If $\Phi$ denotes the c.d.f. of a standard normal distribution and $W_n=\sum_{i=1}^nX_i/\sqrt{n}$, then
\ba
\left|\IP(W_n\leq x)-\Phi(x)\right|\leq  1.88 \frac{\IE\abs{X_1}^3}{\sqrt{n}}.
\ee
\end{theorem}

The theorem quantifies the error in the central limit theorem and has many
related
embellishments such as assuming independent, but not identically distributed 
variables, or allowing a specified dependence structure.  
The proofs of such results typically rely on characteristic function (Fourier) analysis
whereby showing convergence is significantly easier
than obtaining error bounds.

More generally, a central theme of probability theory is proving distributional limit theorems,
and for the purpose of approximation it is of interest to
quantify the rate of convergence in such results.
However, many of the methods commonly employed to show distributional convergence (e.g. Fourier analyisis and method of moments)
only possibly yield an error rate after serious added effort.  Stein's method
is a technique that can quantify the error in the approximation of
one distribution by another in a variety of metrics; note that this last remark has
a wider scope than the discussion above.

Stein's method was initially conceived by Charles Stein in the seminal
paper \cite{stn72} to provide errors in the approximation by the normal distribution of the distribution of
the sum of dependent random variables
of a certain structure.  However, the ideas presented in \cite{stn72} are sufficiently abstract and powerful
to be able to work well beyond that intended purpose, applying to approximation of
more general random variables by distributions other than the normal (such as the Poisson, exponential, etc).

Broadly speaking,
Stein's method has two components:
the first is a framework to convert the problem of bounding the error 
in the approximation of one distribution of interest by another, well understood distribution (e.g. the normal)
into a problem
of bounding the expectation of a certain functional of the
random variable of interest (see \eq{31} for the
normal distribution and \eq{1019} for the Poisson). 
The second component of Stein's method
are techniques to bound the expectation 
appearing in the first component; Stein 
appropriately refers to this step as ``auxiliary randomization."   
With this in mind, it is no surprise that Stein's monograph \cite{stn86}, which reformulates
the method in a more coherent form than \cite{stn72}, is titled
``Approximate Computation of Expectations."

There are now hundreds of papers expanding and applying this 
basic framework above.  For the first component,
converting to a problem of bounding
a certain expectation involving the distribution of interest
has been achieved for many well-known distributions.
Moreover, canonical methods have been established for achieving
this conversion for new distributions \cite{chsh11,rei05} (although by no means is
this process easy or guaranteed to be fruitful). 

For the second component, there is now an array of coupling techniques available
to bound these functionals for various distributions.
Moreover, these coupling techniques can be used in other
types of problems which can be distilled into bounding expectations
of a function of a distribution of interest.
Two examples of the types of
problems where this program has succeeded are
concentration inequalities \cite{cha07a,ghgo11a,ghgo11} (using the well known
Proposition \ref{prop211} below), and
local limit theorems \cite{rr10}.
We will cover the former example in this article.

The purpose of this document is to attempt to elucidate
the workings of these two components at a basic level 
in order to help make Stein's method more 
accessible to the uninitiated.
There are numerous other introductions to Stein's method which this document
draws from,
mainly \cite{cha07d,cgs11} for Normal approximation, \cite{bhj92,cdm05} for
Poisson approximation, and an amalgamation of related topics in the collections \cite{bc05,diho04}.
Most of these references focus on one distribution or
variation of Stein's method in order to achieve depth,
so there are themes and ideas that appear
throughout the method which can be
difficult to glean from these
references. 
We hope to capture these fundamental concepts in uniform language
to give easier entrance to the vast literature
on Stein's method and applications.
A similar undertaking but with smaller scope can be found
in Chapter 2 of \cite{rope07}, which also serves as a nice introduction to
the basics of Stein's method.

Of course 
the purpose of Stein's method is to prove approximation results, so
we will illustrate concepts in examples and applications,
many of which are combinatorial in nature.
In order to facilitate exposition,
we will typically work out 
examples and applications only in the most straightforward way 
and 
provide pointers to the literature where variations of the arguments
produce more thorough results.

The layout of this document is as follows.  In Section \ref{stnorm},
we discuss the basic framework of the first component above
in the context of Stein's method for normal approximation, 
since this setting is the
most studied and contains many of the concepts we will need later.
In Section \ref{norm} we discuss the commonly employed couplings 
used in normal approximation to achieve the second component above.  
We follow the paradigm of these two sections in discussing Stein's method 
for Poisson approximation
in Section \ref{poi}, exponential approximation in Section \ref{exp}, and geometric approximation
in Section \ref{geo}.  In the final Section \ref{con} we discuss how
to use some of the coupling constructions of Section \ref{norm} to prove 
concentration inequalities.

We conclude this section with a discussion of necessary background and notation. 

\subsection{Background and notation}

This is a document based on a graduate course given at U.C. Berkeley in the Spring semester of 2011
and is aimed at an audience having seen probability theory at the level of \cite{grst01}.
That is, we do not rely heavily on measure theoretic concepts, but 
exposure at a heuristic level to concepts such as sigma-fields will be useful. 
Also, basic Markov chain theory concepts such as reversibility are assumed along
with the notion of coupling random variables which will be used frequently in the sequel.

Many of our applications will concern various statistics of 
\ER\ random graphs.  We say
$G=G(n,p)$ is an \ER\ random graph on $n$ vertices with edge probability $p$ 
if for each pair of $\binom{n}{2}$ vertices, there is an edge connecting the vertices with probability
$p$ (and no edge connecting them with probability $1-p$), independent of all
other connections between other pairs of vertices.  These objects are
a simple and classical model of networks that 
are well studied; see \cite{bol01,jlr00} for book length treatments.

For a set $A$, we write $\I[\cdot\in A]$ to denote the function which is one on $A$ and $0$ otherwise.  We
write $g(n)\asymp f(n)$ if $g(n)/f(n)$ tends to a positive constant as $n\to\infty$,
and $g(n)=\bigo(f(n))$ if $g(n)/f(n)$ is bounded as $n\to\infty$.

Since Stein's method is mainly concerned with bounding
the distance between probability distributions in a given metric,
we now discuss the metrics we will use.

\subsubsection{Probability Metrics}\label{probmet}
For two probability measures $\mu$ and $\nu$, the probability metrics we will use have the form 
\ban
d_{\%H}(\mu,\nu)=\sup_{h\in\%H}\left|\int h(x)d\mu(x)-\int h(x)d\nu(x)\right|, \label{23}
\ee
where $\%H$ is some family of ``test" functions. For random variables $X$ and $Y$ with
respective laws $\mu$ and $\nu$, we will abuse notation and write
$d_{\%H}(X,Y)$ in place of $d_{\%H}(\mu,\nu)$.  

We now detail examples of metrics of this form along with some useful properties and relations.
\begin{enumerate}
\item By taking $\%H=\{\I[\cdot \leq x]: x\in\IR\}$ in \eqref{23}, we obtain the \emph{Kolmogorov} metric, which we denote $\dk$.
The Kolmogorov metric is the maximum distance between distribution functions,
so a sequence of distributions converging to a fixed distribution in this metric implies weak convergence. 
\item By taking $\%H=\{h:\IR\rightarrow\IR: \abs{h(x)-h(y)}\leq\abs{x-y}\}$ in \eqref{23}, we obtain the \emph{Wasserstein} metric,
which we denote $\dw$.
The Wasserstein metric is a common metric occurring in many contexts and will be the main metric we use for 
approximation by continuous distributions. 
\item By taking $\%H=\{\I[A\in\IR]: A \in\mbox{Borel}(\IR)\}$ in \eqref{23}, we obtain the \emph{total variation} metric, which we denote $\dtv$.
We will use the total variation metric for approximation by Œdiscrete distributions. 
\end{enumerate}
\begin{proposition}\label{prop22}
Retaining the notation for the metrics above, we have the following.
\begin{enumerate}
\item For random variables $W$ and $Z$, $\dk(W,Z)\leq\dtv(W,Z)$.
\item If the random variable $Z$ has Lebesgue density bounded by $C$,
then for any random variable $W$,
\ba
\dk(W,Z)\leq \sqrt{2C \dw(W,Z)}.
\ee
\item For $W$ and $Z$ random variables taking values in a discrete space $\Omega$, 
\ba
\dtv(W,Z)=\frac{1}{2}\sum_{\omega\in\Omega}\abs{\IP(W=\omega)-\IP(Z=\omega)}.
\ee
\end{enumerate}
\end{proposition}
\begin{proof}
The first item follows from the fact that the supremum on the right side of the inequality is over a larger set, and the third item is
left as an exercise.   For the second item, consider the functions 
$h_x(w)=\I[w\leq x]$, and the `smoothed' $h_{x,\eps}(w)$ defined to be one for $w\leq x$, zero for $w>x+\eps$, and linear between.  
Then we have
\ba
\IE h_x(W)- \IE h_x(Z)& = \IE h_x(W)- \IE h_{x,\eps}(Z) +\IE h_{x,\eps}(Z)- \IE h_x(Z) \\
	&\leq \IE h_{x,\eps}(W)- \IE h_{x,\eps}(Z) + C \eps/2 \\
	&\leq \dw(W,Z)/\eps +C \eps/2.
\ee
Taking $\eps=\sqrt{2\dw(W,Z)/C}$ shows half of the desired inequality and a similar argument yields the other half.
\end{proof}

Due to its importance in our framework, we reiterate the implication
of Item 2 of the proposition that a bound on the Wasserstein metric between a given distribution
and the normal or exponential distribution immediately yields a bound
on the Kolmogorov metric.

\section{Normal Approximation}\label{stnorm}

The main idea behind Stein's method of distributional approximation is to replace the
characteristic function typically used to show distributional convergence with a \emph{characterizing operator}.
\begin{lemma}[Stein's Lemma]\label{lem1}
Define the functional operator $\mathcal{A}$ by
\ba
\mathcal{A}f(x)=f'(x)-xf(x).
\ee
\begin{enumerate}
\item If the random variable $Z$ has the standard normal distribution, then 
$\IE\%{A}f(Z)=0$ for all absolutely continuous $f$ with $\IE \abs{f'(Z)}<\infty$.
\item If for some random variable $W$, $\IE\%{A}f(W)=0$ for all absolutely continuous functions
$f$ with $\IE \abs{f'(Z)}<\infty$, then $W$ has the standard normal distribution.
\end{enumerate}
The operator $\%A$ is referred to as a characterizing operator of the standard normal distribution.
\end{lemma}

Before proving Lemma \ref{lem1}, we record the following lemma and then observe a consequence.
\begin{lemma}\label{lem2}
If $\Phi(x)$ is the c.d.f. of the standard normal distribution, then the unique bounded solution $f_x$ of the differential equation
\ban
f'_x(w)-wf_x(w)=\I[w\leq x] - \Phi(x) \label{1a}
\ee
is given by 
\ba 
f_x(w)&=e^{w^2/2}\int_w^{\infty} e^{-t^2/2}\left(\Phi(x)-\I[t\leq x]\right) dt \\
 &=-e^{w^2/2}\int_{-\infty}^w e^{-t^2/2}\left(\Phi(x)-\I[t\leq x]\right) dt.
\ee
\end{lemma}

Lemmas \ref{lem1} and \ref{lem2} are at the heart of Stein's method; observe the following corollary.
\begin{corollary}\label{cor1}
If $f_x$ is as defined in Lemma \ref{lem2}, then for any random variable $W$, 
\ban
\abs{\IP(W\leq x) - \Phi(x)}= \abs{\IE[f'_x(W)-Wf_x(W)]}. \label{1}
\ee 
\end{corollary}
Although Corollary \ref{cor1} follows directly from Lemma \ref{lem2}, it is important to note that Lemma \ref{lem1}
suggests that \eqref{1} may be a fruitful equality.  That is, the left hand side of \eqref{1} is zero for all $x\in \IR$ if and only if
$W$ has the standard normal distribution.  Lemma \ref{lem1} indicates that the right hand side of \eqref{1} also has this property. 

\begin{proof}[Proof of Lemma \ref{lem2}]
The method of integrating factors shows that
\ba
\frac{d}{dw}\left(e^{-w^2/2}f_x(w)\right)=e^{-w^2/2}\left(\I[w\leq x] - \Phi(x)\right),
\ee
which after integrating and considering the homogeneous solution implies that
\ban
f_x(w)&=e^{w^2/2}\int_w^{\infty} e^{-t^2/2}\left(\Phi(x)-\I[t\leq x]\right) dt + Ce^{w^2/2} \label{1aa}
\ee
is the general solution of \eqref{1a} for any constant $C$.  To show that \eqref{1aa}
is bounded for $C=0$ (and then clearly unbounded for other values of $C$) we use 
\ba
1-\Phi(w)\leq \min\left\{\frac{1}{2}, \frac{1}{w\sqrt{2\pi}}\right\} e^{-w^2/2}, \hspace{4mm} w>0,
\ee
which follows by considering derivatives.  From this point we use the representation
\begin{displaymath}
   f_x(w) = \left\{
     \begin{array}{lr}
       \sqrt{2\pi}e^{w^2/2}\Phi(w)(1-\Phi(x)), & w\leq x \\
        \sqrt{2\pi}e^{w^2/2}\Phi(x)(1-\Phi(w)), & w>x 
     \end{array}
   \right.
\end{displaymath} 
to obtain that $\norm{f_x}\leq\sqrt{\frac{\pi}{2}}$.
\end{proof}

\begin{proof}[Proof of Lemma \ref{lem1}]
We first prove Item $1$ of the lemma.  Let $Z$ be a standard normal random variable and let $f$ be absolutely continuous such that $\IE \abs{f'(Z)}<\infty$.  Then we have
the following formal calculation (justified by Fubini's Theorem) which is essentially integration by parts.
\ba
\IE f'(Z) &=\frac{1}{\sqrt{2\pi}} \int_\IR e^{-t^2/2} f'(t) dt \\
	&=\frac{1}{\sqrt{2\pi}} \int_0^\infty f'(t)\int_t^\infty w e^{-w^2/2} dw dt +\frac{1}{\sqrt{2\pi}}  \int_{-\infty}^0 f'(t)\int_{-\infty}^t w e^{-w^2/2} dw dt \\
	&=\frac{1}{\sqrt{2\pi}} \int_0^\infty w e^{-w^2/2} \left[ \int_0^w f'(t) dt \right]dw +\frac{1}{\sqrt{2\pi}}  \int_{-\infty}^0 w e^{-w^2/2} \left[ \int_w^0 f'(t) dt \right]dw \\
	&=\IE[Zf(Z)].
\ee

For the second item of the Lemma, assume that $W$ is a random variable such that
$\IE [f'(W)-Wf(W)]=0$ for all bounded, continuous, and piecewise continuously differentiable functions $f$ with
$\IE\abs{f'(Z)}<\infty$.  The function $f_x$ satisfying \eqref{1a} is such a function, so that for all $x \in \IR$,
\ba
0=\IE [f_x'(W)-Wf_x(W)]=\IP(W\leq x) - \Phi(x),
\ee
which implies that $W$ has a standard normal distribution. 
\end{proof}

Our strategy for bounding the maximum distance between the distribution function of a random variable $W$ and that
of the standard normal is now fairly obvious: we want to bound
$\IE[f_x(W)-Wf_x(W)]$ for $f_x$ solving \eqref{1a}.  This setup can work, but it turns out that it is easier to work in
the Wasserstein metric.  Since the critical property of the Kolmogorov metric that we use in the discussion above is
the representation \eq{23}, which the Wasserstein metric shares, extending in this direction comes without great effort.

\subsection{The general setup}
For two random variables $X$ and $Y$ and some family of functions $\%H$, recall the metric
\ban
\dh(X,Y)=\sup_{h\in\%H}\abs{\IE h(X) - \IE h(Y)}, \label{32}
\ee  
and note that such a metric only depends on the law of $X$ and $Y$.  For $h\in\%H$, let $f_h$ solve
\ba
f'_h(w)-wf_h(w)=h(w) - \Phi(h)
\ee
where $\Phi(h)$ is the expectation of $h$ with respect to a standard normal distribution.  
We have the following result which easily follows from the discussion above.
\begin{proposition}\label{prop31}
If $W$ is a random variable and $Z$ has the standard normal distribution, then
\ban
\dh(W,Z)=\sup_{h\in\%H}\abs{\IE[f'_h(W)-Wf_h(W)]}. \label{31}
\ee
\end{proposition}

The main idea at this point is to bound the right side of \eqref{31} by using the structure of $W$ and properties of the solutions $f_h$.
The latter issue is handled by the following lemma.
\begin{lemma}\label{lem31}
Let $f_h$ be the solution of the differential equation
\ban
f'_h(w)-wf_h(w)=h(w) - \Phi(h) \label{33}
\ee
which
is given by 
\ba 
f_h(w)&=e^{w^2/2}\int_w^{\infty} e^{-t^2/2}\left(\Phi(h)-h(t)\right) dt \\
 &=-e^{w^2/2}\int_{-\infty}^w e^{-t^2/2}\left(\Phi(h)-h(t)\right) dt.
\ee
\begin{enumerate}
\item\label{i31} If $h$ is bounded, then 
\ba
\norm{f_h}\leq \sqrt{\frac{\pi}{2}}\norm{h(\cdot)-\Phi(h)}, \mbox{\,\, and \,\,}  \norm{f_h'}\leq 2\norm{h(\cdot)-\Phi(h)}.
\ee
\item\label{i32} If $h$ is absolutely continuous, then 
\ba
\norm{f_h}\leq 2\norm{h'}, \hspace{5mm} \norm{f_h'}\leq \sqrt{\frac{2}{\pi}}\norm{h'}, \mbox{\,\, and \,\,}  \norm{f_h''}\leq 2\norm{h'}.
\ee
\end{enumerate}
\end{lemma}
The proof of Lemma \ref{lem31} is similar to 
but more technical than that of Lemma \ref{lem2}.
We refer to \cite{cgs11} (Lemma~2.4) for the proof.

\section{Bounding the error}\label{norm}
We will focus mainly on the Wasserstein metric when approximating by continuous distributions.  This is not a terrible concession as firstly the
Wasserstein metric is a commonly used metric, and also 
by Proposition \ref{prop22}, 
for $Z$ a standard normal random variable and $W$ any random variable we have
\ba
\dk(W,Z)\leq(2/\pi)^{1/4}\sqrt{\dw(W,Z)},
\ee
where $\dk$ is the maximum difference between distribution functions (the Kolmogorov metric); $\dk$ is an intuitive and standard metric
to work with.  

The reason for using the Wasserstein metric is that it has the form \eqref{32} for $\%H$ the set of functions with Lipschitz constant equal to one.
In particular, if $h$ is a test function for the Wasserstein metric, then $\norm{h'}\leq1$ so that we know the solution $f_h$ of equation
\eqref{33} is bounded with two bounded derivatives by Item \ref{i32} of Proposition \ref{lem31}.  Contrast this to the set of test functions for the Kolmogorov metric where the solution $f_h$ of equation
\eqref{33} is bounded with
one bounded derivative (by Item \ref{i31} of Proposition \ref{lem31}) but is not twice differentiable.   

To summarize our progress to this point, we state the following result which is a corollary of Proposition \ref{prop31} and Lemma \ref{lem31}.  
The theorem is at the kernel of Stein's method.
\begin{theorem}\label{thm31}
If $W$ is a random variable and $Z$ has the standard normal distribution, and we define the family of functions $\%F=\{f: \norm{f},\norm{f''}\leq2, \norm{f'}\leq \sqrt{2/\pi}\}$, then
\ban
\dw(W,Z)\leq \sup_{f\in\%F}\abs{\IE[f'(W)-Wf(W)]}. \label{35}
\ee
\end{theorem}

In the remainder of this section, we discuss methods to bound $\abs{\IE[f'(W)-Wf(W)]}$ using the structure of $W$.
We will
identify general structures that are amenable to this task (for other structures in greater generality
see \cite{roe10}), but first we illustrate the type
of result we are looking for in the following standard example.

\subsection{Sum of independent random variables}

We will show the following result which follows from
Theorem \ref{thm31} and Lemma \ref{lem33aa} below.
\begin{theorem}\label{thm32}
Let $X_1, \ldots, X_n$ be independent random variables with $\IE \abs{X_i}^4 <\infty$, $\IE X_i =0$, and $\IE X_i^2 =1$. 
If $W=(\sum_{i=1}^n X_i)/\sqrt{n}$ and $Z$ has the standard normal distribution, then
\ba
\dw(W,Z)\leq \frac{1}{n^{3/2}}\sum_{i=1}^n \IE\abs{X_i}^3 + \frac{\sqrt{2}}{\sqrt{\pi}n}\sqrt{\sum_{i=1}^n\IE[X_i^4]}.
\ee
\end{theorem}
Before the proof we remark that if the $X_i$ of the theorem also have common distribution, then the rate of convergence
is order $n^{-1/2}$, which is the best possible.  It is also useful to compare this result to Theorem \ref{thm1}
which is in a different metric (neither result is recoverable from the other in full strength) and only assumes third moments.  
A small modification in the argument below yields a similar theorem assuming 
only third moments, but the structure of proof for the theorem as stated is one that we shall copy in the sequel.

In order to prepare for arguments to come, we will break the proof into a series of lemmas.  
Since our strategy is to apply Theorem \ref{thm31}
by estimating the right side of \eqref{35} for bounded $f$ with bounded first and second derivative, the first lemma shows 
an expansion of the right side of \eqref{35} using the structure of $W$ as defined in Theorem \ref{thm32}.
\begin{lemma}\label{lem33}
In the notation of Theorem \ref{thm32}, if $W_i=(\sum_{j\not=i} X_i)/\sqrt{n}$ then
\ban
\IE [Wf(W)]&=\IE\left[\frac{1}{\sqrt{n}}\sum_{i=1}^nX_i\left( f(W)- f(W_i)-(W-W_i)f'(W)\right)\right] \label{345}\\
	& \quad +\IE\left[\frac{1}{\sqrt{n}}\sum_{i=1}^nX_i(W-W_i)f'(W)\right]. \label{345a}
\ee
\end{lemma}
\begin{proof}
After noting that the negative of \eqref{345a} is contained in \eqref{345} and removing these terms from consideration,
the lemma is equivalent to
\ban
\IE [Wf(W)]=\IE\left[\frac{1}{\sqrt{n}}\sum_{i=1}^n\left(X_i f(W)- X_i f(W_i)\right)\right]. \label{346}
\ee
Equation \eqref{346} follows easily from the fact that $W_i$ is independent of $X_i$ so that
$\IE[X_i f(W_i)]=0$.  
\end{proof}
The proof of the theorem will follow after we show that \eqref{345} is small and that \eqref{345a} compares favorably to $f'(W)$; we will see
similar strategies frequently in the sequel.
\begin{lemma}\label{lem33aa}
If $f$ is a bounded function with bounded first and second derivative, then in the notation of Theorem \ref{thm32},
\ban
\abs{\IE[f'(W)-Wf(W)]}\leq  \frac{\norm{f''}}{2n^{3/2}}\sum_{i=1}^n \IE\abs{X_i}^3 + \frac{\norm{f'}}{n}\sqrt{\sum_{i=1}^n\IE[X_i^4]}. \label{337}
\ee
\end{lemma}
\begin{proof}
Using the notation and results of Lemma \ref{lem33}, we obtain
\ban
\abs{\IE[f'(W)-Wf(W)]}&\leq \left|\IE\left[\frac{1}{\sqrt{n}}\sum_{i=1}^nX_i\left( f(W)- f(W_i)-(W-W_i)f'(W)\right)\right]\right| \label{347}  \\
	& \quad +\left|\IE\left[f'(W)\left(1-\frac{1}{\sqrt{n}}\sum_{i=1}^nX_i(W-W_i)\right)\right]\right|. \label{347a}
\ee
By Taylor expansion, the triangle inequality, and after
pushing the absolute value inside the expectation, we obtain that \eqref{347} is bounded above by
\ba
\frac{\norm{f''}}{2\sqrt{n}}\sum_{i=1}^n\IE\left[\abs{X_i(W-W_i)^2}\right].
\ee
Since $(W-W_i)=X_i/\sqrt{n}$, we obtain the first term in the bound \eqref{337}.  We find that \eqref{347a} is bounded above by
\ba
\frac{\norm{f'}}{n}\IE\left|\sum_{i=1}^n(1-X_i^2)\right|\leq \frac{\norm{f'}}{n}\sqrt{\var\left(\sum_{i=1}^nX_i^2\right)},
\ee
where we have used the Cauchy-Schwarz inequality.  By independence and the fact that $\var(X_i^2)\leq\IE[X_i^4]$, we obtain
the second term in the bound \eqref{337}.
\end{proof}

We can see from the work above that the strategy to bound $\IE[f'(W)-Wf(W)]$ is to 
use the structure of $W$ to rewrite $\IE[Wf(W)]$ in a way that compares favorably to $\IE[f'(W)]$.  Rather than
attempt this program anew in each application that arises, we will develop out-the-door theorems that provide 
error terms for various canonical structures which arise in many applications.  

\subsection{Dependency Neighborhoods}
We now generalize Theorem \ref{thm32} to
sums of random variables with local dependence.
\begin{definition}
We say that a collection of random variables $(X_1, \ldots, X_n)$ has dependency neighborhoods $N_i\subseteq \{1, \ldots, n\}$,
$i=1, \ldots, n$,
if $X_i$ is independent of $\{X_j\}_{j\not\in N_i}$.
\end{definition}
If we think of constructing a graph with vertices $\{1, \ldots, n\}$ where if there is no edge between $i$ and $j$
then $X_i$ and $X_j$ are independent, then we can define $N_i/\{i\}$ as the neighbors of vertex $i$ in the graph.
For this reason, dependency neighborhoods are frequently referred to as dependency graphs.  Using the Stein's method
framework and a modification of the argument for sums of independent random variables we can prove the following theorem,
some version of which can be read from the main result of \cite{brs89}.
\begin{theorem}\label{thm41}
Let $X_1, \ldots, X_n$  be random variables with $\IE[X_i^4]<\infty$, $\IE[X_i]=0$, $\sigma^2=\var\left(\sum_i X_i\right)$, and
define $W=\sum_i X_i/\sigma$.  Let the collection $(X_1, \ldots, X_n)$ have dependency neighborhoods $N_i$, $i=1, \ldots, n$, with
$D:=\max_{1\leq i \leq n} \abs{N_i}$.  Then for $Z$ a standard normal random variable,
\ban
\dw(W,Z)\leq \frac{D^2}{\sigma^3}\sum_{i=1}^n\IE\abs{X_i}^3+\frac{\sqrt{26}D^{3/2}}{\sqrt{\pi}\sigma^2}\sqrt{\sum_{i=1}^n\IE[X_i^4]}. \label{41}
\ee
\end{theorem}
Note that this theorem quantifies the heuristic that a sum of many locally dependent random variables will be approximately normal.
When viewed as an asymptotic result, it's clear that under some conditions a CLT will hold even with $D$ growing with $n$.
It is also possible to prove similar theorems using further information about the dependence structure of the variables; see \cite{chsh04}.

The proof of the theorem will be analogous to the case of sums of independent random variables (a special case of this theorem), but
the analysis will be a little more complicated due to the dependence. 
\begin{proof}
From Theorem \ref{thm31}, 
to upper bound $\dw(W,Z)$ it is enough to bound $\abs{\IE [f'(W)-Wf(W)]}$, where
$\norm{f}, \norm{f''}\leq 2$ and $\norm{f'}\leq \sqrt{2/\pi}$.  Define $W_i=\sum_{j\not\in N_i} X_j$ and note that $X_i$ is independent
of $W_i$.  As in the proof of Theorem \ref{thm32}, 
we can now write
\ban
\abs{\IE[f'(W)-Wf(W)]}&\leq \left|\IE\left[\frac{1}{\sigma}\sum_{i=1}^nX_i\left( f(W)- f(W_i)-(W-W_i)f'(W)\right)\right]\right| \label{447}  \\
	& \quad +\left|\IE\left[f'(W)\left(1-\frac{1}{\sigma}\sum_{i=1}^nX_i(W-W_i)\right)\right]\right|. \label{447a}
\ee
We now proceed by showing that \eqref{447} is bounded above by the first term in \eqref{41} and 
\eqref{447a} is bounded above by the second.

By Taylor expansion, the triangle inequality, and after
pushing the absolute value inside the expectation, we obtain that \eqref{447} is bounded above by
\ban
\frac{\norm{f''}}{2\sigma}\sum_{i=1}^n\IE\left[\abs{X_i(W-W_i)^2}\right]
&\leq \frac{1}{\sigma^3}\sum_{i=1}^n\IE\left|X_i\left(\sum_{j\in N_i} X_j\right)^2 \right| \notag \\
&\leq \frac{1}{\sigma^3}\sum_{i=1}^n\sum_{j,k\in N_i}\IE\left|X_iX_jX_k\right|.  \label{448}
\ee
The arithmetic-geometric mean inequality implies that 
\ba
\IE\left|X_iX_jX_k\right|\leq \frac{1}{3}\left(\IE|X_i|^3+\IE|X_j|^3+\IE|X_k|^3\right),
\ee
so that \eqref{447} is bounded above by the first term in the bound \eqref{41}, where we use for example that
\ba
\sum_{i=1}^n\sum_{j,k\in N_i}\IE|X_j|^3\leq D^2 \sum_{j=1}^n\IE|X_j|^3.
\ee

Similar consideration implies that \eqref{447a} is bounded above by 
\ban
\frac{\norm{f'}}{\sigma^2}\IE\left|\sigma^2-\sum_{i=1}^nX_i\sum_{j\in N_i}X_j\right|\leq \frac{\sqrt{2}}{\sqrt{\pi}\sigma^2}\sqrt{\var\left(\sum_{i=1}^n\sum_{j\in N_i} X_iX_j\right)}. \label{49}
\ee 
where the inequality follows from the Cauchy-Schwarz inequality
coupled with the representation
\ba
\sigma^2=\IE\left[\sum_{i=1}^nX_i\sum_{j\in N_i}X_j\right].
\ee
The remainder of the proof consists of analysis on \eqref{49}, but note that in practice it may be possible to bound this term directly.
In order to bound the variance under the square root in \eqref{49}, we first compute
\ban
\IE\left[\left(\sum_{i=1}^n\sum_{j\in N_i} X_iX_j\right)^2\right]&=\sum_{i\not=j}\sum_{k\in N_i}\sum_{l \in N_j}\IE[X_iX_jX_kX_l]  \label{48c}\\
&+\sum_{i=1}^n\sum_{j\in N_i}\IE[X_i^2 X_j^2]+\sum_{i=1}^n\sum_{j\in N_i}\sum_{k\in N_i/\{j\}}\IE[X_i^2 X_j X_k]. \label{48a}
\ee
Using the arithmetic-geometric mean inequality, the first term of \eqref{48a} is bounded above by
\ba
\frac{1}{2}\sum_{i=1}^n\sum_{j\in N_i}\left(\IE[X_i^4]+\IE[X_j^4]\right)\leq D\sum_{i=1}^n\IE[X_i^4],
\ee
and the second by
\ba
\frac{1}{4}\sum_{i=1}^n\sum_{j\in N_i}\sum_{k\in N_i/\{j\}}\left(2\IE[X_i^4]+\IE[X_j^4]+\IE[X_k^4]\right)\leq D(D-1)\sum_{i=1}^n\IE[X_i^4].
\ee
We decompose the term \eqref{48c} into two components; 
\ban
\sum_{i\not=j}\sum_{k\in N_i}\sum_{l \in N_j}\IE[X_iX_jX_kX_l] =\sum_{\{i,k\}, \{j,l\} } \IE[X_iX_k]\IE[X_jX_l]+\sum_{\{i,k,j,l\}  }\IE[X_iX_jX_kX_l], \label{43}
\ee
where the first sum denotes the indices in which $\{X_i, X_k\}$ are independent of $\{X_j, X_l\}$, and the second term
consists of those remaining.  Note that by the arithmetic-geometric mean inequality, the second term of \eqref{43} is bounded above by
\ba
6D^3\sum_{i=1}^n\IE[X_i^4],
\ee
since the number of ``connected components" with at most four vertices of the dependency graph induced by the neighborhoods,
is no more than $D\times2D\times3D$.  The first term of \eqref{43} equals
\ba
\sigma^4-\sum_{\{i,k,j,l\}  }\IE[X_iX_k]\IE[X_jX_l],
\ee
and a couple applications of the arithmetic-geometric mean inequality yields
\ba
-\IE[X_iX_k]\IE[X_jX_l]&\leq \frac{1}{2}\left(\IE[X_iX_k]^2+\IE[X_jX_l]^2\right) \\
			& \leq  \frac{1}{2}\left(\IE[X_i^2X_k^2]+\IE[X_j^2X_l^2]\right)  \\
			&\leq \frac{1}{4}\left(\IE[X_i^4]+\IE[X_j^4]+\IE[X_k^4]+\IE[X_l^4]\right).
\ee
Putting everything together, we obtain that
\ba
\var\left(\sum_{i=1}^n\sum_{j\in N_i} X_iX_j\right)&=\IE\left[\left(\sum_{i=1}^n\sum_{j\in N_i} X_iX_j\right)^2\right]-\sigma^4 \\
	&\leq (12D^3+D^2)\sum_{i=1}^n\IE[X_i^4]\leq 13D^3\sum_{i=1}^n\IE[X_i^4],
\ee
which yields the theorem.
\end{proof}
Note that much of the proof of Theorem \ref{thm41} consists of bounding the error in a simple form.  However,
an upper bound for $\dw(W,Z)$ is obtained by adding the intermediate terms \eqref{448} and \eqref{49} which in many applications
may be directly bounded (and produce better bounds).

Theorem \ref{thm41} is an intuitively pleasing result that has many applications; a notable example is \cite{avbe93} where
CLTs for statistics of various random geometric graphs are shown.
We  apply it in the following setting.

\subsubsection{Application: Triangles in \ER\ random graphs} 
Let $G=G(n,p)$ be an \ER\ random graph on $n$ vertices with edge probability $p$ and let $T$ be the number of triangles in $G$.
We can write $T=\sum_{i=1}^N Y_i$, where $N=\binom{n}{3}$, and the $Y_i$ is the indicator that a triangle is formed at 
the ``$i$th" set of three vertices, in some arbitrary but fixed order. For $i\not=j$, $Y_i$ is independent of $Y_j$ if and only if the 
collection of edges between the vertices indexed by $i$ is disjoint from those indexed by $j$.  Thus we
let the set $N_i/\{i\}$ contain indices which share exactly two vertices with those indexed by $i$ so that $\abs{N_i}=3(n-3)+1$
and we can apply Theorem \ref{thm41} with $X_i=Y_i-p^3$
and $D=3n-8$.  Since
\ba
\IE\abs{X_i}^k=p^3(1-p^3)[(1-p^3)^{k-1}+p^{3(k-1)}], \hspace{5mm} k=1,2, \ldots
\ee  
we now only have to compute $\var(T)$ to apply the theorem.  A simple calculation using a decomposition of $T$ into indicators
shows that 
\ba
\sigma^2:=\var(T)=\binom{n}{3}p^3[1-p^3+3(n-3)p^2(1-p)],
\ee
and Theorem \ref{thm41} implies that for $W=(T-\IE[T])/\sigma$ and $Z$ a standard normal random variable
\ba
\dk(W, Z)&\leq
\frac{(3n-8)^2}{\sigma^3}\binom{n}{3}p^3(1-p^3)[(1-p^3)^{2}+p^{6}] \\
&\qquad+\frac{\sqrt{26}(3n-8)^{3/2}}{\sqrt{\pi}\sigma^2}\sqrt{\binom{n}{3}p^3(1-p^3)[(1-p^3)^{3}+p^{9}]}.
\ee

This bound holds for all $n\geq 3$ and $0\leq p \leq 1$,
but some asymptotic analysis shows that if, for example, $p\sim n^{-\alpha}$ for some $0\leq \alpha < 1$
(so that $\var(T)\rightarrow\infty$), then the number of triangles satisfies a CLT for $0\leq \alpha < 2/9$, which is
only a subset of the regime where normal convergence holds \cite{ruc88}.  It is possible that starting from \eqref{448} and \eqref{49} would
yield better rates in a wider regime, and considering finer structure yields better results \cite{bkr89}.

\subsection{Exchangeable pairs}
We begin with a definition.
\begin{definition}
The ordered pair $(W, W')$ of random variables is called an \emph{exchangeable pair} if $(W, W')\ed (W', W)$.  If for
some $0< a \leq 1$, the
exchangeable pair $(W, W')$ satisfies the relation
\ba
\IE[W'|W]=(1-a)W,
\ee
then we call $(W, W')$ an $a$-Stein pair.
\end{definition}
The next proposition contains some easy facts related to Stein pairs.
\begin{proposition}\label{prop51}
Let $(W, W')$ an exchangeable pair.
\begin{enumerate}
\item[1.] If $F:\IR^2\rightarrow \IR$ is an anti-symmetric function; that is $F(x,y)=-F(y,x)$, then $\IE[F(W,W')]=0$.
\end{enumerate}
If $(W, W')$ is an $a$-Stein pair with $\var(W)=\sigma^2$, then
\begin{enumerate}
\item[2.] $\IE[W]=0$ and $\IE[(W'-W)^2]=2a\sigma^2$.
\end{enumerate}
\end{proposition} 
\begin{proof}
Item 1 follows by the following equalities, the first by exchangeability and the second by anti-symmetry of $F$.
\ba
\IE[F(W, W')]=\IE[F(W', W)]=-\IE[F(W, W')].
\ee
The first assertion of Item 2 follows from the fact that $\IE[W]=\IE[W']=(1-a)\IE[W]$, and the second by calculating
\ba
\IE[(W'-W)^2]=\IE[(W')^2]+\IE[W^2]-2\IE[W\IE[W'|W]]=2\sigma^2-2(1-a)\sigma^2=2a\sigma^2.
\ee
\end{proof}
From this point we illustrate the use of the exchangeable pair in the following theorem.
\begin{theorem}\label{thm51}
If $(W, W')$ is an $a$-Stein pair with $\IE[W^2]=1$ and $Z$ has the standard normal distribution, then
\ba
\dw(W,Z)\leq\frac{\sqrt{\var\left(\IE[(W'-W)^2|W]\right)}}{\sqrt{2\pi}a}+\frac{\IE|W'-W|^3}{3a}. 
\ee
\end{theorem}
Before the proof comes a few remarks.
\begin{remark}
The strategy for using Theorem \ref{thm51} to obtain an error in the approximation of the distribution of a random variable 
$W$ by the standard normal is to construct $W'$ on the same space as $W$, such that $(W, W')$ is an $a$-Stein pair.
How can we achieve this construction?  Typically $W=W(\omega)$ is a random variable on some space $\Omega$ with
probability measure $\mu$.  It is not too difficult to see that if $X_0, X_1, \ldots$ is a Markov chain in stationary which
is reversible with respect to $\mu$, then setting (with some abusive notation) $W=W(X_0)$ and $W'=W(X_1)$
defines an exchangeable pair.  Since there is much effort put into constructing reversible Markov chains (e.g. Gibbs sampler),
this is a useful method to construct exchangeable pairs.  However, the linearity condition is not as easily abstractly
constructed and must be verified.
\end{remark}
\begin{remark}
In lieu of the previous remark, it is useful to note that 
\ba
\var\left(\IE[(W'-W)^2|W]\right)\leq \var\left(\IE[(W'-W)^2|\%F]\right),
\ee
for any sigma-field $\%F$ which is larger than the sigma-field generated by $W$.  With notation in the
previous remark, in many instances it is helpful to condition on $X_0$ rather than $W$ when computing the 
error bound from Theorem \ref{thm51}.
\end{remark}
\begin{remark}
A heuristic explanation for the form of the error terms appearing in Theorem \ref{thm51} arises by
considering an Ornstein-Uhlenbeck (O-U) diffusion process.  Define
the diffusion process $(D(t))_{t\geq0}$ by the following properties.
\begin{enumerate}
\item $\IE[D(t+a)-D(t)|D(t)=x]=-ax+\lito(a)$.
\item $\IE[(D(t+a)-D(t))^2|D(t)=x]=2a+\lito(a)$.
\item For all $\eps>0$, $\IP[\abs{D(t+a)-D(t)}>\eps|D(t)=x]=\lito(a)$.
\end{enumerate}
Here the function $g(a)$ is $\lito(a)$ if $g(a)/a$ tends to zero as $a$ tends to zero.  These three properties determine the O-U
diffusion process, and this process is reversible with the standard normal distribution as its stationary distribution.  
What does this have to do with Theorem \ref{51}? 
Roughly, if we think of $W$ as $D(t)$ and $W'$ as $D(t+a)$ for some small $a$, then Item 1 corresponds to the $a$-Stein pair
linearity condition, Item 2 implies that the first term of the error in Theorem \ref{thm51} will be small, and Item 3 relates to the 
second term in the error.  
\end{remark}

\begin{proof}[Proof of Theorem \ref{thm51}]
The strategy of the proof is to use the exchangeable pair to rewrite $\IE[Wf(W)]$ in such a way that compares
favorably to $\IE[f'(W)]$.  To this end, let $f$ be bounded with bounded first and second derivative and let $F(w):=\int_0^wf(t) dt$.
Now, exchangeability and Taylor expansion imply that
\ban
0&=\IE[F(W')-F(W)] \notag \\
	&=\IE\left[(W'-W)f(W)+\frac{1}{2}(W'-W)^2f'(W) + \frac{1}{6}(W'-W)^3f''(W^*)\right], \label{51}
\ee
where $W^*$ is a random quantity in the interval with endpoints $W$ and $W'$.  Now, the linearity condition on the Stein pair
yields
\ban
\IE\left[(W'-W)f(W)\right]=\IE[f(W)\IE[(W'-W)|W]]=-a\IE[Wf(W)]. \label{52}
\ee
Combining \eqref{51} and \eqref{52} we obtain
\ba
\IE[Wf(W)]=\IE\left[\frac{(W'-W)^2f'(W)}{2a} + \frac{(W'-W)^3f''(W^*)}{6a}\right].
\ee
From this point we can easily see 
\ban
\left|\IE[f'(W)-Wf(W)]\right|\leq \norm{f'}\IE\left|1-\frac{\IE[(W'-W)^2|W]}{2a}\right|+\norm{f''}\frac{\IE|W'-W|^3}{6a}, \label{511}
\ee
and the theorem will follow after noting that we are only considering $f$ with $\norm{f'}\leq\sqrt{2/\pi}$, and $\norm{f''}\leq 2$, and
that from Item 2 of Proposition \ref{prop51}, we have $\IE[\IE[(W'-W)^2|W]]=2a$ so that an application of the Cauchy-Schwarz
inequality yields the variance term in the bound. 
\end{proof}

Before moving to a heavier application, we consider the canonical example of a sum of independent random variables.
\begin{example}
Let $X_1, \ldots, X_n$ independent with $\IE[X_i^4]< \infty$, $\IE[X_i]=0$, $\var(X_i)=1$, and $W=n^{-1/2}\sum_{i=1}^n X_i$.
We construct our exchangeable pair by choosing an index uniformly at random and replacing it by an independent copy.  Formally,
let $I$ uniform on $\{1, \ldots, n\}$, $(X_1', \ldots, X_n')$ be an independent copy of $(X_1, \ldots, X_n)$, and define 
\ba
W'=W-\frac{X_I}{\sqrt{n}}+\frac{X_I'}{\sqrt{n}}.
\ee
It is a simple exercise to show that $(W, W')$ is exchangeable, and we now verify that is also a $1/n$-Stein pair.
The calculation below is straightforward; in the
penultimate equality we use the independence of $X_i$ and $X_i'$ and the fact that $\IE[X_i']=0$.
\ba
\IE[W'-W|(X_1, \ldots, X_n)]&=\frac{1}{\sqrt{n}}\IE[X_I'-X_I|(X_1, \ldots, X_n)] \\
	&=\frac{1}{\sqrt{n}}\sum_{i=1}^n\frac{1}{n}\IE[X_i'-X_i|(X_1, \ldots, X_n)] \\
	&=-\frac{1}{n}\sum_{i=1}^n\frac{X_i}{\sqrt{n}}=-\frac{W}{n}.
\ee
Since the conditioning on the larger sigma-field only depends on $W$, we have that
$\IE[W'-W|W]=-W/n$, as desired.

We can now apply Theorem \ref{thm51}.  We first bound
\ba
\IE\abs{W'-W}^3&=\frac{1}{n^{3/2}}\sum_{i=1}^n\IE\abs{X_i-X_i'}^3 \\
&\leq\frac{8}{n^{3/2}}\sum_{i=1}^n\IE\abs{X_i}^3,
\ee 
where we used the arithmetic-geometric mean inequality for the cross terms of the expansion of the cube of the difference
(we could also express the error in terms of these lower moments by independence).
Next we compute
\ba
\IE[(W'-W)^2|(X_1, \ldots, X_n)]&=\frac{1}{n^2}\sum_{i=1}^n \IE[(X_i'-X_i)^2|X_i] \\
	& =\frac{1}{n^2}\sum_{i=1}^n 1+X_i^2.
\ee
Taking the variance we see that 
\ba
\var\left(\IE[(W'-W)^2|W]\right)\leq \frac{1}{n^4}\sum_{i=1}^n\IE[X_i^4].
\ee
Combining the estimates above we have
\ba
\dw(W,Z)\leq \sqrt{\frac{2}{\pi}}\frac{\sqrt{\sum_{i=1}^n\IE[X_i^4]}}{2n}+\frac{2}{3n}\sum_{i=1}^n\IE\abs{X_i}^3.
\ee
Note that if the $X_i$ are i.i.d. then this term is of order $n^{-1/2}$, which is best possible.  
Finally, we could probably get away with only assuming three moments
for the $X_i$ if we use the intermediate term \eqref{511} in the proof of Theorem \ref{51}.
\end{example}

\subsubsection{Application: Anti-voter model}\label{secav}
In this section we consider an application of Theorem \ref{thm51} found in \cite{riro97}; we closely follow their treatment.  
Let $G$ be an $r$-regular\footnote{The term $r$-regular means that every vertex has degree $r$.}
graph with vertex set $V$ and edge set $E$.  Define
a Markov chain on the space $\{-1, 1\}^V$ of labelings of the vertices of $V$ by $+1$ and $-1$.  The chain follows the rule of uniformly 
choosing a vertex $v\in V$, then uniformly choosing a neighbor of $v$ and changing the sign of the label of $v$ to the opposite of
its neighbor.  The model gets its name from thinking of the vertices as people in a town full of curmudgeons where
a positive (negative)
labeling corresponding to a yes (no) vote for some measure.  At each time unit a random person talks to a random neighbor
and decides to switch votes to the opposite of that neighbor.  

It is known \cite{af10} that if the underlying graph $G$ is not bipartite or a cycle, then the anti-voter chain is irreducible and aperiodic
and has a unique stationary distribution.  This distribution can be difficult to describe, but we can use Theorem
\ref{thm51} to obtain an error in the Wasserstein distance to the standard normal distribution for the sum of the labels of 
the vertices.  We now state the theorem and postpone discussion of computing the relevant quantities in the error until after the proof.
\begin{theorem}\label{thm52}
Let $G$ be an $r$-regular graph with $n$ vertices which is not bipartite or a cycle.  Let $\textbf{X}=(X_i)_{i=1}^n\in \{-1,1\}^n$ have
the stationary distribution of the anti-voter chain and let $\textbf{X}'=(X_i')_{i=1}^n$ be one step in the chain.  Let
$\sigma_n^2=\var(\sum_i X_i)$, $W=\sigma_n^{-1}\sum_i X_i$, and $W'=\sigma_n^{-1}\sum_i X_i'$.  
Then $(W,W')$ is a $2/n$-Stein pair, and if $Z$ has the standard normal distribution, then
\ba
\dw(W,Z)\leq  \frac{4n}{3 \sigma_n^3} +\frac{\sqrt{\var(Q)}}{r\sigma_n^2\sqrt{2\pi}},
\ee  
where
\ba
Q=\sum_{i=1}^n\sum_{j\in N_i} X_i X_j,
\ee
and $N_i$ denotes the neighbors of $i$.
\end{theorem}

Part of the first assertion of the theorem is that $(W, W')$ is exchangeable, which is non-trivial to verify since the anti-voter
chain is not necessarily reversible.  However, we can apply the following lemma - the proof
here appears in \cite{roe06}.
\begin{lemma}\label{lem51}
If $W$ and $W'$ are identically distributed integer-valued random variables defined on the same space
such that $\IP(|W'-W|\leq 1)=1$, then $(W, W')$ is an exchangeable pair.
\end{lemma}
\begin{proof}
The fact that $W$ and $W'$ only differ by at most one almost surely imply
\ba
\IP(W'\leq k)=\IP(W<k)+\IP(W=k, W'\leq k)+\IP(W=k+1, W'=k),
\ee
while we also have
\ba
\IP(W\leq k)= \IP(W<k) + \IP(W=k, W'\leq k) + \IP(W=k, W'=k+1).
\ee
Since $W$ and $W'$ have the same distribution, the left hand sides of the equations above are equal, and equating the right hand sides
yields
\ba
\IP(W=k+1, W'=k)=\IP(W=k, W'=k+1),
\ee
which is the lemma.
\end{proof}

\begin{proof}[Proof of Theorem \ref{thm52}]  
For the proof below let $\sigma:=\sigma_n$ so that $\sigma W=\sum_{i=1}^n X_i$.
The exchangeability of $(W, W')$ follows by Lemma \ref{lem51} since $\IP(\sigma(W'-W)/2 \in \{-1, 0, 1\})=1$.  To show the linearity
condition for the Stein pair, we define some auxiliary quantities related to $\textbf{X}$.
Let $a_1=a_1(\textbf{X})$ be the number of edges in $G$ which have a one at each end vertex when labeled by $\textbf{X}$.  Similarly,
let $a_{-1}$ be the analogous quantity with negative ones at each end vertex and $a_0$ be the number of edges with a different labal
at each end vertex.  Due to the fact that $G$ is $r$-regular, the number of ones in $\textbf{X}$ is
$(2a_1+a_0)/r$ and the number of negative ones in $\textbf{X}$ is $(2a_{-1}+a_0)/r$.  Note that these two observation imply
\ban
\sigma W= \frac{2}{r}\left(a_1-a_{-1}\right). \label{55}
\ee
Now, since conditional on $\textbf{X}$ the event $\sigma W'=\sigma W +2$ is equal to the event that the chain
moves to $\textbf{X}'$ by choosing a vertex labeled $-1$ and then choosing a neighbor
with label $-1$, we have
\ban
\IP(\sigma(W'-W)=2|\textbf{X})=\frac{2a_{-1}}{nr} \label{56}
\ee
and similarly
\ban
\IP(\sigma(W'-W)=-2|\textbf{X})=\frac{2a_{1}}{nr}. \label{57}
\ee
Using these last two formulas and \eqref{55}, we obtain
\ba
\IE[\sigma(W'-W)|\textbf{X}]=\frac{2}{nr}\left(a_{-1}-a_{1}\right)=-\frac{2 \sigma W}{n},
\ee
as desired.

From this point we will compute the error terms from Theorem \ref{thm51}.  The first thing to note
is that $|W'-W|\leq 2/\sigma$ implies
\ba
\frac{\IE|W'-W|^3}{3a}\leq \frac{4n}{3\sigma^3},
\ee
which contributes the first part of the error from the Theorem.
Now, \eqref{56} and \eqref{57} imply
\ban
\IE[(W'-W)^2|\textbf{X}]=\frac{8}{\sigma^2n r}\left(a_{-1}+a_1\right), \label{58}
\ee
and since
\ba
&2a_1+2a_{-1}+2a_0=\sum_{i=1}^n\sum_{j\in N_i} 1 = nr, \\
&2a_1+2a_{-1}-2a_0=\sum_{i=1}^n\sum_{j\in N_i} X_i X_j = Q,
\ee
we have
\ban
Q=4(a_{-1}+a_1)-rn, \notag
\ee
which combining with \eqref{58} and a small calculation yields the second error term of the theorem.
\end{proof}

In order for Theorem \ref{thm52} to be useful for a given graph $G$, we need lower bounds on $\sigma_n^2$ and upper
bounds on $\var(Q)$.  The former item can be accomplished by the following result of \cite{af10} (Chapter 14).
\begin{lemma}\cite{af10}
Let $G$ be an $r$-regular graph and let $\kappa=\kappa(G)$ be the minimum over subsets of vertices $A$ of the quantity of edges that have both ends in $A$ or
both ends in $A^c$. If $\sigma^2$ is the variance of the stationary distribution of the anti-voter model on $G$, then
\ba
\frac{2\kappa}{r}\leq \sigma^2 \leq n.
\ee 
\end{lemma}

The strategy to upper bound $\var(Q)$ is to associate the anti-voter model to a so-called ``dual process" from interacting particle system
theory.  This discussion is outside the scope of our work, but see \cite{af10, dowe84, riro97}.

\subsection{Size-bias coupling}\label{sbc}
Our next method of rewriting $\IE[Wf(W)]$ to be compared to $\IE[f'(W)]$ is through the size-bias coupling which first appeared
in the context of Stein's method for normal approximation in \cite{gori96}.

\begin{definition}
For a random variable $X\geq0$ with $\IE[X]=\mu<\infty$, we say the random variable $X^s$ has the \emph{size-bias} distribution with respect to 
$X$ if for all $f$ such that $\IE|Xf(X)|<\infty$ we have
\ba
\IE[Xf(X)]=\mu\IE[f(X^s)].
\ee 
\end{definition}
Before discussing existence of the size-bias distribution, we remark that our use of $X^s$ is a bit more transparent than the exchangeable
pair.  To wit, if $\var(X)=\sigma^2<\infty$ and $W=(X-\mu)/\sigma$, then
\ban
\IE[Wf(W)]&=\IE\left[\frac{X-\mu}{\sigma}f\left(\frac{X-\mu}{\sigma}\right)\right] \notag \\
	&=\frac{\mu}{\sigma}\left[f\left(\frac{X^s-\mu}{\sigma}\right)-f\left(\frac{X-\mu}{\sigma}\right)\right], \label{71}
\ee
so that if $f$ is differentiable, then the Taylor expansion of \eqref{71} about $W$ allows us to compare $\IE[Wf(W)]$ to
$\IE[f'(W)]$.  We will make this precise shortly, but first we tie up a loose end.
\begin{proposition}
If $X\geq0$ is a random variable with $\IE[X]=\mu<\infty$ and distribution function $F$, then the size-bias distribution of $X$ 
is absolutely continuous with respect to the measure of $X$ with density read from
\ba
dF^s(x)=\frac{x}{\mu} dF(x).
\ee
\end{proposition}
\begin{corollary}
If $X\geq0$ is an integer-valued random variable  with $\IE[X]=\mu<\infty$ then the random variable $X^s$ with the size-bias distribution
of $X$ is such that
\ba
\IP(X^s=k)=\frac{k \IP(X=k)}{\mu}.
\ee
\end{corollary}
The size-bias distribution arises in other contexts such as the waiting time paradox and sampling theory \cite{argo11}.
We now record our main Stein's method size-bias normal approximation theorem.
\begin{theorem}\label{thm71}
Let $X\geq0$ be a random variable with $\IE[X]=\mu<\infty$ and $\var(X)=\sigma^2$.  Let $X^s$ be defined on the same space as $X$
and have the size-bias distribution with respect to $X$.  If $W=(X-\mu)/\sigma$ and $Z\sim N(0,1)$, then
\ba
\dw(W, Z) \leq\frac{\mu}{\sigma^2}\sqrt{\frac{2}{\pi}}\sqrt{\var(\IE[X^s-X|X])}+\frac{\mu}{\sigma^3}\IE[(X^s-X)^2].
\ee
\end{theorem}
\begin{proof}
Our strategy (as usual) is to bound $\abs{\IE[f'(W)-Wf(W)]}$ for $f$ bounded with two bounded derivatives.  Starting from \eqref{71},
a Taylor expansion yields
\ba
\IE[Wf(W)]=\frac{\mu}{\sigma}\IE\left[\frac{X^s-X}{\sigma} f'\left(\frac{X-\mu}{\sigma}\right)+\frac{(X^s-X)^2}{2\sigma^2}
f''\left(\frac{X^*-\mu}{\sigma}\right)\right],
\ee
for some $X^*$ in the interval with endpoints $X$ and $X^s$.  Using the definition of $W$ in terms of $X$ in the previous expression, we
obtain
\ban
\abs{\IE[f'(W)-Wf(W)]}&\leq \left|\IE\left[f'(W)\left(1-\frac{\mu}{\sigma^2}(X^s-X)\right)\right]\right| \label{72}\\
	&+\frac{\mu}{2\sigma^3} \left|\IE\left[f''\left(\frac{X^*-\mu}{\sigma}\right)(X^s-X)^2\right]\right|. \label{73}
\ee
Since we are taking the supremum over functions $f$ with $\norm{f'}\leq \sqrt{2/\pi}$ and $\norm{f''}\leq 2$, 
it is clear that \eqref{73} 
is bounded above by the second term of the error stated in the theorem and \eqref{72} is bounded above by 
\ba
\sqrt{\frac{2}{\pi}}\IE\left|1-\frac{\mu}{\sigma^2}\IE[X^s-X|X]\right|\leq \frac{\mu}{\sigma^2}\sqrt{\frac{2}{\pi}}\sqrt{\var(\IE[X^s-X|X])};
\ee
here we use the Cauchy-Schwarz inequality after noting that by the definition of $X^s$, $\IE[X^s]=(\sigma^2 +\mu^2)/\mu$.
\end{proof}

\subsubsection{Coupling construction}
At this point it is appropriate to discuss methods to couple a random variable $X$ to a size-bias version $X^s$.
In the case that $X=\sum_{i=1}^n X_i$, where $X_i\geq0$ and $\IE[X_i]=\mu_i$,
we have the following recipe for constructing a size-bias version of $X$.
\begin{enumerate}
\item For each $i=1,\ldots, n$, let $X_i^{s}$ have the size-bias distribution of $X_i$ independent of $(X_j)_{j\not=i}$
and $(X_j^s)_{j\not=i}$.  Given $X_i^s=x$, define the vector
$(X_j^{(i)})_{j\not=i}$ to have the distribution of $(X_j)_{j\not=i}$ conditional on $X_i=x$.
\item Choose a random summand $X_I$, where the index $I$ is chosen proportional to $\mu_i$ and independent of all else. 
Specifically, $\IP(I=i)=\mu_i/\mu$, where $\mu=\IE[X]$.
\item Define $X^s=\sum_{j\not=I}X_j^{(I)} + X_I^s$. 
\end{enumerate}  
\begin{proposition}\label{prop71}
Let $X=\sum_{i=1}^n X_i$, with $X_i\geq0$, $\IE[X_i]=\mu_i$, and $\mu=\IE[X]=\sum_i \mu_i$.  If $X^s$ is constructed 
by Items 1 - 3 above, then $X^s$ has the size-bias distribution of $X$.
\end{proposition}
\begin{proof}
Let $\textbf{X}=(X_1, \ldots, X_n)$ and for $i=1, \ldots, n$, let $\textbf{X}^i$ be a vector with coordinate $j$ equal to $X_j^{(i)}$ for $j\not=i$
and coordinate $i$ equal to $X_i^s$ as in item 1 above.  In order to prove the result, it is enough to show
\ban
\IE[Wf(\textbf{X})]=\mu\IE[f(\textbf{X}^I)], \label{77aa}
\ee
for $f:\IR^n\rightarrow \IR$ such that $\IE|Wf(\textbf{X})|< \infty$.  Equation \eq{77aa} follows easily after we show that for all $i=1, \ldots, n$,
\ban
\IE[X_i f(\textbf{X})]=\mu_i\IE[f(\textbf{X}^i)]. \label{75}
\ee
To see \eqref{75}, note that for $h(X_i)=\IE[f(\textbf{X})|X_i]$,
\ba
\IE[X_i f(\textbf{X})]&=\IE[X_i h(X_i)] \\
	&=\mu_i\IE[h(X_i^s)],
\ee
which is the right hand side of \eqref{75}.
\end{proof}

Note the following special cases of Proposition \ref{prop71}.
\begin{corollary}\label{cor71}
Let $X_1, \ldots, X_n$ be non-negative independent random variables with $\IE[X_i]=\mu_i$,
and for each $i=1,\ldots, n$, let $X_i^{s}$ have the size-bias distribution of $X_i$ independent of $(X_j)_{j\not=i}$
and $(X_j^s)_{j\not=i}$.  If $X=\sum_{i=1}^n X_i$, $\mu=\IE[X]$, and $I$
is chosen independent of all else with $\IP(I=i)=\mu_i/\mu$, then
$X^s=X-X_I+X_I^s$ has the size-bias distribution of $X$.
\end{corollary}
\begin{corollary}\label{cor72}
Let $X_1, \ldots, X_n$ be zero-one random variables with $\IP(X_i=1)=p_i$. 
For each $i=1,\ldots, n$, let $(X_j^{(i)})_{j\not=i}$ have the distribution of $(X_j)_{j\not=i}$ conditional on $X_i=1$.
If $X=\sum_{i=1}^n X_i$, $\mu=\IE[X]$,
and $I$
is chosen independent of all else with $\IP(I=i)=p_i/\mu$, then $X^s=\sum_{j\not=I} X_j^{(I)} +1$ has the size-bias distribution of $X$.
\end{corollary}
\begin{proof}
Corollary \ref{71} is obvious since due to independence, the conditioning in the construction has no effect.
Corollary \ref{cor72} follows
after noting that for $X_i$ a zero-one random variable, $X_i^s=1$.
\end{proof}

\subsubsection{Applications}
\begin{example}\label{ex71}
We can use Corollary \ref{cor71} in Theorem \ref{thm71}
to bound the Wasserstein distance between the normalized sum of independent variables with finite third moment
and the normal distribution - we leave this as an exercise. 
\end{example}
\begin{example}\label{ex72}
Let $G=G(n,p)$ be an \ER\ graph and for $i=1, \ldots, n$, let $X_i$ be the indicator that vertex $v_i$ (under some arbitrary but fixed labeling)
has degree zero so that $X=\sum_{i=1}^n X_i$ is the number of isolated vertices of $G$.  We will use Theorem \ref{thm71} to
obtain an upper bound on the Wasserstein metric between the normal distribution and the distribution of $W=(X-\mu)/\sigma$ where $\mu=\IE[X]$
and $\sigma^2=\var(X)$.

Since $X$ is a sum of identically distributed indicators, we can use Corollary \ref{cor72} to construct $X^s$, a size-bias
version of $X$.  Corollary \ref{cor72} states that in order to size-bias $X$, we first choose
an index $I$ uniformly at random from the set $\{1, \ldots, n\}$, then size-bias $X_I$ by setting it equal to one, and finally adjust the remaining
summands conditional on $X_I=1$ (the new size-bias value).  We can realize
$X_I^s=1$ by erasing any edges connected to vertex $v_I$.  Given that
$X_I=1$ ($v_I$ is isolated), the graph $G$ is just an \ER\ graph on the remaining $n-1$ vertices.  
Thus $X^s$ can be realized as the number of isolated vertices in $G$ after erasing all the edges connected to 
$v_I$.

In order to apply Theorem \ref{thm71} using this construction, we need to compute $\IE[X]$, $\var(X)$, $\var(\IE[X^s-X|X])$, and $\IE[(X^s-X)^2]$.
Since the chance that a given vertex is isolated is $(1-p)^{n-1}$, we have 
\ba
\mu:=\IE[X]=n(1-p)^{n-1},
\ee
and also that
\ban
\sigma^2:=\var(X)&=\mu\left(1-(1-p)^{n-1}\right)+n(n-1)\cov(X_1, X_2) \notag \\
	&=\mu[1+(np-1)(1-p)^{n-2}], \label{77}
\ee
since $\IE[X_1X_2]=(1-p)^{2n-3}$.  Let $d_i$ be the degree of $v_i$ in $G$ and let 
$D_i$ be the number of vertices connected to $v_i$ which have degree one.  Then it is clear that
\ba
X^s-X=D_I+\I[d_I>0],
\ee
so that
\ban
\var(\IE[X^s-X|G])&=\frac{1}{n^2}\var\left(\sum_{i=1}^n(D_i+\I[d_i>0])\right) \\
	&\leq \frac{2}{n^2}\left[\var\left(\sum_{i=1}^nD_i\right)+\var\left(\sum_{i=1}^n\I[d_i>0]\right)\right]. \label{78}
\ee
Since $\sum_{i=1}^n\I[d_i>0]=n-X$, the second variance term of \eqref{78} is given by \eqref{77}.
Now, $\sum_{i=1}^nD_i$ is the number of vertices in $G$ with degree one
which can be expressed as $\sum_{i=1}^nY_i$, where $Y_i$ is the indicator that $v_i$
has degree one in $G$.  Thus,
\ba
&\var\left(\sum_{i=1}^nD_i\right)=n(n-1)p(1-p)^{n-2}\left(1-(n-1)p(1-p)^{n-2}\right)+n(n-1)\cov(Y_1, Y_2) \\
	&=n(n-1)p(1-p)^{n-2}\left[1-(n-1)p(1-p)^{n-2}+(1-p)^{n-2}+(n-1)^2p^2(1-p)^{n-3}\right],
\ee 
since $\IE[Y_1 Y_2]=p(1-p)^{2n-4}+(n-1)^2p^2(1-p)^{2n-5}$ (the first term corresponds to $v_1$ and $v_2$ being joined).

The final term we need to bound is
\ba
\IE[(X^s-X)^2]&=\IE\left[\IE[(X^s-X)^2|X]\right] \\
	&=\frac{1}{n}\sum_{i=1}^n\IE[(D_i+\I[d_i>0])^2] \\
	&\leq \frac{1}{n}\sum_{i=1}^n\IE[(D_i+1)^2] \\
	&=\IE[D_1^2]+2\IE[D_1]+1.
\ee
Expressing $D_1$ as a sum of indicators, it is not difficult to show
\ba
\IE[D_1^2]=(n-1)p(1-p)^{n-2}+(n-1)(n-2)p^2(1-p)^{2n-5},
\ee
and after noting that $D_1\leq D_1^2$ almost surely, we can combine the estimates above with
Theorem \ref{thm71} to obtain an explicit upper bound between the distribution of $W$ and the standard normal in the Wasserstein metric. 
In particular, we can read the following result from our work above.
\begin{theorem}
If $X$ is the number of isolated vertices in an \ER\ graph $G(n,p)$, $W=(X-\mu)/\sigma$, and
for some $1\leq \alpha <2$ we have $\lim_{n\rightarrow\infty}n^\alpha p=c\in(0, \infty)$, then
\ba
\dw(W,Z)\leq \frac{C}{\sigma},
\ee
for some constant $C$.
\end{theorem}
\begin{proof}
The asymptotic hypothesis $\lim_{n\rightarrow\infty}n^\alpha p=c\in(0, \infty)$ for some $1\leq \alpha <2$
implies that $(1-p)^n$ tends to a finite positive constant.  Thus we can see that
$\mu\asymp n$, $\sigma^2\asymp n^{2-\alpha}$,  $\var(\IE[X^s-X|X])\asymp \sigma^2/n^2$, and $\IE[(X^s-X)^2]\asymp n^{1-\alpha}$,
from which the result follows from Theorem \ref{thm71}.
\end{proof}
\end{example}

Example \ref{ex71} can be generalized to counts of vertices of a given degree $d$ at some computational expense
\cite{gol10, gori96}; related results
pertain to the number of subgraphs counts in an \ER\ graph (such as the number of triangles) \cite{gol10}. 
We will examine such constructions 
in greater detail in our treatment of Stein's method for Poisson approximation where the size-bias coupling will play a large role.

\subsection{Zero-bias coupling}
Our next method of rewriting $\IE[Wf(W)]$ to be compared to $\IE[f'(W)]$ is through the zero-bias coupling first introduced
in \cite{gore97}.

\begin{definition}
For a random variable $W$ with $\IE[W]=0$ and $\var(W)=\sigma^2<\infty$,
we say the random variable $W^z$ has the \emph{zero-bias} distribution with respect to 
$W$ if for all absolutely continuous $f$ such that $\IE|Wf(W)|<\infty$ we have
\ba
\IE[Wf(W)]=\sigma^2\IE[f'(W^z)].
\ee 
\end{definition}
Before discussing existence and properties of the zero-bias distribution, we note that it is appropriate to
view the zero-biasing as a distributional transform which has the normal distribution as its unique fixed point.
Also note that zero-biasing is our most transparent effort to compare $\IE[Wf(W)]$ to $\IE[f'(W)]$, culminating
in the following result.
\begin{theorem}\label{thm81}
Let $W$ be a mean zero, variance one random variable 
and let $W^z$ be defined on the same space as $W$
and have the zero-bias distribution with respect to $W$.  If  $Z\sim N(0,1)$, then
\ba
\dw(W, Z) \leq 2\IE|W^z-W|.
\ee
\end{theorem}
\begin{proof}
Let $\%F$ be the set of functions such that$\norm{f'}\leq \sqrt{2/\pi}$ and $\norm{f}, \norm{f''}\leq 2$. Then
\ba
\dw(W,Z)&\leq\sup_{f\in\%F}\left|\IE[f'(W)-Wf(W)]\right| \\
	&=\sup_{f\in\%F}\left|\IE[f'(W)-f'(W^z)]\right|  \\
	&\leq \sup_{f\in\%F}\norm{f''}\IE\left|W-W^z\right|. 
\ee
\end{proof}
Before proceeding further, we discuss some fundamental properties of the zero-bias distribution.
\begin{proposition}
Let $W$ be a random variable with $\IE[W]=0$ and $\var(W)=\sigma^2<\infty$.  
\begin{enumerate}
\item\label{i81} There is a unique probability distribution for a random variable $W^z$ satisfying 
\ban
\IE[Wf(W)]=\sigma^2\IE[f'(W^z)] \label{81}
\ee 
for all absolutely continuous $f$ such that $\IE|Wf(W)|<\infty$.
\item\label{i82} The distribution of $W^z$ as defined by \eqref{81} is absolutely continuous with respect to Lebesgue measure
with density
\ban
p^z(w)=\sigma^{-2}\IE\left[W\I[W>w]\right]=-\sigma^{-2}\IE\left[W\I[W\leq w]\right]. \label{82}
\ee
\end{enumerate}
\end{proposition}
\begin{proof}
Assume that $\sigma^2=1$; the proof for general $\sigma$ is similar.
We will show Items \ref{i81} and \ref{i82} simultaneously by showing that $p^z$ defined by \eqref{82} is a probability density 
which defines a distribution satisfying \eqref{81}.

Let $f(x)=\int_0^x g(t) dt$ for a non-negative function $g$ integrable on compact domains.  Then
\ba
\int_0^\infty f'(u)\IE[W\I[W>u]] du&= \int_0^{\infty} g(u)\IE[W\I[W>u]] du \\
	&=\IE[W\int_0^{\max\{0, W\}}g(u) du = \IE[Wf(W)\I[W\geq0].
\ee
and similarly $\int_{-\infty}^0 f'(u)p^z(u) du =\IE[Wf(W)\I[W\leq0]$,
which implies that 
\ban
\int_\IR f'(u) p^z(u)du = \IE[Wf(W)] \label{83}
\ee
for all $f$ as above.  However, \eqref{83} extends to all absolutely continuous $f$ such that $\IE|Wf(W)|< \infty$ by routine
analytic considerations (e.g. considering the positive and negative part of $f$).  

We now show that $p^z$
is a probability density.  That $p^z$ is non-negative follows by considering the two representations in \eqref{82} - note that these representations
are equal since $\IE[W]=0$.   We also have
\ba
\int_0^\infty p^z(u) du= \IE[W^2 \I[W>0] \mbox{\, and \,} \int_{-\infty}^0 p^z(u) du= \IE[W^2 \I[W<0], 
\ee
so that $\int_\IR p^z(u) du = \IE[W^2]=1$.

Finally, uniqueness follows since for random variables $X$ and $Y$ such that
$\IE[f'(X)]=\IE[f'(Y)]$ for all continuously differentiable $f$ with compact support (say), then $X\ed Y$.
\end{proof}

The next result shows that little generality is lost in only considering 
$W$ with $\var(W)=1$ as we have done in Theorem \ref{thm81}.  The result can be read from
the density formula above or by a direct computation.
\begin{proposition}
If $W$ has mean zero and finite variance then $(aW)^z\ed~aW^z$.
\end{proposition}

\subsubsection{Coupling construction}
How do we construct a zero-bias coupling for a random variable?  In general this can be difficult, but we now discuss
the nicest case of a sum of independent random variables and work out a neat theoretical application using the 
construction.  Another canonical method of construction that is useful in practice can be derived from a Stein pair
- see \cite{gore97}. 

Let $X_1, \ldots, X_n$ be independent random variables with $\IE[X_i]=0$, $\var(X_i)=\sigma_i^2$, $\sum_{i=1}^n\sigma_i^2=1$,
and define $W=\sum_{i=1}^n X_i$.  We have the following recipe for constructing a zero-bias version of $W$.
\begin{enumerate}
\item For each $i=1,\ldots, n$, let $X_i^{z}$ have the zero-bias distribution of $X_i$ independent of $(X_j)_{j\not=i}$ and
$(X_j^z)_{j\not=i}$.
\item Choose a random summand $X_I$, where the index $I$ satisfies $\IP(I=i)=\sigma_i^2$ and is independent of all else.
\item Define $W^z=\sum_{j\not=I}X_j + X_I^z$. 
\end{enumerate}  
\begin{proposition}\label{prop81}
Let $W=\sum_{i=1}^n X_i$ be defined as above.  If $W^z$ is constructed 
as per Items 1 - 3 above, then $W^z$ has the zero-bias distribution of $W$.
\end{proposition}
\begin{proof}
We must show that $\IE[Wf(W)]=\IE[f'(W^z)]$ for all appropriate $f$.  Using the definition of zero-biasing
in the coordinate $X_i$ and the fact that $W-X_i$ is independent of $X_i$, we have
\ba
\IE[Wf(W)]&=\sum_{i=1}^n X_i f(W-X_i +X_i) \\
	&=\sum_{i=1}^n \sigma_i^2 f(W-X_i +X_i^z) \\
	&=\IE[f'(W-X_I+X_I^z)].
\ee
Since $\sum_{j\not=I}X_j + X_I^z=W-X_I+X_I^z$, the proof is complete.
\end{proof}

\subsubsection{Lindeberg-Feller condition}
We now discuss the way in which zero-biasing appears naturally in the proof of the Lindeberg-Feller CLT.  Our treatment closely
follows \cite{gol09a}.  

Let $(X_{i,n})_{1\leq n,  1\leq i \leq n}$ be a triangular array of random variables\footnote{That is, for each $n$, 
$(X_{i,n})_{1\leq i \leq n}$ is a collection of independent random variables.} such that $\var(X_{i,n})=\sigma_{i,n}^2<\infty$.  Let
$W_n=\sum_{i=1}^n X_{i,n}$, and assume that $\var(W_n)=1$.  A sufficient condition for $W_n$ to satisfy a CLT
as $n\rightarrow\infty$ is the Lindeberg condition: for all $\eps>0$,
\ban
\sum_{i=1}^n \IE[X_{i, n}^2\I[\abs{X_{i,n}}>\eps]]\rightarrow 0, \mbox{ as } n\rightarrow \infty. \label{85}
\ee
The condition ensures that no single term dominates in the sum so that the limit is not altered by the distribution of a summand.  Note that
the condition is not sufficient as we could take $X_{1,n}$ to be standard normal and the rest of the terms zero.  
We now have the following result.
\begin{theorem}\label{thm82}
Let $(X_{i,n})_{1 \leq n, 1\leq i \leq n}$ be the triangular array defined above and let $I_n$ be a random
variable independent of the $X_{i,n}$ and such that $\IP(I_n=i)=\sigma_{i, n}^2$.  For each $1\leq i \leq n$,  let
$X_{i,n}^z$ have the zero-bias distribution of $X_{i,n}$ independent of all else.  Then the Lindeberg condition \eqref{85} holds if and only if
\ban
X_{I_n, n}^z \stackrel{p}{\rightarrow} 0 \mbox{ as } n \to \infty. \label{86}
\ee
\end{theorem}
From this point, we can use a modification of Theorem \ref{thm81} to prove the following result which also follows
from Theorem \ref{thm82} and the classical Lindeberg-Feller CLT mentioned above.
\begin{theorem}\label{thm83}
In the notation of Theorem \ref{thm82} and the remarks directly preceding it, if $X_{I_n, n}^z \to 0$ in probability as $n \to \infty$,
then $W_n$ satisfies a CLT.
\end{theorem}
Before proving these two results, we note that Theorem \ref{thm83} is heuristically explained by Theorem \ref{thm81} and the zero-bias
construction of $W_n$.   Specifically, $|W_n^z-W_n|=\abs{X_{I_n,n}^z-X_{I_n, n}}$ and Theorem \ref{thm81} implies that $W_n$
is approximately normal if this latter quantity is small (in expectation). The proof of Theorem \ref{thm83} uses a modification
of the error in Theorem \ref{thm81} and the (non-trivial) fact that  
$X_{I_n, n}^z \to 0$ in probability implies that $X_{I_n,n} \to 0$ in probability.  Finally, the quantity $\abs{X_{i,n}^z-X_{i, n}}$
will also be small if $X_{i, n}$ is approximately normal, which indicates that the zero-bias approach will
show convergence in the CLT for the sum of independent random variables when such a result holds.  

\begin{proof}[Proof of Theorem \ref{thm82}]
We first perform a preliminary calculation to relate the Lindeberg-Feller condition to the zero-bias quantity of interest.
For some fixed $\eps>0$, let $f'(x)=\I[\abs{x}\geq\eps]$ and $f(0)=0$.
Using that $xf(x)=(x^2-\eps\abs{x})\I[\abs{x}\geq\eps]$ and the definition of the zero-bias transform, we find
\ba
\IP(|X_{I_n, n}^z|\geq \eps)&=\sum_{i=1}^n\sigma_{i,n}^2\IP(\abs{X_{i,n}^z}\geq \eps) \\
      &=\sum_{i=1}^n\sigma_{i,n}^2\IE[f'(X_{i,n}^z)] \\
	&=\sum_{i=1}^n\IE\left[(X_{i,n}^2-\eps\abs{X_{i,n}})\I[\abs{X_{i,n}}\geq\eps]\right].
\ee
From this point we note
\ba
\frac{x^2}{2}\I[\abs{x}\geq2\eps]\leq (x^2-\eps \abs{x})\I[\abs{x}\geq\eps]\leq x^2\I[\abs{x}\geq\eps]
\ee
which implies that for all $\eps>0$,
\ba
\frac{1}{2}\sum_{i=1}^n\IE\left[X_{i,n}^2\I[\abs{X_{i,n}}\geq2\eps]\right]\leq \IP(|X_{I_n, n}^z|\geq \eps) \leq \sum_{i=1}^n\IE\left[X_{i,n}^2\I[\abs{X_{i,n}}\geq\eps]\right],
\ee
so that \eq{85} and \eq{86} are equivalent.
\end{proof}
\begin{proof}[Proof of Theorem \ref{thm83}]
According to the proof of Theorem \ref{thm81}, it is enough to show that 
\ban
\abs{\IE[f'(W_n)-f'(W_n^z)]}\to 0 \mbox{ as } n \to \infty \label{89a}
\ee
for all bounded $f$ with two bounded derivatives.  We will show that $|W_n^z-W_n|\to 0$ in probability which implies
\eq{89a} by the following calculation.
\ba
\abs{\IE[f'(W_n)-f'(W_n^z)]}&\leq \IE|f'(W_n)-f'(W_n^z)| \\
	&=\int_0^\infty \IP(|f'(W_n)-f'(W_n^z)|\geq t) dt \\
	&= \int_0^{2\norm{f'}} \IP(|f'(W_n)-f'(W_n^z)|\geq t) dt \\
	&\leq \int_0^{2\norm{f'}} \IP(\norm{f''}|W_n-W_n^z|\geq t) dt \\
	&\leq \int_0^{2\norm{f'}} \IP(|W_n-W_n^z|\geq t/\norm{f''}) dt, 
\ee
which tends to zero by dominated convergence.  

We now must show that $|W_n^z-W_n|\to 0$ in probability.  Since we are assuming that $X_{I_n, n}^z \to 0$ in probability, 
and $|W_n^z-W_n|=|X_{I_n, n}^z-X_{I_n, n}|$,  it is enough to show that $X_{I_n, n}\to 0$ in probability.  
For $\eps>0$, and $m_n:=\max_{1\leq i \leq n}\sigma_{i,n}^2$, 
\ba
\IP(\abs{X_{I_n, n}}\geq \eps)&\leq \frac{\var(X_{I_n, n})}{\eps^2} \\
	&=\frac{1}{\eps^2}\sum_{i=1}^n \sigma^4_{i,n} \\
	&\leq\frac{m_n}{\eps^2}\sum_{i=1}^n \sigma^2_{i,n}=\frac{m_n}{\eps^2}.
\ee
From this point we show $m_n\to 0$, which will complete the proof.  For any $\delta>0$, we have
\ban
\sigma_{i,n}^2&=\IE[X_{i,n}^2\I[\abs{X_{i,n}}\leq \delta]]+\IE[X_{i,n}^2\I[\abs{X_{i,n}}> \delta]] \notag \\
	&\leq \delta^2+ \IE[X_{i,n}^2\I[\abs{X_{i,n}}> \delta]]. \label{89}
\ee
Using the calculations in the proof of Theorem \ref{thm82} based on the 
assumption that $X_{I_n, n}^z \to 0$ in probability, it follows that
\ba
\sum_{i=1}^n \IE[X_{i,n}^2\I[\abs{X_{i,n}}> \delta]] \to 0 \mbox{ as } n\to \infty,
\ee 
so that the second term of \eq{89} goes to zero as $n$ goes to infinity uniformly in $i$.
Thus we have that $\limsup_{n}m_n\leq \delta^2$ for all $\delta>0$ which implies that
$m_n\to 0$ since $m_n>0$.
\end{proof}

\subsection{Normal approximation in the Kolmogorov metric}
Our previous work has been to develop bounds on the Wasserstein metric between a distribution of interest and the normal distribution.  
For $W$ a random variable and $Z$ standard normal, we have the inequality
\ba
\dk(W,Z)\leq (2/\pi)^{1/4} \sqrt{\dw(W,Z)},
\ee
so that our previous effort implies bounds for the Kolmogorov metric.  However, it is often the case that this
inequality is suboptimal - for example if $W$ is a standardized binomial random variable with parameters $n$ and $p$,
then both $\dk(W,Z)$ and $\dw(W, Z)$ are of order $n^{-1/2}$.  In this section we 
develop Stein's method for normal approximation in the Kolmogorov metric in hopes of reconciling this 
discrepancy.\footnote{Of course improved rates will come at the cost of additional hypotheses, but we will see that 
the theorems are still useful in application.}  We follow \cite{cgs11} in our exposition below but
similar results using related methods appear elsewhere \cite{rope10, riro97, shsu06}.

Recall the following restatement of Corollary \ref{cor1}. 
\begin{theorem}\label{thm91}
Let $\Phi$ denote the standard normal distribution function and let $f_x(w)$ be the unique bounded solution of
\ban
f_x'(w)-wf_x(w)=\I[w\leq x]-\Phi(x). \label{91}
\ee
If $W$ is a random variable with finite mean and $Z$ is standard normal, then
\ba
\dk(W,Z)=\sup_{x\in\IR}\abs{\IE[f_x'(W)-Wf_x(W)]}.
\ee
\end{theorem}
Moreover, we have the following lemma, which can be read from \cite{cgs11}, Lemma 2.3.
\begin{lemma}\label{lem91}
If $f_x$ is the unique bounded solution to \eq{91}, then
\ba
\norm{f_x}\leq \sqrt{\frac{\pi}{2}}, \mbox{ \hspace{1in} } \norm{f_x'}\leq 2, 
\ee
and for all $u, v, w \in \IR$,
\ba
\abs{(w+u)f_x(w+u)-(w+v)f_x(w+v)}\leq (\abs{w}+\sqrt{2\pi}/4)(\abs{u}+\abs{v})
\ee
\end{lemma}
Our program can be summed up in the following corollary to the results above.
\begin{corollary}\label{cor91}
If $\%F$ is the set of functions satisfying the bounds of Lemma \ref{lem91} and
$W$ is a random variable with finite mean and $Z$ is standard normal, then
\ba
\dk(W,Z)\leq\sup_{f\in\%F}\abs{\IE[f'(W)-Wf(W)]}.
\ee
\end{corollary} 
\subsubsection{Zero-bias transformation}
To get better out-the-door rates using the zero-bias transform, we must assume a boundedness condition.
\begin{theorem}
Let $W$ be a mean zero, variance one random variable and suppose there is $W^z$ having the zero-bias
distribution of $W$ on the same space as $W$ such that $\abs{W^z-W}\leq \delta$ almost surely.  If
$Z$ is standard normal, then
\ba
\dk(W,Z)\leq \left(1+\frac{1}{\sqrt{2\pi}}+\frac{\sqrt{2\pi}}{4}\right)\delta.
\ee
\end{theorem} 
\begin{proof}
Our strategy of proof is to show that the condition
$\abs{W^z-W}\leq \delta$ implies that 
$\abs{\dk(W,Z)-\dk(W^z,Z)}$ is bounded by a constant times $\delta$.  From this point we will only 
need to show that $\dk(W^z, Z)$ is of order $\delta$, which is not as difficult due heuristically to the fact that
the zero-bias transform is smooth (absolutely continuous with respect to Lebesgue measure).

We implement the first part of the program.  For $z\in \IR$,
\ban
\IP(W\leq z)-\IP(Z\leq z)&\leq \IP(W\leq z) - \IP(Z\leq z+\delta)+ \IP(Z\leq z+\delta)- \IP(Z\leq z) \notag \\
	&\leq \IP(W^z\leq z+\delta)-\IP(Z\leq z+\delta)+\frac{\delta}{\sqrt{2\pi}} \notag \\
	&\leq \dk(W^z, Z)+\frac{\delta}{\sqrt{2\pi}}, \label{93}
\ee
where the second inequality follows since $\{W\leq z\}\subseteq \{W^z \leq z+\delta\}$ and since
$Z$ has density bounded by $(2\pi)^{-1/2}$.
Similarly, 
\ba
\IP(W\leq z)-\IP(Z\leq z)&\geq \IP(W\leq z) - \IP(Z\leq z-\delta)+ \IP(Z\leq z-\delta)- \IP(Z\leq z) \\
	&\geq \IP(W^z\leq z-\delta)-\IP(Z\leq z-\delta)-\frac{\delta}{\sqrt{2\pi}},
\ee
which after taking the supremum over $z$ and combining with \eq{93} implies that
\ban
\abs{\dk(W,Z)-\dk(W^z, Z)}\leq \frac{\delta}{\sqrt{2\pi}}. \label{94}
\ee
Now, by Corollary \ref{cor91} (and using the notation there), we have
\ban
\dk(W^z, Z)\leq \sup_{f\in\%F}\left|\IE[f'(W^z)-W^zf(W^z)]\right|, \label{94aa}
\ee
and for $f\in \%F$, we find after using the definition of the zero-bias transform and Lemma \ref{lem91} 
\ban
\left|\IE[f'(W^z)-W^zf(W^z)]\right|&=\left|\IE[Wf(W)-W^zf(W^z)]\right| \notag \\
			&\leq\IE\left[\left(\abs{W}+\frac{\sqrt{2\pi}}{4}\right)\left|W^z-W\right|\right] \notag \\
			&\leq \delta\left(1+\frac{\sqrt{2\pi}}{4}\right). \label{95}
\ee
Combining \eq{94}, \eq{94aa}, and \eq{95} yields the theorem.
\end{proof}
Theorem \ref{thm91} can be applied to sums of independent random variables which are almost surely bounded (note that $W$ bounded
implies $W^z$ bounded), and can also be used to derive a bound in Hoeffding's combinatorial CLT
under some boundedness assumption.  

\subsubsection{Exchangeable pairs}
To get better rates using exchangeable pairs, we again assume a boundedness condition.
A slightly more general version of this theorem appears in \cite{shsu06}.
\begin{theorem}\label{thm92}
If $(W, W')$ is an $a$-Stein pair with $\var(W)=1$ and such that $|W'-W|\leq \delta$, then
\ba
\dk(W,Z)\leq \frac{\sqrt{\var\left(\IE[(W'-W)^2|W]\right)}}{2a}+\frac{\delta^3}{2a}+\frac{3\delta}{2}.
\ee
\end{theorem}
\begin{proof}
Let $f_x$ the bounded solution of \eq{91}.  Exchangeability implies
\ba
\IE[Wf_x(W)]=\frac{1}{2a}\IE[(W'-W)(f_x(W')-f_x(W))],
\ee
so that we can see
\ban
\IE[f_x'(W)-Wf_x(W)]&=\IE\left[f_x'(W)\left(1-\frac{(W'-W)^2}{2a}\right)\right] \label{98}\\
	&\qquad+\IE\left[\frac{W'-W}{2a}\int_{0}^{W'-W}\left[f_x'(W)-f_x'(W+t)\right]dt\right]. \label{99}
\ee
Exactly as in the proof of Theorem \ref{thm51}, 
(the result analogous to Theorem \ref{thm92} but for the Wasserstein metric)
the term \eq{98} contributes the first error term from the theorem (using the bounds
of Lemma \ref{lem91}).  Now, since
$f_x$ satisfies \eq{91}, we can rewrite \eqref{99}
\ban
&\IE\left[\frac{W'-W}{2a}\int_{0}^{W'-W}\left[Wf_x(W)-(W+t)f_x(W+t)\right]dt\right] \label{911} \\
&\qquad +\IE\left[\frac{W'-W}{2a}\int_{0}^{W'-W}\left[\I[W\leq x]-\I[W+t\leq x\right]dt\right], \label{912}
\ee
and we can apply Lemma \ref{91} to find that the absolute value of \eq{911} is bounded above by
\ba
\IE\left[\frac{\abs{W'-W}}{2a}\int_{0}^{W'-W}\left(\abs{W}+\frac{\sqrt{2\pi}}{4}\right)\abs{t}dt\right]
\leq \IE\left[\frac{\abs{W'-W}^3}{4a}\left(\abs{W}+\frac{\sqrt{2\pi}}{4}\right)\right]\leq \frac{\delta^3}{2a}.
\ee
In order to bound the absolute value of \eq{912}, we consider separately the cases $W'-W$ positive and negative.  For example,
\ba
&\left|\IE\left[\frac{(W'-W)\I[W'<W]}{2a}\int_{W'-W}^{0}\I[x<W\leq x-t]dt\right]\right| \\
&\qquad\leq \frac{1}{2a}\IE\left[(W'-W)^2\I[W'<W]\I[x<W\leq x+\delta]\right],
\ee 
where we have used that $\abs{W'-W}\leq \delta$.  A similar inequality can be obtained for $W'>W$ and 
combining these terms implies that the absolute value of \eq{912} is bounded above 
\ban
\frac{1}{2a}\IE\left[(W'-W)^2\I[x<W\leq x+\delta]\right]. \label{9112}
\ee
Lemma \ref{lem92} below shows \eq{9112} is bounded above by $3\delta /2$,
which proves the theorem.
\end{proof}
\begin{lemma}\label{lem92}
If $(W, W')$ is an $a$-Stein pair with $\var(W)=1$ and such that $|W'-W|\leq \delta$, then for all $x\in \IR$
\ba
\IE\left[(W'-W)^2\I[x<W\leq x+\delta]\right]\leq 3\delta a.
\ee
\end{lemma}
\begin{proof}
Let $g'(w)=\I[x-\delta<w\leq x+2\delta]$ and $g(x+\delta/2)=0$. Using that $\norm{g}\leq 3\delta/2$ in the first inequality below, we have
\ba
3\delta a&\geq 2a \IE[Wg(W)] \\
	&=\IE\left[(W'-W)(g(W')-g(W))\right] \\
	&=\IE\left[(W'-W)\int_0^{W'-W}g'(W+t) dt\right] \\
	&\geq \IE\left[(W'-W)\int_0^{W'-W}\I[x-\delta<W+t\leq x+2\delta]\I[x<W\leq x+\delta] dt\right] \\
	&=\IE\left[(W'-W)^2\I[x<W\leq x+\delta]\right],
\ee
as desired.
\end{proof}
Theorem \ref{thm92} can be applied to sums of independent random variables which are almost surely bounded,
and can also be applied to the anti-voter model to yield rates in the Kolmogorov metric that are comparable to those we obtained
in the Wasserstein metric in Section \ref{secav}.

\section{Poisson Approximation}\label{poi}
One great advantage of Stein's method is that it can easily be adapted to various distributions and metrics. In this section we 
develop Stein's method for bounding the total variation distance (see Section \ref{probmet}) 
between a distribution of interest and the Poisson distribution.
We will move quickly through the material analogous to that of Section \ref{stnorm} 
for normal approximation, as the general framework is similar.  We follow the exposition of \cite{bhj92}.

\begin{lemma}\label{lem101}
For $\lambda>0$, define the functional operator $\mathcal{A}$ by
\ba
\mathcal{A}f(k)=\lambda f(k+1)-kf(k).
\ee
\begin{enumerate}
\item If the random variable $Z$ has the Poisson distribution with mean $\lambda$, then 
$\IE\%{A}f(Z)=0$ for all bounded $f$.
\item If for some non-negative integer-valued random variable $W$, $\IE\%{A}f(W)=0$ for all bounded functions 
$f$, then $W$ has the Poisson distribution with mean $\lambda$.
\end{enumerate}
The operator $\%A$ is referred to as a characterizing operator of the Poisson distribution.
\end{lemma}
Before proving the lemma, we state one more result and then its consequence.
\begin{lemma}\label{lem102}
Let $\P$ denote probability with respect to a Poisson distribution with mean $\lambda$ and $A\subseteq\IN\cup\{0\}$.  The unique
solution $f_A$  of
\ban
\lambda f_A(k+1)-kf_A(k)=\I[k\in A]-\P(A) \label{101}
\ee
with $f_A(0)=0$ is given by
\ba
f_A(k)=\lambda^{-k}e^\lambda (k-1)!\left[\P(A\cap U_k)-\P(A)\P(U_k)\right],
\ee
where $U_k=\{0, 1, \ldots, k-1\}$.
\end{lemma}
Analogous to normal approximation, this setup immediately yields the following promising result.
\begin{corollary}\label{cor101}
If $W\geq0$ is an integer-valued random variable with mean $\lambda$, then
\ba
\left|\IP(W\in A)-\P(A)\right|=\left|\IE[\lambda f_A(W+1) - Wf_A(W)]\right|.
\ee
\end{corollary}
\begin{proof}[Proof of Lemma \ref{lem102}]
The relation \eq{101} defines $f_A$ recursively, so it is obvious that the solution is unique 
under the boundary condition $f_A(0)=0$.  The fact that the solution is as claimed can be easily
verified by substitution into the recursion \eq{101}.
\end{proof}
\begin{proof}[Proof of Lemma \ref{lem101}]
Item 1 follows easily by direct calculation: if $Z\sim\Po(\lambda)$ and $f$ is bounded, then
\ba
\lambda\IE[f(Z+1)]&=e^{-\lambda}\sum_{k=0}^\infty\frac{\lambda^{k+1}}{k!}f(k+1) \\
	&=e^{-\lambda}\sum_{k=0}^\infty\frac{\lambda^{k+1}}{(k+1)!}(k+1)f(k+1) \\
	&=\IE[Zf(Z)].
\ee
For Item 2, let $\IE\%{A}f(W)=0$ for all bounded functions $f$.  Lemma \ref{lem103} below shows that
$f_k\equiv f_{\{k\}}$ is bounded, and
then $\IE\%{A}f_k(W)=0$ implies that $W$ has Poisson point probabilities.  
Alternatively, for $j\in \IN\cup\{0\}$, we could take $f(k)=\I[k=j]$ so that the $\IE\%{A}f(W)=0$ implies that
\ba
\lambda\IP(W=j-1)=j\IP(W=j),
\ee
which is defining since $W$ is a non-negative integer-valued random variable.  A third proof can be obtained by
taking $f(k)=e^{-uk}$, from which the Laplace transform of $W$ can be derived.
\end{proof} 
We now derive useful properties of the solutions $f_A$ of \eq{101}.
\begin{lemma}\label{lem103}
If $f_A$ solves \eqref{101}, then
\ban
\norm{f_A}\leq\min\left\{1, \lambda^{-1/2}\right\} \mbox{ and \,} \norm{\Delta f_A}\leq \frac{1-e^{-\lambda}}{\lambda}\leq \min\left\{1, \lambda^{-1}\right\},
\label{102}
\ee
where $\Delta f(k):=f(k+1)-f(k)$.
\end{lemma}
\begin{proof}
The proof of Lemma \ref{lem103} follows from careful analysis.  We prove the second assertion and refer to \cite{bhj92} for further details.
Upon rewriting
\ba
f_A(k)=\lambda^{-k}(k-1)!e^\lambda\left[\P(A\cap U_k)\P(U_k^c)-\P(A\cap U_k^c)\P(U_k)\right],
\ee
some consideration leads us to observe that for $j\geq1$, $f_j:=f_{\{j\}}$ satisfies
\begin{itemize}
\item $f_j(k)\leq 0$ for $k\leq j$ and $f_j(k)\geq 0$ for $k> j$,
\item $\Delta f_j(k)\leq 0$ for $k\not=j$, and  $\Delta f_j(j)\geq 0$,
\item $\Delta f_j(j)\leq \min\left\{j^{-1}, (1-e^{-\lambda})/\lambda\right\}$.
\end{itemize}
And also $\Delta f_0(k)<0$.  Since
\ba
\Delta f_A(k)=\sum_{j\in A} f_j(k)
\ee
is a sum of terms which are all negative except for at most one, we find
\ban
\Delta f_A(k)\leq \frac{1-e^{-\lambda}}{\lambda}. \label{105}
\ee
Since $f_{A^c}=-f_A$, \eq{105} yields the second assertion.
\end{proof}

We can now state our main Poisson approximation theorem which follows from Corollary \ref{cor101} and Lemma \ref{lem103}.
\begin{theorem}\label{thm101}
Let $\%F$ be the set of functions satisfying \eq{102}.  If $W\geq0$ is an integer-valued random variable with mean $\lambda$ and $Z\sim\Po(\lambda)$, then
\ban
\dtv(W,Z)\leq \sup_{f\in\%F}\left|\IE[\lambda f(W+1) - Wf(W)]\right|. \label{1019}
\ee
\end{theorem}

We are ready to apply Theorem \ref{thm101} to some examples, but first some remarks.  Recall 
that our main strategy for normal approximation was to find some structure in $W$, the random variable of interest,
that allows us to compare $\IE[Wf(W)]$ to $\IE[f'(W)]$ for appropriate $f$.  The canonical such structures were 
\begin{enumerate}
\item Sums of independent random variables,
\item Sums of locally dependent random variables,
\item Exchangeable pairs,
\item Size-biasing,
\item Zero-biasing.
\end{enumerate}
Note that each of these structures essentially provided a way to break down
$\IE[Wf(W)]$ into a functional of $f$ and some auxiliary random variables. 
Also, from the form of the Poisson characterizing operator, we want to find some  
structure in $W$ (the random variable of interest) that allows us to 
compare $\IE[Wf(W)]$ to $\lambda\IE[f(W+1)]$ for appropriate $f$.  These two observations imply that the first four items
on the list above may be germane to
Poisson approximation, which is exactly the program we will pursue (since zero-biasing involves $f'$, we won't
find use for it in our discrete setting).

\subsection{Law of small numbers}
It is well known that if $W_n\sim\Bi(n,\lambda/n)$ and $Z\sim\Po(\lambda)$ then $\dtv(W_n, Z)\to 0$ as $n\to \infty$, and
it is not difficult to obtain a rate of this convergence.  From this fact, it is easy to believe that
if $X_1, \ldots, X_n$ are independent indicators with $\IP(X_i=1)=p_i$, then $W=\sum_{i=1}^n X_i$
will be approximately Poisson if $\max_i p_i$ is small.  In fact, we will show the following result.
\begin{theorem}
Let $X_1, \ldots, X_n$ independent indicators with $\IP(X_i=1)=p_i$, $W=\sum_{i=1}^n X_i$, and $\lambda=\IE[W]=\sum_i p_i$.
If $Z\sim\Po(\lambda)$, then
\ba
\dtv(W, Z)&\leq \min\{1, \lambda^{-1}\} \sum_{i=1}^n p_i^2 \\
	&\leq \min\{1, \lambda\} \max_i p_i.
\ee
\end{theorem}\label{thm102}
\begin{proof}
The second inequality is clear and is only included to address the discussion preceding the theorem.  For the first
inequality, we apply Theorem \ref{thm101}.  Let $f$ satisfy \eq{102} and note that
\ban
\IE[Wf(W)]&=\sum_{i=1}^n\IE[X_i f(W)] \notag\\
	&=\sum_{i=1}^n \IE[f(W)|X_i=1] \IP[X_i=1] \notag\\
	&=\sum_{i=1}^n p_i\IE[f(W_i+1)], \label{108}
\ee
where $W_i=W-X_i$ and \eq{108} follows since $X_i$ is independent of $W_i$.  Since $\lambda f(W+1)=\sum_i p_i f(W+1)$, we obtain
\ba
\left|\IE[\lambda f(W+1)-Wf(W)]\right|&=\left|\sum_{i=1}^n p_i\IE[f(W+1)-f(W_i+1)]\right| \\
	&\leq \sum_{i=1}^n p_i \norm{\Delta f} \IE\abs{W-W_i}\\
	&=\min\{1, \lambda^{-1}\} \sum_{i=1}^n p_i \IE[X_i],
\ee
where the inequality is by rewriting $f(W+1)-f(W_i+1)$ as a telescoping sum of $\abs{W-W_i}$ first differences of $f$.  Combining this
last calculation with Theorem \ref{thm101} yields the desired result. 
\end{proof}

\subsection{Dependency neighborhoods}
Analogous to normal approximation, we can generalize Theorem \ref{thm102} to sums of locally dependent variables \cite{agg89, agg90}.
\begin{theorem}\label{thm103}
Let $X_1, \ldots, X_n$ indicator variables with $\IP(X_i=1)=p_i$, $W=\sum_{i=1}^n X_i$, and $\lambda=\IE[W]=\sum_i p_i$.
For each $i$, let $N_i\subseteq\{1, \ldots, n\}$ such that $X_i$ is independent
of $\{X_j: j\not\in N_i\}$.  If $p_{i j}:=\IE[X_i X_j]$ and $Z\sim\Po(\lambda)$, then
\ba
\dtv(W,Z)\leq \min\{1, \lambda^{-1}\}\left(\sum_{i=1}^n\sum_{j\in N_i} p_i p_j +\sum_{i=1}^n\sum_{j\in N_i/\{i\}} p_{i j}\right).
\ee
\end{theorem}  
\begin{remark}
The neighborhoods $N_i$ can be defined with greater flexibility (i.e. dropping the assumption that $X_i$ is independent of the variables not 
indexed by $N_i$) at the cost of an additional error term that (roughly) measures dependence (see \cite{agg89, agg90}).
\end{remark}
\begin{proof}
We want to mimic the proof of Theorem \ref{thm102} up to \eqref{108}, the point where the hypothesis of independence is used.
Let $f$ satisfy \eq{102}, $W_i=W-X_i$, and $V_i=\sum_{j\not\in N_i} X_j$.  Since $X_if(W)=X_if(W_i+1)$ almost surely, we find
\ban
\IE[\lambda f(W+1)-Wf(W)]&=\sum_{i=1}^n p_i\IE[f(W+1)-f(W_i+1)]  \label{109}\\
	&\qquad+\sum_{i=1}^n \IE[(p_i-X_i)f(W_i+1)]  \label{1010}
\ee
As in the proof of Theorem \ref{thm102}, the absolute value of \eq{109} is bounded above by $\norm{\Delta f} \sum_{i} p_i^2$.  Due to 
the independence of $X_i$ and $V_i$, and the fact that $\IE[X_i]=p_i$, we find that \eq{1010} is equal to
\ba
\sum_{i=1}^n \IE[(p_i-X_i)(f(W_i+1)-f(V_i+1))],
\ee
so that the absolute value of \eq{1010} is bounded above by
\ba
\norm{\Delta f}\sum_{i=1}^n \IE\bigg[\left|p_i-X_i\right| \left|W_i-V_i\right|\bigg]&\leq \norm{\Delta f}\sum_{i=1}^n \IE\bigg[(p_i+X_i)\sum_{j\in N_i/\{i\}}X_j\bigg] \\
	&=\norm{\Delta f}\sum_{i=1}^n\sum_{j\in N_i/\{i\}}\left(p_i p_j + p_{i j}\right).
\ee
Combining these bounds for \eq{109} and \eq{1010} yields the theorem. 
\end{proof}

\subsubsection{Application: Head runs}
In this section we consider an example that arises in an application from biology,
that of DNA comparison.  We postpone discussion of the details of this relation until the end of the section.

In a sequence of zeroes and ones we call an occurrence of the
pattern $\cdots011\cdots10$ (or
$11\cdots10\cdots$ or $\cdots011\cdots1$ at the boundaries of the sequence) with exactly $k$ ones 
a head run of length $k$.  
Let $W$ be the number of head runs of length at least $k$ in a sequence of $n$ independent tosses
of a coin with head probability $p$.  More precisely, let $Y_1, \ldots, Y_n$ be i.i.d. indicator variables with
$\IP(Y_i=1)=p$ and let 
\ba
X_1=\prod_{j=1}^kY_j,
\ee
and for $i=2, \ldots, n-k+1$ let 
\ba
X_i=(1-Y_{i-1})\prod_{j=0}^{k-1}Y_{i+j}.
\ee
Then $X_i$ is the indicator that a run of ones of length at least $k$ begins at position $i$ in the sequence
$(Y_1, \ldots, Y_n)$ so that we set $W=\sum_{i=1}^{n-k+1} X_i$.  Note that the factor $1-Y_{i-1}$ is used to
``de-clump" the runs of length greater than $k$ so that we do not
count the same run more than once.  At this point we can apply Theorem \ref{thm103} with only a little effort to obtain the following result.
\begin{theorem}\label{thm104}
Let $W$ be the number of head runs of at least length $k$ in a sequence of $n$ independent tosses
of a coin with head probability $p$ as defined above.  If $\lambda=\IE[W]=p^k((n-k)(1-p)+1)$ and $Z\sim \Po(\lambda)$, then
\ban
\dtv(W,Z)\leq \lambda^2\frac{2k+1}{n-k+1}+2\lambda p^k. \label{1011}
\ee
\end{theorem}
\begin{remark}\label{re101}
Although Theorem \ref{thm104} provides an error for all $n, p, k$, it can also be interpreted asymptotically as $n \to \infty$ and $\lambda$
bounded away from zero and infinity.  Roughly, if
\ba
k=\frac{\log(n(1-p))}{\log(1/p)}+c
\ee
for some constant $c$, then for fixed $p$, $\lim_{n\to\infty} \lambda=p^c$. In this case the bound \eq{1011} is of order $\log(n)/n$.
\end{remark}
\begin{proof}[Proof of Theorem \ref{thm104}]
As discussed in the remarks preceding the theorem, $W$ has representation as a sum of indicators: $W=\sum_{i=1}^{n-k+1} X_i$.
The fact that $\lambda$ is as stated follows from this representation using that $\IE[X_1]=p^k$ and $\IE[X_i]=(1-p)p^k$ for $i\not=1$.

We will apply Theorem \ref{thm103} with $N_i=\{1\leq j \leq n-k+1: \abs{i-j}\leq k\}$ which clearly has the property that $X_i$ is independent
of $\{X_j: j\not\in N_i\}$.  Moreover, if $j\in N_i/\{i\}$, then $\IE[X_i X_j]=0$ since two runs of length at least $k$ cannot begin within $k$ positions
of each other.  Theorem \ref{thm103} now implies
\ba
\dtv(W, Z)\leq \sum_{i=1}^n\sum_{j\in N_i} \IE[X_i]\IE[X_j].
\ee
It only remains to show that this quantity is bounded above by \eq{1011}
which follows by grouping and counting the terms of the sum into those 
that contain $\IE[X_1]$ and those that do not.
\end{proof}

A related quantity which is of interest in the biological application below is 
$R_n$, the length of the longest head run in $n$ independent coin tosses.  Due to the equality of events, we have
$\IP(W=0)=\IP(R_n<k)$, so that we can use Remark \ref{re101} to roughly state 
\ba
\left|\IP\left(R_n-\frac{\log(n(1-p))}{\log(1/p)}<x\right)-e^{-p^x}\right|\leq C \left(\frac{\log(n)}{n}\right).
\ee
The inequality above needs some qualification due to the fact that $R_n$ is integer-valued, but
it can be made precise - see \cite{agg89, agg90, agw90} for more details.

Theorem \ref{thm104} was relatively simple to derive,
but many embellishments are possible which can be also handled similarly, but with more technicalities.
For example, for $0< a \leq 1$, we can define a ``quality $a$" run of length $j$ to be a run of length $j$ with at least $aj$ heads.  We could then
take $W$ to be the number of quality $a$ runs of length at least $k$ and $R_n$ to be the longest quality $a$ run in $n$ independent coin tosses.
A story analogous to that above emerges.  

These particular results can also be viewed as elaborations of the classical theorem:
\begin{theorem}[\ER\ Law]
If $R_n$ is the longest quality $a$ head run in a sequence of $n$ independent tosses
of a coin with head probability $p$ as defined above, then almost surely,
\ba
\frac{R_n}{\log(n)}\to \frac{1}{H(a, p)},
\ee
where for $0< a< 1$, $H(a, p)=a\log(a/p)+(1-a)\log((1-a)/(1-p))$, and $H(1, p)=\log(1/p)$.
\end{theorem}
\begin{remark}
Some of the impetus for the results above and especially their embellishments stems from an application in computational biology - see
\cite{agg89, agg90, agw90, wat95} for an entry into this literature.  We briefly describe this application here.

DNA is made up of long sequences of the letters $A, G, C$, and $T$ which
stand for certain amino acids.  Frequently it is desirable to know how closely\footnote{For example, whether the two sequences
have a similar biological function or whether one sequence could be transformed to the other by few mutations.} 
two sequences of DNA are related.  

Assume for simplicity that the two sequences of DNA to be compared both have length $n$.  One possible
measure of closeness between these sequences is the length of the longest run where the sequences agree
when compared coordinate-wise.  More precisely, if sequence \textbf{A} is $A_1 A_2 \cdots A_n$, sequence \textbf{B} 
is $B_1 B_2 \cdots B_n$, and we define $Y_i=\I[A_i=B_i]$, then the measure of closeness between the sequences 
\textbf{A} and \textbf{B} would be the length of the longest run of ones in 
$(Y_1, \ldots, Y_n)$.

Now, given the sequences \textbf{A} and \textbf{B}, how long should the longest run be in order to consider them close?
The usual statistical setup to handle this question is to assume a probabilistic model under the hypothesis that the 
sequences are not related, and then compute the probability of the event ``at least as long a run" as the observed run.  If this
probability is low enough, then it is likely that the sequences are closely related (assuming the model is accurate).  

We make the (likely unrealistic) assumption that sequences of DNA are generated as independent picks from the alphabet
$\{A, G, C, T\}$ under some probability distribution with frequencies $p_A$, $p_G$, $p_C$, and $p_T$.  The hypothesis
that the sequences are unrelated corresponds to the sequences being generated \emph{independently}.

In this framework, the distribution of the longest run between two unrelated sequences of DNA of length $n$ is exactly
$R_n$ above with $p:=\IP(Y_i=1)=p_A^2+p_G^2+p_C^2+p_T^2$.  Thus the work above can be used to approximate
tail
probabilities of the longest run length
under the assumption that two sequences of DNA are unrelated and then used to determine the likeliness of the observed longest run lengths.
  
\end{remark}

\subsection{Size-bias Coupling}
The most powerful method of rewriting $\IE[Wf(W)]$ so that it can be usefully compared to $\IE[W]\IE[f(W+1)]$ is through the size-bias
coupling already defined in Section \ref{sbc} - recall the relevant definitions and properties there.  The book \cite{bhj92} is almost entirely
devoted to Poisson approximation through the size-bias coupling (although that terminology is not used), so we will spend
some time fleshing out their powerful and general results.

\begin{theorem}\label{thm121}
Let $W\geq0$ an integer-valued random variable with $\IE[W]=\lambda>0$ and let $W^s$ be a size-bias coupling of $W$.  
If $Z\sim \Po(\lambda)$,
then
\ba
\dtv(W,Z)\leq \min\{1, \lambda\}\IE\abs{W+1-W^s}.
\ee 
\end{theorem}
\begin{proof}
Let $f$ bounded and $\norm{\Delta f}\leq \min\{1, \lambda^{-1}\}$.  Then
\ba
\left|\IE[ \lambda f(W+1) -Wf(W)]\right|&=\lambda\left|\IE[ f(W+1) -f(W^s)]\right|\\
	&\leq \lambda \norm{\Delta f}\IE\abs{W+1-W^s},
\ee
where we have used the definition of the size-bias distribution and rewritten $f(W+1) -f(W^s)$ as a telescoping sum of $\abs{W+1-W^s}$ terms.
\end{proof}

Due to the canonical ``law of small numbers" for Poisson approximation, 
we will mostly be concerned with approximating a sum of indicators by a Poisson distribution.  Recall the following 
construction of a size-bias coupling from Section \ref{sbc}, and useful special case.

\begin{corollary}\label{cor121}
Let $X_1, \ldots, X_n$ be indicator variables with $\IP(X_i=1)=p_i$, $W=\sum_{i=1}^n X_i$, and $\lambda=\IE[W]=\sum_i p_i$.
If for each $i=1,\ldots, n$, $(X_j^{(i)})_{j\not=i}$ has the distribution of $(X_j)_{j\not=i}$ conditional on $X_i=1$ and
$I$ is a random variable independent of all else such that $\IP(I=i)=p_i/\lambda$, then $W^s=\sum_{j\not=I} X_j^{(I)} +1$
has the size-bias distribution of $X$.
\end{corollary}
\begin{corollary}\label{cor122}
Let $X_1, \ldots, X_n$ be exchangeable indicator variables 
and
let $(X_j^{(1)})_{j\not=1}$ have the distribution of $(X_j)_{j\not=1}$ conditional on $X_1=1$.
If $W=\sum_{i=1}^n X_i$, then the size-bias distribution of $X$ can be represented
by $X^s=\sum_{j\not=1}X_j^{(1)} + 1$.
\end{corollary}
\begin{proof}
Corollary \ref{cor121} was proved in Section \ref{sbc} and 
and Corollary \ref{cor122} follows from the fact that exchangeability
implies that $I$ is uniform and $\sum_{j\not=i}X_j^{(i)} + X_i^s\ed \sum_{j\not=1}X_j^{(1)} + X_1^s$.
\end{proof}

\begin{example}[Law of small numbers]
Let $W=\sum_{i=1}^n X_i$ where the $X_i$ are independent indicators with $\IP(X_i=1)=p_i$. According to 
Corollary \ref{cor121}, in order to size-bias $W$, we first choose an index $I$ with $\IP(I=i)=p_i/\lambda$, where
$\lambda=\IE[W]=\sum_ip_i$.  Given $I=i$ we construct $X_j^{(i)}$ having the distribution of $X_j$ conditional
on $X_i=1$.  However, by independence, $(X_j^{(i)})_{j\not=i}$ has the same distribution as $(X_j)_{j\not=i}$ so that we can take
$W^s=\sum_{j\not= I} X_j+1$.  Applying Theorem \ref{thm121} we find that for $Z\sim\Po(\lambda)$, 
\ba
\dtv(W,Z)\leq\min\{1, \lambda\}\IE[X_I]=\min\{1, \lambda\}\sum_{i=1}^n\frac{p_i}{\lambda}\IE[X_i]=\min\{1, \lambda^{-1}\}\sum_{i=1}^n p_i^2,
\ee
which agrees with our previous bound for this example.
\end{example}

\begin{example}[Isolated Vertices]\label{ex122}
Let $W$ be the number of isolated vertices in an \ER\ random graph on $n$ vertices with edge probabilities $p$.
Note that $W=\sum_{i=1}^n X_i$, where $X_i$ is the indicator that vertex $v_i$ (in some arbitrary but fixed labeling)
is isolated.  We constructed a size-bias coupling of $W$ in Section \ref{sbc} 
using Corollary \ref{cor121}, and we can simplify
this coupling by using Corollary \ref{cor122}\footnote{This simplification would not have yielded a useful error bound in Section \ref{sbc} 
since the size-bias normal approximation theorem contains a variance term; there 
the randomization provides an extra factor of $1/n$.
}
as follows.  

We first generate an \ER\ random graph $G$, and then erase all edges connected to vertex $v_1$.  
Then take $X_j^{(1)}$ be
the indicator that vertex $v_j$ is isolated in this new graph.  
By the independence of the edges in the graph, 
it is clear that $(X_j^{(1)})_{j\not=1}$ has the distribution of $(X_j)_{j\not=1}$ conditional on $X_1=1$,
so that by Corollary \ref{cor122}, we can take $W^s=\sum_{j\not=1}X_j^{(1)} + 1$ and of course we take $W$
to be the number of isolated vertices in $G$. 

In order to apply Theorem \ref{thm121}, we only need to compute $\lambda=\IE[W]$ and
$\IE\abs{W+1-W^s}$.  From Example \ref{ex72} in Section \ref{sbc}, $\lambda=n(1-p)^{n-1}$ and from the construction above
\ba
\IE\abs{W+1-W^s}&=\IE\left|X_1+\sum_{j=2}^n X_j - X_j^{(1)} \right| \\
	&=\IE[X_1]+\sum_{j=2}^n\IE\left[X_j^{(1)}-X_j\right],
\ee
where we use the fact that $X_j^{(1)}\geq X_j$ which follows since we
can only increase the number of isolated vertices by erasing edges.
Thus, $X_j^{(1)}-X_j$ is equal to zero or one and the latter happens only if
vertex $v_j$ has degree one and is connected to $v_1$ which occurs with
probability $p(1-p)^{n-2}$.  Putting this all together in Theorem \ref{thm121}, we obtain the following.
\begin{proposition}
Let $W$ the number of isolated vertices in an \ER\ random graph and $\lambda=\IE[W]$. If $Z\sim\Po(\lambda)$, then
\ba
\dtv(W, Z)&\leq \min\{1, \lambda\}\left((n-1)p(1-p)^{n-2}+(1-p)^{n-1}\right)\\
	&\leq \min\{\lambda, \lambda^2\}\left(\frac{p}{1-p}+\frac{1}{n}\right).
\ee
\end{proposition}
To interpret this result asymptotically,
if $\lambda$ is to stay away from zero and infinity as $n$ gets large, $p$ must be of order $\log(n)/n$,  in which case the error above
is of order $\log(n)/n$.
\end{example}

\begin{example}[Degree $d$ vertices]
We can generalize Example \ref{ex122} by taking $W$ to
be the number of degree $d\geq0$ vertices in an \ER\ random graph on $n$ vertices with edge probabilities $p$.
Note that $W=\sum_{i=1}^n X_i$, where $X_i$ is the indicator that vertex $v_i$ (in some arbitrary but fixed labeling)
has degree $d$.  We can construct a size-bias coupling of $W$
by using Corollary \ref{cor122}
as follows.  
Let $G$ be an \ER\ random graph.  
\begin{itemize}
\item If the degree of vertex $v_1$ is $d_1 \geq d$, then erase $d_1-d$ edges chosen 
uniformly at random from the $d_1$ edges connected to $v_1$.  
\item If the degree of vertex $v_1$ is $d_1 < d$, then add edges from $v_1$ to the $d-d_1$ vertices
not connected to $v_1$ chosen 
uniformly at random from the $n-d_1-1$ vertices unconnected to $v_1$.
\end{itemize}
Let $X_j^{(1)}$ be
the indicator that vertex $v_j$ has degree $d$ in this new graph.  
By the independence of the edges in the graph, 
it is clear that $(X_j^{(1)})_{j\not=1}$ has the distribution of $(X_j)_{j\not=1}$ conditional on $X_1=1$,
so that by Corollary \ref{cor122}, we can take $W^s=\sum_{j\not=1}X_j^{(1)} + 1$ and of course we take $W$
to be the number of isolated vertices in $G$. 

Armed with this coupling, we could apply Theorem \ref{thm121} to yield a bound in the variation distance
between $W$ and a Poisson distribution. However, the analysis for this particular example is a bit technical, so we
refer to Section 5.2 of \cite{bhj92} for the details.  

\end{example}

\subsubsection{Increasing size-bias couplings}

A crucial simplification occurred in Example \ref{ex122} because the size-bias coupling was increasing in a certain sense.
The following result quantifies this simplification.
\begin{theorem}\label{thm122}
Let $X_1, \ldots, X_n$ be indicator variables with $\IP(X_i=1)=p_i$, $W=\sum_{i=1}^n X_i$, and $\lambda=\IE[W]=\sum_i p_i$.
For each $i=1,\ldots, n$, let $(X_j^{(i)})_{j\not=i}$ have the distribution of $(X_j)_{j\not=i}$ conditional on $X_i=1$ and
let $I$ be a random variable independent of all else, such that $\IP(I=i)=p_i/\lambda$ so that
$W^s=\sum_{j\not=I} X_j^{(I)} +1$ has the size-bias distribution of $W$.
If $X_j^{(i)}\geq X_j$ for all $i\not= j$, and $Z\sim \Po(\lambda)$,
then 
\ba
\dtv(W, Z) \leq \min\{1, \lambda^{-1}\}\left(\var(W)-\lambda+2\sum_{i=1}^n p_i^2\right).
\ee
\end{theorem}   
\begin{proof}
Let $W_i=\sum_{j\not=i} X_j^{(i)} +1$.  From Theorem \ref{thm121} and the size-bias construction of Corollary \ref{cor121}, we have
\ban
\dtv(W,Z)&\leq \min\{1, \lambda^{-1}\}\sum_{i=1}^n p_i\IE\abs{W+1-W_i}\notag \\
	&=\min\{1, \lambda^{-1}\}\sum_{i=1}^n p_i\IE\left[\sum_{j\not=i}\left( X_j^{(i)}-X_j\right)+X_i\right] \notag  \\
	&=\min\{1, \lambda^{-1}\}\sum_{i=1}^n p_i\IE\left[W_i-W-1+2 X_i\right], \label{122}
\ee
where the penultimate equality uses the monotonicity of the size-bias coupling.  Using again the construction of the size-bias coupling we obtain
that \eq{122} is equal to
\ba
\min\{1, \lambda^{-1}\}\left(\lambda \IE[W^s]-\lambda^2-\lambda+2\sum_{i=1}^n p_i^2\right),
\ee
which yields the desired inequality by the definition of the size-bias distribution.
\end{proof}

\subsubsection{Application: Subgraph counts}\label{ss121}
Let $G=G(n,p)$ be an \ER\ random graph on $n$ vertices with edge probability $p$ and let $H$ be a graph
on $0<v_H\leq n$ vertices with $e_H$ edges and no isolated vertices.  We want to analyze the number
of copies of $H$ in $G$; that is, the number of subgraphs of the complete graph on $n$ vertices
which are isomorphic to $H$ which appear in $G$.  For example, we could take $H$ to be a triangle so that
$v_H=e_H=3$.  

Let $\Gamma$ be the set of all copies of $H$ in $K_n$, the complete graph on $n$ vertices and for $\alpha \in \Gamma$,
let $X_\alpha$ be the indicator that there is a copy of $H$ in $G$ at $\alpha$ and set $W=\sum_{\alpha\in\Gamma}X_\alpha$.
We now have the following result.
\begin{theorem}\label{thm123}
Let $W$ be the number of copies of a graph $H$ with no isolated vertices in $G$ as defined above and let $\lambda=\IE[W]$.
If $H$ has $e_H$ edges and $Z\sim\Po(\lambda)$, then
\ba
\dtv(W, Z) \leq \min\{1, \lambda^{-1}\}\left(\var(W)-\lambda+2\lambda p^{e_H}\right).
\ee
\end{theorem}
\begin{proof}
We will show that Theorem \ref{thm122} applies to $W$.
Since $W=\sum_{\alpha\in\Gamma}X_\alpha$ is a sum of exchangeable indicators, we can apply Corollary \ref{cor122}
to construct a size-bias coupling of $W$.  To this end, for a fixed $\alpha\in \Gamma$, let $X_\beta^{(\alpha)}$ be the
indicator that there is a copy of $H$ in $G\cup\{\alpha\}$ at $\beta$.  Here, $G\cup\{\alpha\}$ means we add the minimum
edges necessary to $G$ to have a copy of $H$ at $\alpha$.    The following three evident facts now imply the theorem:
\begin{enumerate}
\item $(X_\beta^{(\alpha)})_{\beta\not=\alpha}$ has the distribution of $\left(X_\beta\right)_{\beta\not=\alpha}$
given that $X_\alpha=1$.
\item For all $\beta\in \Gamma/\{\alpha\}$, $X_\beta^{(\alpha)}\geq X_\beta$.
\item $\IE[X_{\alpha}]=p^{e_H}$.
\end{enumerate} 
\end{proof}

Theorem \ref{thm123} is a very general result, but it can be difficult to interpret. That is,
what properties of a subgraph $H$ make $W$ approximately Poisson?  We can begin to answer that question
by expressing the mean and variance of $W$ in terms of properties of $H$ which yields the following.
\begin{corollary}\label{cor123}
Let $W$ be the number of copies of a graph $H$ with no isolated vertices in $G$ as defined above and let $\lambda=\IE[W]$.
For fixed $\alpha\in\Gamma$, let $\Gamma_\alpha^t\subseteq \Gamma$ be the set of subgraphs of $K_n$ isomorphic
to $H$ with exactly $t$ edges not in $\alpha$.  If $H$ has $e_H$ edges and $Z\sim\Po(\lambda)$, then
\ba
\dtv(W, Z)\leq \min\{1, \lambda\}\left(p^{e_H}+\sum_{t=1}^{e_H-1}\abs{\Gamma_\alpha^t}\left(p^t-p^{e_H}\right)\right).
\ee
\end{corollary}
\begin{proof}
The corollary follows after deriving the mean and variance of $W$.  The terms $\abs{\Gamma_\alpha^t}$ account
for the number of covariance terms for different types of pairs of indicators.  In detail, 
\ba
\var(W)&=\sum_{\alpha\in\Gamma}\var(X_\alpha)+\sum_{\alpha\in\Gamma}\sum_{\beta\not=\alpha}\cov(X_\alpha, X_\beta) \\
	&=\lambda (1-p^{e_H})+\sum_{\alpha\in\Gamma}\sum_{t=1}^{e_H}\sum_{\beta\in\Gamma_\alpha^t}\cov(X_\alpha, X_\beta) \\
	&=\lambda (1-p^{e_H})+\sum_{\alpha\in\Gamma}p^{e_H}\sum_{t=1}^{e_H}\sum_{\beta\in\Gamma_\alpha^t}\left(\IE[X_\beta|X_\alpha=1]-p^{e_H}\right) \\
	& =\lambda\left(1-p^{e_H}  + \sum_{t=1}^{e_H-1}\abs{\Gamma_\alpha^t}\left(p^t-p^{e_H}\right)\right),
\ee 
since $\lambda=\sum_{\alpha\in\Gamma}p^{e_H}$ and for $\beta\in\Gamma_\alpha^t$, $\IE[X_\beta|X_\alpha=1]=p^t$.
\end{proof}
It is possible to rewrite the error in other forms which can be used to make some general statements (see \cite{bhj92}, Chapter 5), but
we content ourselves with some examples.
\begin{example}[Triangles]
Let $H$ be a triangle.  In this case, $e_H=3$, $\abs{\Gamma_\alpha^2}=3(n-3)$, and $\abs{\Gamma_\alpha^1}=0$
since triangles either share one edge or all three edges (corresponding to $t=2$ and $t=0$).  Thus Corollary \ref{cor123}
implies that for $W$ the number of triangles in $G$ and $Z$ an appropriate Poisson variable, 
\ban
\dtv(W, Z)\leq \min\{1, \lambda\}\left(p^{3}+3(n-3)p^2(1-p)\right). \label{124}
\ee
Since $\lambda=\binom{n}{3}p^3$ we can view \eq{124} as an asymptotic result with $p$ of order $1/n$.  In this case,
\eq{124} is of order $1/n$.
\end{example}
\begin{example}[k-cycles]
More generally, we can let $H$ be a $k$-cycle (a triangle is $3$-cycle). Now note that
for some constants $c_t$ and $C_k$, 
\ba
\abs{\Gamma_\alpha^t}\leq \binom{k}{k-t}c_t n^{t-1}\leq C_k n^{t-1},
\ee
since we choose the $k-t$ edges shared in the $k$-cycle $\alpha$, and then we have
order $n^{t-1}$ sequences of vertices to create a cycle with $t$ edges outside of the $k-t$ edges shared with $\alpha$.
The second equality follows by maximizing $\binom{k}{k-t}c_t$ over the possible values of $t$.  We can now find for
$W$ the number of 
$k$-cycles in $G$ and $Z$ an appropriate Poisson variable, 
\ban 
\dtv(W, Z)\leq\min\{1, \lambda\}\left(p^{k}+ C_k p \sum_{t=1}^{k-1} (np)^{t-1}\right). \label{125}
\ee
To interpret this bound asymptotically, we note that
$\lambda=\abs{\Gamma}p^k$ and 
\ba
\abs{\Gamma}=\binom{n}{k}\frac{(k-1)!}{2}p^k,
\ee
since the number of non-isomorphic $k$-cycles on $K_k$ is $k!/(2k)$ (since $k!$ is the number of permutations of the vertices, which
over counts by a factor of $2k$ due to reflections and rotations).  Thus $\lambda$ is of order $(np)^k$ for fixed $k$ so that 
we take $p$ to be of order $1/n$ and in this regime \eq{125} is of order $1/n$.
\end{example}

Similar results can be derived for \emph{induced} and \emph{isolated} subgraph counts - again we refer to \cite{bhj92}, Chapter 5.

\subsubsection{Implicit coupling}
In this section we show that it can be possible to apply Theorem \ref{thm122} without
constructing the size-bias coupling explicitly.  We first need some terminology.
\begin{definition}
We say a function $f:\IR^n\to\IR$ is increasing (decreasing) if for 
all $\textbf{x}=(x_1, \ldots, x_n)$ and $\textbf{y}=(y_1, \ldots, y_n)$ such that $x_i\leq y_i$ for all $i= 1, \ldots, n$,
we have $f(\textbf{x})\leq f(\textbf{y})$ ($f(\textbf{x})\geq f(\textbf{y})$).
\end{definition}
\begin{theorem}\label{thm124}
Let $\textbf{Y}=(Y_i)_{j=1}^N$ be a finite collection of independent indicators and assume $X_1, \ldots, X_n$ 
are increasing or decreasing functions from $\{0,1\}^N$ into $\{0, 1\}$.
If $W=\sum_{i=1}^n X_i(\textbf{Y})$ and $\IE[W]=\lambda$,
then
\ba
\dtv(W, Z) \leq \min\{1, \lambda^{-1}\}\left(\var(W)-\lambda+2\sum_{i=1}^n p_i^2\right).
\ee
\end{theorem}
\begin{proof}
We will show that there exists a size-bias coupling of $W$ satisfying the hypotheses of Theorem \ref{thm122}, which implies the result.  
From Lemma \ref{lem121} below, it is enough to show that 
$\cov(X_i(\textbf{Y}), \phi\circ\textbf{X}(\textbf{Y}))\geq0$ for all increasing indicator functions $\phi$.
However, since each $X_i(\textbf{Y})$ is an increasing or decreasing function applied to independent indicators, then so is
$\phi\circ\textbf{X}(\textbf{Y})$.
Thus we may apply
the FKG inequality (see Chapter 2 of \cite{lig05})  which in this case states
\ba
\IE[X_i(\textbf{Y}) \phi\circ\textbf{X}(\textbf{Y})]\geq \IE[X_i(\textbf{Y})]\IE[\phi\circ\textbf{X}(\textbf{Y})],
\ee
as desired.
\end{proof}
\begin{lemma}\label{lem121}
Let $\textbf{X}=(X_1, \ldots, X_n)$ be a vector of indicator variables
and let $\textbf{X}^{(i)}=(X_1^{(i)}, \ldots, X_n^{(i)})\ed\textbf{X}|X_i=1$.  Then the following are equivalent:
\begin{enumerate}
\item\label{i121} There exists a coupling such that $X_j^{(i)}\geq X_j$.
\item\label{i122} For all increasing indicator functions $\phi$, $\IE[\phi(\textbf{X}^{(i)})]\geq\IE[\phi(\textbf{X})]$.
\item\label{i123} For all increasing indicator functions $\phi$, $\cov(X_i, \phi(\textbf{X}))\geq0$.
\end{enumerate}
\end{lemma}
\begin{proof}
The equivalence \ref{i121}$\Leftrightarrow$\ref{i122} follows from a general version of Strassen's theorem which
can be found in \cite{lig05}.  In one dimension, Strassen's theorem says that there exists a coupling 
of random variables $X$ and $Y$ such that $X\geq Y$ if and only if $F_X(z)\leq F_Y(z)$ for all $z\in \IR$ where
$F_X$ and $F_Y$ are distribution functions.

The equivalence \ref{i122}$\Leftrightarrow$\ref{i123} follows from the following calculation.
\ba 
\IE[\phi(\textbf{X}^{(i)})]=\IE[\phi(\textbf{X})|X_i=1]=\IE[X_i\phi(\textbf{X})|X_i=1]=\frac{\IE[X_i\phi(\textbf{X})]}{\IP(X_i=1)}.
\ee
\end{proof}
\begin{example}[Subgraph counts]
Theorem \ref{thm124} applies to the example of Section \ref{ss121}, since the indicator of a copy of $H$ at a given location
is an increasing function of the edge indicators of the graph $G$.  
\end{example}
\begin{example}[Large degree vertices]
Let $d\geq 0$ and let $W$ be the number of vertices with degree at least $d$.  Clearly $W$ is a sum of indicators that are increasing
functions of the edge indicators of the graph $G$ so that Theorem \ref{thm124} can be applied.  After some technical analysis, we
arrive at the following result - see Section 5.2 of \cite{bhj92} for details.
If $q_1=\sum_{k\geq d}\binom{n-1}{k}p^k(1-p)^{n-k-1}$ and $Z\sim\Po(nq_1)$, then
\ba
\dtv(W,Z)\leq q_1+\frac{d^2(1-p)\left[\binom{n-1}{d}p^d(1-p)^{n-d-1}\right]^2}{(n-1)p q_1}.
\ee
\end{example}
\begin{example}[Small degree vertices]
Let $d\geq 0$ and let $W$ be the number of vertices with degree at most $d$.  Clearly $W$ is a sum of indicators that are decreasing
functions of the edge indicators of the graph $G$ so that Theorem \ref{thm124} can be applied. After some technical analysis, we
arrive at the following result - see Section 5.2 of \cite{bhj92} for details.
If $q_2=\sum_{k\leq d}\binom{n-1}{k}p^k(1-p)^{n-k-1}$ and $Z\sim\Po(nq_2)$, then
\ba
\dtv(W,Z)\leq q_2+\frac{(n-d-1)^2 p \left[\binom{n-1}{d}p^d(1-p)^{n-d-1}\right]^2}{(n-1)(1-p) q_2}.
\ee

\end{example}

\subsubsection{Decreasing size-bias couplings}
In this section we prove and apply a result complementary to Theorem \ref{thm122}.
\begin{theorem}\label{thm125}
Let $X_1, \ldots, X_n$ be indicator variables with $\IP(X_i=1)=p_i$, $W=\sum_{i=1}^n X_i$, and $\lambda=\IE[W]=\sum_i p_i$.
For each $i=1,\ldots, n$, let $(X_j^{(i)})_{j\not=i}$ have the distribution of $(X_j)_{j\not=i}$ conditional on $X_i=1$ and
let $I$ be a random variable independent of all else, such that $\IP(I=i)=p_i/\lambda$ so that
$W^s=\sum_{j\not=I} X_j^{(I)} +1$ has the size-bias distribution of $W$.  If 
$X_j^{(i)}\leq X_j$ for all $i\not= j$, and $Z\sim \Po(\lambda)$,
then 
\ba
\dtv(W, Z) \leq \min\{1, \lambda\}\left(1-\frac{\var(W)}{\lambda}\right).
\ee
\end{theorem}   
\begin{proof}
Let $W_i=\sum_{j\not=i} X_j^{(i)} +1$.  From Theorem \ref{thm121} and the size-bias construction of Corollary \ref{cor121}, we have
\ban
\dtv(W,Z)&\leq \min\{1, \lambda^{-1}\}\sum_{i=1}^n p_i\IE\abs{W+1-W_i}\notag \\
	&=\min\{1, \lambda^{-1}\}\sum_{i=1}^n p_i\IE\left[\sum_{j\not=i} \left(X_j- X_j^{(i)}\right)+X_i\right] \notag  \\
	&=\min\{1, \lambda^{-1}\}\sum_{i=1}^n p_i\IE\left[W-W_i+1\right], \label{128}
\ee
where the penultimate equality uses the monotonicity of the size-bias coupling.  Using again the construction of the size-bias coupling we obtain
that \eq{128} is equal to
\ba
\min\{1, \lambda^{-1}\}\left(\lambda^2-\lambda\IE[W^s]+\lambda\right),
\ee
which yields the desired inequality by the definition of the size-bias distribution.
\end{proof}
\begin{example}[Hypergeometric distribution]
Suppose we have $N$ balls in an urn in which $1\leq n \leq N$ are colored red and we draw $1\leq m \leq N$ balls uniformly at
random without replacement so that each of the $\binom{N}{m}$ subsets of balls is equally likely.  Let $W$ be the number of 
red balls.  It is well known that if $N$ is large and $m/N$ is small, then $W$ has approximately a binomial distribution since the dependence
diminishes.  Thus, we would also expect $W$ to be approximately Poisson distributed if in addition, $n/N$ is small.
We can use Theorem \ref{thm125} to make this heuristic precise.  

Label the red balls in the urn arbitrarily and let $X_i$ be the indicator that ball $i$ is chosen in the $m$-sample
so that we have the representation $W=\sum_{i=1}^nX_i$.  Since the $X_i$ are exchangeable, we can use Corollary \ref{cor122} 
to size-bias $W$. If the ball labelled one already appears in the $m$-sample then do nothing, otherwise, we 
force $X_1^s=1$ by adding the ball labelled one to the sample and putting back a ball chosen uniformly at random
from the initial $m$-sample.  If we let $X_i^{(1)}$ be the indicator that ball $i$ is in the sample 
after this procedure, then it is clear that $(X_i^{(1)})_{i\geq2}$ has the distribution of $(X_i)_{i\geq2}$ conditional on $X_1=1$
so that we can take $W^s=\sum_{i\geq2}X_i^{(1)} +1$.  

In the construction of $W^s$ above, no additional red balls labeled $2, \ldots, n$ can be added to the $m$-sample.
Thus $X_i^{(1)}\leq X_i$
for $i\geq 2$, and we
can apply Theorem \ref{thm125}.  A simple calculation yields
\ba
\IE[W]=\frac{nm}{N} \mbox{\, and \,} \var(W) = \frac{nm (N-n) (n-m)}{N^2 (N-1)},
\ee
so that for $Z\sim\Po(nm/N)$, 
\ba
\dtv(W, Z)\leq \min\left\{1, \frac{nm}{N}\right\}\left(\frac{n}{N-1}+\frac{m}{N-1}-\frac{nm}{N(N-1)}-\frac{1}{N-1}\right).
\ee
\begin{remark}
This same bound can be recovered after noting the well known fact that 
$W$ has the representation as a sum of \emph{independent} indicators  (see \cite{pit97}) 
which implies that Theorem \ref{thm125} can be applied.
\end{remark}
\end{example}
\begin{example}[Coupon collecting]\label{ex123}
Assume that a certain brand of cereal puts a toy in each carton.  There are $n$ distinct
types of toys and 
each week you pick up a carton of this cereal from the 
grocery store in such a way as to receive a uniformly random
type of toy, independent of the toys received previously. 
The classical coupon collecting problem asks the question
of how many cartons of cereal you must pick up in order to have
received all $n$ types of toys.

We formulate the problem as follows.  Assume you have $n$ boxes and $k$ balls are tossed independently into these boxes
uniformly at random.  Let $W$ be the number of empty boxes after tossing all $k$ balls into the boxes.
Viewing the $n$ boxes as types of toys and the $k$ balls as cartons of cereal, it is easy to see that the 
event $\{W=0\}$ corresponds to the event that $k$ cartons of cereal are sufficient to receive all $n$ types of toys.
We will use Theorem \ref{thm125} to show that $W$ is approximately Poisson which will yield an estimate with error
for the probability of this event.

Let $X_i$ be the indicator that box $i$ (under some arbitrary labeling) is empty after tossing the $k$ balls so that
$W=\sum_{i=1}^n X_i$.  Since the $X_i$ are exchangeable, we can use Corollary \ref{cor122} 
to size-bias $W$ by
first setting $X_1^s=1$ by emptying box $1$ (if it is not already empty) and then redistributing the
balls in box $1$ uniformly among boxes $2$ through $n$.
If we let $X_i^{(1)}$ be the indicator that box $i$ is empty  
after this procedure, then it is clear that $(X_i^{(1)})_{i\geq2}$ has the distribution of $(X_i)_{i\geq2}$ conditional on $X_1=1$
so that we can take $W^s=\sum_{i\geq2}X_i^{(1)} +1$.  

In the construction of $W^s$ above we can only add balls to boxes $2$ through $n$, which implies $X_i^{(1)}\leq X_i$ for $i\geq2$ so that we
can apply Theorem \ref{thm125}.  In order to apply the theorem we only need to compute the mean and variance of $W$.  First note
that $\IP(X_i=1)=((n-1)/n)^k$ so that 
\ba
\lambda:=\IE[W]=n\left(1-\frac{1}{n}\right)^k,
\ee
and also that for $i\not=j$, $\IP(X_i=1, X_j=1)=((n-2)/n)^k$ so that
\ba
\var(W)=\lambda\left[1-\left(1-\frac{1}{n}\right)^k\right]+n(n-1)\left[\left(1-\frac{2}{n}\right)^k-\left(1-\frac{1}{n}\right)^{2k}\right].
\ee
Using these calculations in Theorem \ref{thm125} yields that for $Z\sim\Po(\lambda)$, 
\ban
\dtv(W, Z)\leq \min\{1, \lambda\}\left(\left(1-\frac{1}{n}\right)^k+(n-1)\left[\left(1-\frac{1}{n}\right)^k - \left(1-\frac{1}{n-1}\right)^k\right]\right). \label{129}
\ee

In order to interpret this result asymptotically, let $k = n\log(n)-cn$ for some constant $c$ so that $\lambda=e^c$ is bounded away from zero and infinity as $n\to\infty$.  In this case \eq{129} is asymptotically of order
\ba
\frac{\lambda}{n}+\lambda\left[1-\left(\frac{n(n-2)}{(n-1)^2}\right)^k\right],
\ee 
and since for $0<a \leq 1$, $1-a^x\leq-\log(a)x$ and also $\log(1+x)\leq x$, we find
\ba
1-\left(\frac{n(n-2)}{(n-1)^2}\right)^k&\leq k \log\left(\frac{(n-1)^2}{n(n-2)}\right)  \leq \frac{k}{n(n-2)} =\frac{\log(n)-c}{(n-2)},
\ee
which implies
\ba
\dtv(W, Z)\leq C \frac{\log(n)}{n}.
\ee
\end{example}
\begin{example}[Coupon collecting continued]
We can embellish the coupon collecting problem of Example \ref{ex123} in a number of ways; recall
the notation and setup there.  For example, rather
than distribute the balls into the boxes uniformly and independently, we could
distribute each ball independently according to some probability distribution, say $p_i$ is the chance that a ball
goes into box $i$ with $\sum_{i=1}^n p_i=1$.  Note that Example \ref{ex123} had $p_i=1/n$ for all $i$.  

Let $X_i$ be the indicator that box $i$ (under some arbitrary labeling) is empty after tossing $k$ balls and
$W=\sum_{i=1}^n X_i$.
In this setting,
the $X_i$ are not necessarily exchangeable so that we use Corollary \ref{cor121} to construct $W^s$.  First we
compute
\ba
\lambda:=\IE[W]=\sum_{i=1}^n (1-p_i)^k
\ee
and let $I$ be a random variable such that $\IP(I=i)=(1-p_i)^k/\lambda$.  Corollary \ref{cor121} now states that in order 
to construct $W^s$, we empty box $I$ (forcing $X_I^s=1$) and then redistribute the balls that were removed into the remaining
boxes independently and with chance of landing in box $j$ equal to $p_j/(1-p_I)$.  If we let $X_j^{(I)}$ be the indicator that box $j$ is empty  
after this procedure, then it is clear that $(X_j^{(I)})_{j\not=I}$ has the distribution of $(X_j)_{j\not=I}$ conditional on $X_I=1$
so that we can take $W^s=\sum_{j\not=I}X_j^{(I)} +1$.  

Analogous to Example \ref{ex123}, $X_j^{(i)}\leq X_j$ for  $j\not=i$ so that we can apply Theorem \ref{thm125}.   The mean and variance
are easily computed, but not as easily interpreted.  A little analysis (see \cite{bhj92} Section 6.2) yields that for $Z\sim\Po(\lambda)$,
\ba
\dtv(W,Z)\leq \min\{1, \lambda\}\left[\max_{i}(1-p_i)^k+\frac{k}{\lambda}\left(\frac{\lambda \log(k)}{k-\log(\lambda)}+\frac{4}{k}\right)^2\right].
\ee  

\begin{remark}
We could also consider the number of boxes with at most $m\geq 0$ balls; we just studied the case $m=0$.  The coupling is still
decreasing because it can be constructed by first choosing a box randomly and if it has greater than $m$ balls in it, redistributing
a random
number of them among the other boxes.  Thus Theorem \ref{thm125} can be applied and an error in the Poisson approximation can 
be obtained in terms of the mean and the variance.  We again refer to Chapter 6 of \cite{bhj92} for the details.  

Finally, could also consider the number of boxes containing exactly $m\geq0$ balls and containing at least $m\geq1$ balls, Theorem \ref{thm122} applies to the latter problem.  See Chapter 6 of \cite{bhj92}.
\end{remark}
\end{example}

\subsection{Exchangeable Pairs}

In this section, we will develop Stein's method for Poisson approximation using exchangeable pairs as
detailed in \cite{cdm05}.  
The applications for this theory are not as developed as that of dependency neighborhoods and size-biasing, but
the method fits well into our framework and the ideas here prove useful elsewhere \cite{rr10}.

As we have done for dependency neighborhoods and size-biasing, we could develop exchangeable pairs
for Poisson approximation by following our development for normal approximation 
which involved rewriting $\IE[Wf(W)]$ as the expectation of a term involving the exchangeable pair and $f$. 
However, this approach is not as useful as a different one which has the added advantage of
removing the $a$-Stein pair linearity condition.  
\begin{theorem}\label{thm161}
Let $W$ be a non-negative integer valued random variable and let $(W, W')$ be an 
exchangeable pair.   If $\%F$ is a sigma-field with $\sigma(W)\subseteq\%F$, and $Z\sim\Po(\lambda)$, then
for all $c\in\IR$, 
\ba
\dtv(W, Z)\leq \min\{1, \lambda^{-1/2}\}\bigg(\IE\big|\lambda-c\IP(W'=W+1|\%F)\big|
+\IE\big|W-c\IP(W'=W-1|\%F)\big|\bigg).
\ee
\end{theorem}
Before the proof, a few remarks.
\begin{remark}
Typically $c$ is chosen to be approximately equal
to $\lambda/\IP(W'=W+1)=\lambda/\IP(W'=W-1)$ so that the 
terms in absolute value have a small mean.
\end{remark}
\begin{remark}
Similar to exchangeable pairs for normal approximation, there is a stochastic interpretation for the 
terms appearing in the error of Theorem \ref{thm161}.  We can define a birth-death
process on $\IN\cup\{0\}$ where the birth rate at state $k$ is $\alpha(k)=\lambda$ and death rate 
at state $k$ is $\beta(k)=k$.  This birth-death process has a $\Po(\lambda)$ stationary distribution
so that the theorem says that if there is a reversible Markov chain with stationary distribution
equal to the distribution of $W$ such that the chance of increasing by one is approximately 
proportional to some constant $\lambda$ and the chance
of decreasing by one is approximately proportional to the current state, then $W$ will be approximately Poisson.
\end{remark}
\begin{proof}
As usual, we want to bound $\abs{\IE[\lambda f(W+1)-Wf(W)]}$ for functions $f$ such that
$\norm{f}\leq \min\{1, \lambda^{-1/2}\}$ and $\norm{\Delta f}\leq \{1, \lambda^{-1}\}$.  Now, the
function 
\ba
F(w, w')=\I[w'=w+1]f(w')-\I[w'=w-1]f(w),
\ee
is anti-symmetric, so that $\IE[F(W,W')]=0$.  Moreover, by conditioning on $\%F$ we obtain
that for all $c\in\IR$,
\ba
c\IE\left[\IP(W'=W+1|\%F)f(W+1)-\IP(W'=W-1|\%F)f(W)\right]=0.
\ee
which implies that
\ban
&\IE[\lambda f(W+1)-Wf(W)]= \notag \\
&\qquad\IE\left[\left(\lambda-c\IP(W'=W+1|\%F)\right)f(W+1)
-\left(W-c\IP(W'=W-1|\%F)\right)f(W)\right]. \label{161}
\ee
Taking the absolute value and applying the triangle inequality yields the theorem.
\end{proof}
\begin{example}[Law of small numbers]
Let $W=\sum_{i=1}^n X_i$ where the $X_i$ are independent indicators with $\IP(X_i=1)=p_i$ and define $W'=W-X_I+X_I'$,
where $I$ is uniform on $\{1, \ldots, n\}$ independent of $W$ and $X_1', \ldots, X_n'$ are independent copies of
the $X_i$ independent of each other and all else.  It is easy to see that $(W, W')$ is an exchangeable pair and 
that 
\ba
&\IP(W'=W+1|(X_i)_{i\geq1})=\frac{1}{n}\sum_{i=1}^n(1-X_i)p_i,  \\
&\IP(W'=W-1|(X_i)_{i\geq1})=\frac{1}{n}\sum_{i=1}^nX_i(1-p_i),
\ee
so that Theorem \ref{thm161} with $c=n$ yields
\ba
\dtv(W, Z)&\leq \min\{1, \lambda^{-1/2}\}\bigg(\IE\left|\sum_{i=1}^np_i-\sum_{i=1}^n(1-X_i)p_i\right|
+\IE\left|\sum_{i=1}^nX_i-\sum_{i=1}^nX_i(1-p_i)\right|\bigg) \\
&=2\min\{1, \lambda^{-1/2}\}\IE\left[\sum_{i=1}^np_iX_i\right]=2\min\{1, \lambda^{-1/2}\}\sum_{i=1}^np_i^2.
\ee
This bound is not as good as that obtained using size-biasing (for example), but the better result can be 
recovered with exchangeable pairs by bounding the absolute value of \eq{161} directly.
\end{example}
\begin{example}[Fixed points of permutations]
Let $\pi$ be a permutation of $\{1, \ldots, n\}$ chosen uniformly at random, let $\tau$ be a uniformly chosen random transposition,
and let $\pi'=\pi\tau$.  Let $W$ be the number of fixed points of $\pi$ and $W'$ be the number of fixed points of $\pi'$.  Then $(W, W')$
is exchangeable and if $W_2$ is the number of transpositions when $\pi$ is written as a product of disjoint cycles, then
\ba
&\IP(W'=W+1|\pi)=\frac{n-W-2W_2}{\binom{n}{2}},  \\
&\IP(W'=W-1|\pi)=\frac{W(n-W)}{\binom{n}{2}}.
\ee
To see why these expressions are true, note that in order for the number of fixed points of a permutation to increase
by exactly one after multiplication by a transposition, a letter must be fixed that is not already and is not in a transposition (else
the number of fixed points would increase by two).  Similar considerations lead to the second expression.

By considering $W$ as a sum of indicators, it is easy to see that $\IE[W]=1$ so that
applying Theorem \ref{thm161} with $c=(n-1)/2$ yields
\ba
\dtv(W, Z)&\leq \IE\left|1-\frac{n-W-2W_2}{n}\right|
+\IE\left|W-\frac{W(n-W)}{n}\right| \\
&=\frac{1}{n}\IE\left[W+2W_2\right]+\frac{1}{n}\IE[W^2]\\
&=4/n,
\ee
where the final inequality follows by considering $W$ and $W_2$ as a sum of indicators which leads
to $\IE[W^2]=2$ and $\IE[W_2]=1/2$.  As is well known, the true rate of convergence is much better than order $1/n$;
it is not clear how to get a better rate with this method.

We could also handle the number of $i$-cycles in a random permutation, but the analysis is a bit more tedious
and is not worth pursuing due to the fact that so much more is known in this example - see \cite{abt03} for a thorough account.

\end{example}

\section{Exponential approximation}\label{exp}

In this section we will 
develop Stein's method for bounding the Wasserstein distance (see Section \ref{probmet}) 
between a distribution of interest and the Exponential distribution.
We will move quickly through the material analogous to that of Section \ref{stnorm} 
for normal approximation, as the general framework is similar.
Our treatment follows \cite{rope10} closely; alternative approaches can be found in \cite{cfr06}.

\begin{definition}
We say that a random variable $Z$ has the exponential distribution with rate $\lambda$, denoted $Z\sim\Exp(\lambda)$
if $Z$ has density $\lambda e^{-\lambda z}$ for $z>0$.	
\end{definition}

\begin{lemma}\label{lem171}
Define the functional operator $\mathcal{A}$ by
\ba
\mathcal{A}f(x)=f'(x)-f(x)+f(0).
\ee
\begin{enumerate}
\item If $Z\sim\Exp(1)$, then
$\IE\%{A}f(Z)=0$ for all absolutely continuous $f$ with $\IE\abs{f'(Z)}<\infty$.
\item If for some non-negative random variable $W$, $\IE\%{A}f(W)=0$ for all absolutely continuous $f$
with $\IE\abs{f'(Z)}<\infty$, then $W\sim\Exp(1)$.
\end{enumerate}
The operator $\%A$ is referred to as a characterizing operator of the exponential distribution.
\end{lemma}
Before proving the lemma, we state one more result and then its consequence.
\begin{lemma}\label{lem172}
Let $Z\sim\Exp(1)$ and for some function $h$ let $f_h$ be the unique solution of
\ban
f_h'(x)-f_h(x)=h(x)-\IE[h(Z)] \label{171}
\ee
such that $f_h(0)=0$. 
\begin{enumerate}
\item\label{i171} If $h$ is non-negative and bounded, then 
\ba
\norm{f_h}\leq \norm{h} \mbox{\, and \,} \norm{f_h'}\leq 2 \norm{h}.
\ee 
\item\label{i172} If $h$ is absolutely continuous, then 
\ba
\norm{f_h'}\leq \norm{h'} \mbox{\, and \,}  \hspace{.5in} \norm{f_h''}\leq 2 \norm{h'}. 
\ee
\end{enumerate}
\end{lemma}
Analogous to normal approximation, this setup immediately yields the following promising result.
\begin{theorem}\label{thm171}
Let $W\geq0$ be a random variable with finite mean and $Z\sim\Exp(1)$.
\begin{enumerate}
\item If $\%F_W$ is the set of functions with $\norm{f'}\leq 1$, $\norm{f''}\leq 2$, 
and $f(0)=0$, then 
\ba
\dw(W, Z)\leq \sup_{f\in\%F_W}\left|\IE[f'(W)-f(W)]\right|.
\ee
\item If $\%F_K$ is the set of functions with $\norm{f}\leq 1$, $\norm{f'}\leq 2$, 
and $f(0)=0$, then 
\ba
\dk(W, Z)\leq \sup_{f\in\%F_K}\left|\IE[f'(W)-f(W)]\right|.
\ee
\end{enumerate}
\end{theorem}
\begin{proof}[Proof of Lemma \ref{lem172}]
Writing $\tilde{h}(t):=h(t)-\IE[h(Z)]$, the relation \eq{171} can easily be solved to yield
\ban
f_h(x)  
=-e^x\int_{x}^\infty\tilde{h}(t)e^{-t}dt. \label{172}
\ee
\begin{enumerate}
\item[\ref{i171}.] If $h$ is bounded, then \eq{172} implies that
\ba
\abs{f_h(x)}\leq e^x\int_x^{\infty} \abs{\tilde{h}(t)}e^{-t} dt \leq \norm{h}.
\ee
Since $f_h$ solves \eqref{171}, we have
\ba
\abs{f_h'(x)}=\abs{f_h(x)+\tilde{h}(x)}\leq \norm{f_h}+\norm{\tilde{h}}\leq 2\norm{h},
\ee
where we have used the bound on $\norm{f_h}$ above and that $h$ is non-negative.
\item[\ref{i172}.] If $h$ is absolutely continuous, then by the form of \eq{172} it is clear
that $f_h$ is twice differentiable.  Thus we have that $f_h$ satisfies
\ba
f_h''(x)-f_h'(x)=h'(x),
\ee
so that we can use the arguments of the proof of the previous item to establish the bounds
on $\norm{f_h'}$ and $\norm{f_h''}$.

\end{enumerate}
\end{proof}
\begin{proof}[Proof of Lemma \ref{lem171}]
Item 1 essentially follows by integration by parts.  More formally,
for $f$ absolutely continuous, we have
\ba
\IE[f'(Z)]&=\int_0^\infty f'(t) e^{-t} dt =\int_0^\infty f'(t) \int_t^\infty e^{-x} dx dt \\
	&=\int_0^\infty e^{-x} \int_0^x f'(t)dt dx =\IE[f(Z)]-f(0),
\ee
as desired.  For the second item, assume that $W\geq0$ satisfies
$\IE[f'(W)]=\IE[f(W)]-f(0)$ for all absolutely continuous $f$ with $\IE\abs{f'(Z)}<\infty$.
The functions $f(x)=x^k$ are in this family, so that 
\ba
k\IE[W^{k-1}]=\IE[W^k],
\ee
and this relation determines the moments of $W$ as those of an exponential distribution with rate one, which
satisfy Carleman's condition (using Stirling's approximation).
Alternatively, the hypothesis on $W$ also determines the Laplace transform as that of an exponential
variable.
\end{proof}

It is clear from the form of the error in Theorem \ref{thm171} that we want to 
find some structure in $W$, the random variable of interest, that allows us to compare
$\IE[f(W)]$ to $\IE[f'(W)]$ for appropriate $f$.  Unfortunately the tools we have previously developed
for the analogous task in Poisson and Normal approximation will not help us directly here.  
However, we will be able to define a transformation amenable to the form of the exponential
characterizing operator which will prove fruitful.  An alternative approach (followed in \cite{cfr06, chsh11}) is to use exchangeable
pairs with a modified $a$-Stein condition.

\subsection{Equilibrium coupling}
We begin with a definition.
\begin{definition}
Let $W\geq0$  a random variable with $\IE[W]=\mu$.  We say that $W^e$ has the \emph{equilibrium} distribution
with respect to $W$ if 
\ban
\IE[f(W)]-f(0)=\mu\IE[f'(W^e)] \label{173}
\ee
for all Lipschitz functions $f$.  
\end{definition}
We will see below that the equilibrium distribution is absolutely continuous with respect to Lebesgue measure, so that the
right hand side of \eq{173} is well defined.  Before this discussion, we note
the following consequence of this definition.
\begin{theorem}\label{thm172}
Let $W\geq0$ a random variable with $\IE[W]=1$ and $\IE[W^2]<\infty$.  If $W^e$ has the equilibrium distribution with
respect to $W$ and is coupled to $W$, then
\ba
\dw(W,Z)\leq 2\IE\abs{W^e-W}.
\ee
\end{theorem}
\begin{proof}
Note that if $f(0)=0$ and $\IE[W]=1$, then 
$\IE[f(W)]=\IE[f'(W^e)]$ so that for $f$ with bounded first and second derivative and such that $f(0)=0$, we have
\ba
\abs{\IE[f'(W)-f(W)]}=\abs{\IE[f'(W)-f'(W^e)]}\leq\norm{f''}\IE\abs{W^e-W}.
\ee
Applying Theorem \ref{thm171} now proves the desired conclusion.
\end{proof}
We now state a constructive definition of the equilibrium distribution which will also be useful later.
\begin{proposition}\label{prop171}
Let $W\geq0$ be a random variable with $\IE[W]=\mu$ and let $W^s$ have the size-bias distribution
of $W$.  If $U$ is uniform on the interval $(0,1)$ and independent of $W^s$, then $UW^s$ has the equilibrium
distribution of $W$. 
\end{proposition} 
\begin{proof}
Let $f$ be Lipschitz with $f(0)=0$.  Then
\ba
\IE[f'(UW^s)]=\IE\left[\int_{0}^1f'(uW^s)du\right]=\IE\left[\frac{f(W^s)}{W^s}\right]=\mu^{-1}\IE[f(W)],
\ee
where in the final equality we use the definition of the size-bias distribution.
\end{proof}
\begin{remark}\label{rem171}
This proposition shows that the equilibrium distribution is the same as that from renewal theory.  That
is, a renewal process in stationary with increments distributed as a random variable $Y$ is given by
$Y^e+Y_1+\cdots+Y_n$, where $Y^e$ has the equilibrium distribution of  $Y$ and and is independent of
the i.i.d. sequence $Y_1, Y_2, \ldots$
\end{remark}

We will use Theorem \ref{thm172} to treat some less trivial applications shortly, but first we handle a canonical exponential approximation result.
\begin{example}[Geometric distribution]\label{ex171}
Let $N$ be geometric with parameter $p$ with positive support (denoted $N\sim\Ge(p)$), specifically, $\IP(N=k)=(1-p)^{k-1}p$ for $k\geq 1$.
It is well known that as $p\to0$, $pN$ converges weakly to an exponential distribution; this fact is not surprising as a simple
calculation shows that if $Z\sim\Exp(\lambda)$, then the smallest integer no greater than $Z$ is geometrically distributed.  We can use
Theorem \ref{thm172} above to obtain an error in this approximation.

A little calculation shows that $N$ has the same distribution as a variable which is uniform on $\{1, \ldots, N^s\}$,
where $N^s$ has the size-bias distribution of $N$ 
(heuristically this is due to the memoryless property of the geometric distribution - see Remark \ref{rem171}).  Thus
Proposition \ref{prop171} implies that for $U$ uniform on $(0,1)$ independent of $N$, $N-U$ has the equilibrium distribution of $N$. 

It is easy to verify that for a constant $c$ and a non-negative variable $X$, $(cX)^e\ed cX^e$, so that if we define $W=pN$ 
our remarks above imply that $W^e:=W-pU$ has the equilibrium distribution with respect to $W$.
We now apply Theorem \ref{thm172} to find that for $Z\sim\Exp(1)$, 
\ba
\dw(W, Z)\leq 2\IE[pU]=p.
\ee
\end{example}

\subsection{Application: Geometric sums}
Our first application is a generalization of the following classical result which in turn generalizes Example \ref{ex171}.
\begin{theorem}[R\'enyi's Theorem]
Let $X_1, X_2, \ldots$ be an i.i.d. sequence of non-negative random variables with $\IE[X_i]=1$ and let $N\sim\Ge(p)$ independent
of the $X_i$.  If $W=p\sum_{i=1}^N X_i$ and $Z\sim\Exp(1)$, then
\ba
\lim_{p\to 0} \dk(W, Z)=0.
\ee 
\end{theorem}
The special case where $X_i\equiv1$ is handled in Example \ref{ex171}; intuitively, the example can be generalized because 
for $p$ small, $N$ is large so that the law of large numbers implies that $\sum_{i=1}^N X_i$ is approximately equal to $N$.
We will show the following result which implies R\'enyi's Theorem.
\begin{theorem}\label{thm1710}
Let $X_1, X_2, \ldots$ be square integrable, non-negative, and independent random variables with $\IE[X_i]=1$.
Let $N>0$ be an integer valued random variable with $\IE[N]=1/p$ for
some $0<p\leq 1$ and let $M$ be defined on the same space as $N$ such that
\ba
\IP(M=m)=p \IP(N\geq m).
\ee
If $W=p\sum_{i=1}^N X_i$, $Z\sim\Exp(1)$, and $X_i^e$ is an equilibrium coupling of $X_i$ independent of $N, M$, and $(X_j)_{j\not=i}$,
then 
\ban
\dw(W, Z)&\leq 2p\left(\IE\abs{X_M-X_M^e}+\IE\abs{N-M}\right) \label{177} \\
	&\leq 2p\left(1+\frac{\mu_2}{2}+\IE\abs{N-M}\right), \label{178}
\ee
where $\mu_2:=\sup_{i} \IE[X_i^2]$.
\end{theorem}
Before the proof of the theorem, we make a few remarks.
\begin{remark}
The theorem can be a little difficult to parse on first glance, so we make a few comments to interpret the error. 
The random variable $M$ is a discrete version of the equilibrium transform which we have already seen above in Example \ref{ex171}.
More specifically, it is easy to verify that
if $N^s$ has the size-bias distribution of $N$, then $M$ is distributed uniformly on the set $\{1, \ldots, N^s\}$.  If $N\sim\Ge(p)$,
then we can take $M\equiv N$ so that the final term of the error of \eq{177} and \eq{178}
is zero.  Thus $\IE\abs{N-M}$ quantifies the proximity of the distribution of $N$ to
a geometric distribution.  
We will formalize this more precisely when we cover Stein's method for geometric approximation below.

The first term of the error in \eq{178} can be interpreted as a term measuring the regularity of the $X_i$. 
The heuristic is that the law of large numbers needs to kick in so that $\sum_{i=1}^NX_i\approx N$,
and the theorem shows that in order for this to occur, it is enough that the $X_i$'s have uniformly bounded variances.  Moreover, the first term 
of the error \eq{177} shows that the approximation also benefits from having the $X_i$ be close to exponentially distributed.  In particular,
if all of the $X_i$ are exponential and $N$ is geometric, then the theorem shows $\dw(W, Z)=0$ which can also be
easily verified using Laplace transforms.
\end{remark}
\begin{remark}
A more general theorem can be proved with a little added technicality which allows for the $X_i$ to have different means and 
for the $X_i$ to have a certain dependence - see~\cite{rope10}.
\end{remark}
\begin{proof}
We will show that 
\ban
W^e=p\left[\sum_{i=1}^{M-1}X_i+X_M^e\right]  \label{179}
\ee
is an equilibrium coupling of $W$.  From this point Theorem \ref{thm172} implies that
\ba
\dw(W, Z)&\leq 2\IE\abs{W^e-W} \\
	&= 2p \IE\left|X_M^e-X_M+\sgn(M-N)\sum_{i=M\wedge N+1}^{M \vee N} X_i\right|\\
	&\leq 2p \IE\left[\abs{X_M^e-X_M}+\abs{N-M}\right],	
\ee
which proves \eq{177}.  The second bound \eq{178} follows from \eq{177} after noting that
\ba
\IE[\abs{X_M^e-X_M}\big|M]&\leq \IE[X_M^e|M]+\IE[X_M|M] \\
	&=\frac{1}{2}\IE[X_M^2|M]+1\leq \frac{\mu_2}{2}+1,
\ee
where the equality follows from the definition of the equilibrium coupling.  

It only remains to show \eq{179}.  Let $f$ be Lipschitz with $f(0)=0$ and define
\ba
g(m)=f\left(p\sum_{i=1}^m X_i\right).
\ee
On one hand, using independence and the defining relation of $X_m^e$, we obtain
\ba
\IE\left[f'\left(p\sum_{i=1}^{M-1} X_i + pX_M^e\right)\bigg|M\right]=p^{-1}\IE[g(M)-g(M-1)|M],
\ee 
and on the other, the definition of $M$ implies
\ba
p^{-1}\IE[g(M)-g(M-1)|(X_i)_{i\geq1}]=\IE[g(N)|(X_i)_{i\geq1}],
\ee
so that altogether we obtain $\IE[f'(W^e)]=\IE[g(N)]=\IE[f(W)]$,
as desired.
\end{proof}
\subsection{Application: Critical branching process}
In this section we obtain an error in a classical theorem of Yaglom pertaining to the
generation size of a critical Galton-Watson branching process conditioned on non-extinction.  
We will attempt to have this section be self-contained, but it will be helpful
to have been exposed to the elementary properties and definitions of
branching processes, found for example in the first chapter of~\cite{atne72}.

Let $Z_0=1$, and $Z_1$ be a non-negative integer valued random variable with finite mean.  For $i, j\geq 1$, let 
$Z_{i, j}$ be i.i.d. copies of $Z_1$ and define for $n\geq1$
\ba
Z_{n+1}=\sum_{i=1}^{Z_n}Z_{n, i}.
\ee
We think of of $Z_n$ as the generation size of a population that initially has one individual and where 
each individual in a generation has a $Z_1$ distributed number of offspring (or children) independently of
the other individuals in the generation.  We also assume that all individuals in a generation have offspring
at the same time (creating the next generation) and die immediately after reproducing.

It is a basic fact \cite{atne72} that if $\IE[Z_1]\leq 1$ and $\IP(Z_1=1<1)$ then the population almost surely dies out, whereas 
if $\IE[Z_1]>1$, then the probability the population lives forever is strictly positive.
Thus the case where $\IE[Z_1]=1$ is referred to as the critical case and
a fundamental result of the behavior in this case is the following.
\begin{theorem}[Yaglom's Theorem]
Let $1=Z_0, Z_1,\ldots$ be the generation sizes of a Galton-Watson branching process where
$\IE[Z_1]=1$ and $\var(Z_1)=\sigma^2<\infty$.  If $Y_n\ed (Z_n|Z_n>0)$ and $Z\sim \Exp(1)$, then
\ba
\lim_{n\to\infty}\dk\left(\frac{2Y_n}{n\sigma^2}, Z\right)=0.
\ee
\end{theorem}
We will provide a rate of convergence in this theorem under a stricter moment assumption.
\begin{theorem}
Let $Z_0, Z_1, \ldots$ as in Yaglom's Theorem above and assume also that $\IE\abs{Z_1}^3<\infty$.  If $Y_n\ed (Z_n|Z_n>0)$ and $Z\sim \Exp(1)$, then for some constant $C$, 
\ba
\dw\left(\frac{2Y_n}{n\sigma^2}, Z\right)\leq C\frac{\log(n)}{n}.
\ee
\end{theorem}
\begin{proof}
We will construct a copy of $Y_n$ and $Y_n^e$ having the equilibrium distribution on the same space and then show that
$\IE\abs{Y_n-Y_n^e}\leq C \log(n)$.  Once this is established, the result is proved 
by Theorem \ref{thm172} and the fact that $(cY_n)^e\ed cY_n^e$ for any constant $c$.

In order to couple $Y_n$ and $Y_n^e$, we will construct a ``size-bias" tree and then find copies of the variables we need
in it.  The clever construction we will use is due to \cite{lpp95} and implicit in their work is the fact that $\IE\abs{Y_n-Y_n^e}/n\to0$
(used to show that $n\IP(Z_n>0)\to2/\sigma^2$), but we will weed out a rate from their analysis. 

We view the size-bias tree
as labeled and ordered, in the sense that, if $w$ and $v$
are vertices in the tree from the same generation and $w$ is to the left of $v$,
then the offspring of $w$ is to the left of the offspring of $v$. Start in
generation $0$ with one vertex $v_0$ and let it have a number of offspring
distributed according to the size-bias distribution of $Z_1$. Pick one of
the
offspring of $v_0$ uniformly at random and call it $v_1$. To each of the
siblings of $v_1$ attach an independent Galton-Watson branching process having the
offspring distribution of $Z_1$. For $v_1$ proceed as for $v_0$, i.e., give
it a
size-bias number of offspring, pick one uniformly at random, call it $v_2$,
attach independent Galton-Watson branching process to the siblings of $v_2$ and
so on. It is clear that this process will always give an infinite tree as the ``spine''
$v_0,v_1,v_2,\dots$ will be infinite.  See Figure 1 of \cite{lpp95}
for an illustration of this tree.

Now, for a fixed tree $t$, let $G_n(t)$ be the chance that the original branching process 
driven by $Z_1$ agrees with $t$ up to generation $n$, let $G_n^s(t)$ be the chance that
the size-bias tree just described agrees with $t$ up to generation $n$, and for $v$ 
an individual of $t$ in generation $n$, let
$G_n^s(t, v)$ be the chance that the size-bias tree agrees with $t$ up to generation $n$
and has the vertex $v$ as the distinguished vertex $v_n$ in generation $n$.
We claim that
\ban
G_n^s(t, v)=G_n(t). \label{1710}
\ee 
Before proving this claim, we note some immediate consequences which imply
that our size-bias tree naturally contains a copy of $Y_n^e$.
Let $S_n$ be the size of generation $n$ in the size-bias tree. 
\begin{enumerate}
\item\label{i171a} $S_n$ has the size-bias distribution of $Z_n$.
\item\label{i172a} If $Y_n^s$ has the size-bias distribution of $Y_n$, then
$S_n\ed Y_n^s$.
\item\label{i173} Given $S_n$, $v_n$ is uniformly distributed among the individuals of generation $n$.
\item\label{i174} If $R_n$ is the number of individuals to the right (inclusive) of $v_n$ and
$U$ is uniform on $(0,1)$, independent of all else, then $Y_n^e:=R_n-U$ has the equilibrium
distribution of $Y_n$.
\end{enumerate}
To show the first item, note that \eq{1710} implies  
\ban
G_n^s(t)=t_n G_n(t), \label{1711}
\ee
 where $t_n$ is the number 
of individuals in the $n$th generation of $t$.  Now, $\IP(S_n=k)$ is obtained by integrating
the left hand side of \eq{1711} over trees $t$ with $t_n=k$, and performing the same
integral on the right hand side of \eq{1711} yields $k\IP(Z_n=k)$.  The second
item follows from the more general fact that conditioning a  non-negative random variable to be positive
does not change the size-bias distribution.  Item \ref{i173} can be read from the
right hand side of \eq{1710}, since it does not depend on $v$.  For Item \ref{i174},
Item \ref{i173} implies that
$R_n$ is uniform on $\{1, \ldots, S_n\}$ so that $R_n-U\ed US_n$, from which the result
follows from Item \ref{i172a} and Proposition \ref{prop171}.

At this point, we would like to find a copy of $Y_n$ in the size-bias tree, but before proceeding
further we prove \eq{1710}.  Since trees formed below distinct vertices in a given generation 
are independent, we will prove the formula by writing down a recursion.  To this end, for
a given planar tree $t$ with $k$ individuals in the first generation, label the subtrees with these
$k$ individuals as a root from left to right by $t_1, t_2, \ldots, t_k$.  Now, a vertex $v$ in 
generation $n+1$ of $t$ lies in exactly one of the subtrees $t_1, \ldots, t_k$, say $t_i$.  With this setup,
we have
\ba
G_{n+1}^s(t,v)=G_n(t_i, v)\prod_{j\not=i} G_n(t_j)[k\IP(Z_1=k)]\frac{1}{k}.
\ee 
The first factor corresponds to the chance of seeing the tree $t_i$ up to generation $n+1$ below the distinguished vertex $v_1$ and choosing $v$
as the distinguished vertex in generation $n+1$.  The second factor is the chance of seeing the remaining subtrees up 
to generation $n+1$, and the remaining factors correspond to having $k$ offspring initially (with the size-bias distribution of $Z_1$)
and choosing vertex $v_1$ (the root 
of $t_i$) as the distinguished vertex initially.  With this formula in hand, it is enough to verify that \eq{1710} follows this recursion.

We must now find a copy of $Y_n$ in our size-bias tree.  If $L_n$ is the number of individuals to the left of $v_n$ (exclusive, so
$S_n=L_n+R_n$), then we claim that
\ban
S_n\big|\{L_n=0\}\ed Y_n. \label{1712}
\ee
Indeed, we have
\ba
\IP(S_n=k|L_n=0)&=\frac{\IP(L_n=0|S_n=k)\IP(S_n=k)}{\IP(L_n=0)}  \\
	&=\frac{\IP(S_n=k)}{k\IP(L_n=0)}=\frac{\IP(Z_n=k)}{\IP(L_n=0)},
\ee
where we have used Items \ref{i172a} and \ref{i173} above and the claim now follows since
\ba
\IP(L_n=0)=\sum_{k\geq 1}\IP(L_n=0|S_n=k)\IP(S_n=k)=\sum_{k\geq1}\frac{\IP(S_n=k)}{k}=\IP(Z_n>0).
\ee

We are only part of the way to finding a copy of $Y_n$ in the size-bias tree since we still need 
to realize $S_n$ given the event $L_n=0$.  
Denote by $S_{n,j}$ the number of particles in generation $n$ that stem from any
of the siblings of $v_j$ (but not $v_j$ itself). Clearly, $S_n = 1 +
\sum_{j=1}^n S_{n,j}$, where the summands are independent. Likewise, let
$L_{n,j}$ and $R_{n,j}$, be the number of particles in generation
$n$ that stem from the siblings to the left and right of $v_j$ (exclusive)
and note that $L_{n,n}$ and $R_{n,n}$ are just the number of siblings of $v_n$ to
the left and to the right, respectively. We have the relations $L_n =
\sum_{j=1}^n L_{n,j}$ and $R_n = 1 + \sum_{j=1}^n R_{n,j}$. Note that for
fixed~$j$, $L_{n,j}$ and $R_{n,j}$ are in general not independent, as they are linked through
the offspring size of $v_{j-1}$.

Let now $R_{n,j}'$ be independent random variables such that
\be
    R'_{n,j} \ed R_{n,j} \big| \{ L_{n,j} = 0\}.
\ee
and
\ba                                                           
    R_{n,j}^* = R_{n,j} \I[L_{n,j}=0] + R_{n,j}' \I[L_{n,j}>0]
    = R_{n,j} + (R_{n,j}' - R_{n,j}) \I[L_{n,j}>0].
\ee
Finally, if $R_n^* = 1 + \sum_{j=1}^n R_{n,j}^*$, then \eq{1712} implies that 
we can take $Y_n:=R_n^*$.

Having coupled $Y_n$ and $Y_n^e$, we can now proceed to show 
$\IE\abs{Y_n^e-Y_n}=\bigo(\log(n))$.  By Item \ref{i174} above, 
\ba
\abs{Y_n-Y_n^e}&=\left|1-U+\sum_{j=1}^n(R_{n,j}' - R_{n,j}) \I[L_{n,j}>0]\right| \\
	&\leq \abs{1-U}+\sum_{j=1}^nR_{n,j}'\I[L_{n,j}>0]+\sum_{j=1}^nR_{n,j}\I[L_{n,j}>0].
\ee
Taking expectation in the inequality above, our result will follow after we show that
\ba
    (i) &\enskip \IE \left[R_{n,j}'\I[L_{n,j}>0]\right]\leq \sigma^2\IP(L_{n,j}>0),\\
    (ii) &\enskip \text{$\IE\left[R_{n,j} \I[L_{n,j}>0]\right] \leq \IE[Z_1^3]
            \IP(Z_{n-j}>0),$} \\
    (iii) &\enskip \text{$\IP(L_{n,j}>0)\leq \sigma^2\IP(Z_{n-j}>0)\leq C (n-j+1)^{-1}$ for some $C>0$.}
\ee
For part (i), independence implies that
\ba
\IE \left[R_{n,j}'\I[L_{n,j}>0]\right]=\IE[R_{n,j}']\IP(L_{n,j}>0),
\ee
and using that $S_{n, j}$ and $\I[L_{n,j}=0]$ are negatively correlated (in the second inequality) below, we find
\ba
\IE[R_{n,j}']&=\IE[R_{n,j}|L_{n,j}=0]\\
&\leq \IE[S_{n,j}-1|L_{n,j}=0] \\
	&\leq \IE[S_{n,j}]-1 \\
	&\leq \IE[S_n]-1=\sigma^2.
\ee

For part (ii), if
$X_j$ denotes the number of
siblings of $v_j$, having the size-bias distribution of $Z_1$ minus $1$, we have
\ba
   \IE\left[R_{n,j} \I[L_{n,j}>0]\right] &\leq   \IE[X_j \I[L_{n,j}>0]]  \\
   &\leq \sum_{k}k\IP(X_j=k,L_{n,j}>0)\\ 
     &\leq \sum_{k}k\IP(X_j=k)\IP(L_{n,j}>0|X_j=k) \\
     &\leq \sum_{k}k^2\IP(X_j=k)\IP(Z_{n-j}>0) \\
    &\leq \IE[Z_1^3]\IP(Z_{n-j}>0),
\ee
where we have used that $\IE[R_{n,j}\I[L_{n,j}>0]|X_j]\leq X_j \I[L_{n,j}>0]$ in the first inequality
and that $\IP(L_{n,j}>0|X_j=k)\leq k\IP(Z_{n-j}>0)$ in the penultimate inequality.

Finally, we have
\ba
    \IP(L_{n,j}>0) &= \IE\left[\IP(L_{n,j}>0|X_j)\right] \\
    &\leq \IE\left[X_j\IP(Z_{n-j}>0\right] \\
    &\leq \sigma^2\IP(Z_{n-j}>0).
\ee
Using Kolmogorov's estimate (see Chapter 1, Section 9 of \cite{atne72}),
we have $\lim_{n\to\infty}n\IP(Z_n>0) = 2/\sigma^2$, which
implies the final statement of (iii).
\end{proof}

\section{Geometric approximation}\label{geo}

Due to Example \ref{ex171}, if $W>0$ is integer-valued such that $W/\IE[W]$ is approximately exponential and $\IE[W]$ is large, then 
we expect that $W$ will be approximately geometrically distributed.  In fact, if we write $\IE[W]=1/p$, and let $X\sim\Ge(p)$ and 
$Z\sim\Exp(1)$, then the triangle inequality implies that
\ba
\left|\dw(pW, Z)-\dw(X, W)\right|\leq p.
\ee
However, if we want to bound $\dtv(W, X)$, then the inequality above is not useful.  For example, if $W\ed kX$ for some positive integer $k$,
then $\dtv(W, X)\approx (k-1)/k$ since the support of $W$ and $X$ do not match.  
This
issue of support mismatch is typical in bounding the total variation distance between integer-valued random variables
and can be handled by introducing a term into the bound that quantifies the ``smoothness"
of the random variable of interest.  

The version of Stein's method for geometric approximation which we will discuss below
can be used to handle these types of technicalities \cite{prr10}, but the
arguments can be a bit tedious.  Thus, we will 
develop
a simplified version of the method and apply it to an example where these technicalities
do not arise and where exponential approximation
does not hold.

We parallel the development of Stein's method for exponential approximation above, so
we will move quickly through the initial theoretical framework;
our work below follows \cite{prr10}.

\subsection{Main theorem}
A typical issue when discussing the geometric distribution is whether
to have the support begin at zero or one.  
In our work below we will focus on the geometric distribution which puts
mass at zero; that is $N\sim\Geo(p)$ if for $k=0, 1, \ldots$, we have $\IP(N=k)=(1-p)^k p$ .
Developing the theory below for the geometric distribution with positive
support is similar in flavor, but different in detail - see \cite{prr10}.

As usual we begin by defining the characterizing operator that we will use.
\begin{lemma}\label{lem173}
Define the functional operator $\mathcal{A}$ by
\ba
\mathcal{A}f(k)=(1-p)\Delta f(k)-p f(k)+ p f(0).
\ee
\begin{enumerate}
\item If $Z\sim\Geo(p)$, then
$\IE\%{A}f(Z)=0$ for all bounded $f$.
\item If for some non-negative random variable $W$, $\IE\%{A}f(W)=0$ for all bounded $f$,
then $W\sim\Geo(p)$.
\end{enumerate}
The operator $\%A$ is referred to as a characterizing operator of the geometric distribution.
\end{lemma}
We now state the properties of the solution to the Stein equation that we need.
\begin{lemma}\label{lem174}
If $Z\sim\Geo(p)$, $A\subseteq \IN\cup\{0\}$, and $f_A$ is the unique solution with $f_A(0)=0$ of
\ba
(1-p)\Delta f_A(k)-p f_A(k)=\I[k\in A]-\IP(Z\in A),
\ee
then $-1\leq \Delta f(k) \leq 1$.
\end{lemma}
These two lemmas lead easily to the following result.
\begin{theorem}\label{thm178}
Let $\%F$ be the set of functions with $f(0)=0$ and $\norm{\Delta f}\leq 1$ and let
$W\geq 0$ be an integer-valued random variable with $\IE[W]=(1-p)/p$ for some $0<p\leq 1$.  If $N\sim\Geo(p)$, 
then
\ba
\dtv(W, N)\leq \sup_{f \in \%F}\left|\IE[(1-p)\Delta f(W)-p f(W)]\right|.
\ee
\end{theorem}
Before proceeding further, we briefly indicate the proofs of Lemmas \ref{lem173} and \ref{lem174}.
\begin{proof}[Proof of Lemma \ref{lem173}]
The first assertion is a simple computation while the second can be verified by choosing
$f(k)=\I[k=j]$ for each $j=0, 1, \ldots$ which yields a recursion for the point probabilities for $W$.
\end{proof}
\begin{proof}[Proof of Lemma \ref{lem174}]
After noting that
\ba			
f_A(k)=\sum_{i\in A }(1-p)^{i} - \sum_{i\in A, i\geq k} (1-p)^{i-k},
\ee
we easily see
\ba
\Delta f_A(k)=\I[k\in A]-p\sum_{i\in A, i\geq k+1} (1-p)^{i-k-1},
\ee
which is the difference of two non-negative terms, each of which is bounded above by one.
\end{proof}

It is clear from the form of the error of Theorem \ref{thm178} that 
it may be fruitful to attempt to define a discrete version of the equilibrium distribution used
in the exponential approximation formulation above, which is the program we will follow.
An alternative coupling is used in~\cite{pek96}.

\subsection{Discrete equilibrium coupling}
We begin with a definition.
\begin{definition}
Let $W\geq0$  a random variable with $\IE[W]=(1-p)/p$ for some $0<p\leq 1$.  We say that $W^e$ has the \emph{discrete equilibrium} distribution
with respect to $W$ if for all functions $f$ with $\norm{\Delta f}<\infty$,
\ben
p\IE[f(W)]-pf(0)=(1-p)\IE[\Delta f(W^e)].    \label {1717}
\ee
\end{definition}
The following result provides a constructive definition of the discrete equilibrium distribution, so that the
right hand side of \eq{1717} defines a probability distribution.
\begin{proposition}\label{prop172}
Let $W\geq0$ be an integer-valued random variable with $\IE[W]=(1-p)/p$ for some $0<p\leq 1$ and let $W^s$ have the size-bias distribution
of $W$.  If conditional on $W^s$, $W^e$ is uniform on $\{0, 1, \ldots, W^s-1\}$, then $W^e$
has the discrete equilibrium distribution with respect to $W$.
\end{proposition} 
\begin{proof}
Let $f$ be such that $\norm{\Delta f}<\infty$ and $f(0)=0$.  If $W^e$ is uniform on $\{0, 1, \ldots, W^s-1\}$ as dictated by the proposition, then
\ba
\IE[\Delta f(W^e)]=\IE\left[\frac{1}{W^s}\sum_{i=0}^{W^s-1}\Delta f(i)\right]=\IE\left[\frac{f(W^s)}{W^s}\right]=\frac{p}{1-p}\IE[f(W)],
\ee
where in the final equality we use the definition of the size-bias distribution.
\end{proof}
\begin{remark}\label{rem175}
This proposition shows that the equilibrium distribution is the same as that from renewal theory - see Remark~\ref{rem171}.
\end{remark}

\begin{theorem}\label{thm177}
Let $N\sim\Geo(p)$ and $W\geq0$ an integer-valued random variable with 
$\IE[W]=(1-p)/p$ for some $0<p\leq 1$ such that $\IE[W^2]<\infty$.  If $W^e$ has the equilibrium distribution with
respect to $W$ and is coupled to $W$, then
\ba
\dtv(W,N)\leq 2(1-p)\IE\abs{W^e-W}.
\ee
\end{theorem}
\begin{proof}
If $f(0)=0$ and  $\norm{\Delta f}\leq 1$, then 
\ba
\big|\IE[(1-p)\Delta f(W)-p f(W)]\big|&=(1-p)\abs{\IE[\Delta f(W)-\Delta f(W^e)]} \\
	&\leq 2(1-p)\IE\abs{W^e-W},
	\ee
where the inequality follows after noting that $\Delta f(W)-\Delta f(W^e)$ can be written
as a sum of $\abs{W^e-W}$ terms each of size at most $\abs{\Delta f(W+i+1)- \Delta f(W+i)}\leq2\norm{\Delta f}$.
Applying Theorem \ref{thm178} now proves the desired conclusion. 
\end{proof}

\subsection{Application: Uniform attachment graph model}

Let $G_n$ be a directed random graph on $n$ nodes defined by the following recursive construction.
Initially the graph starts with one node with a single loop where one
end of the loop contributes to the ``in-degree" and the other to the ``out-degree."
Now, for $2\leq m\leq n$, given the graph with $m-1$ nodes, add node
$m$ along with an edge directed from $m$
to a node chosen uniformly at random among
the $m$ nodes present.  Note that this model allows edges connecting a node with
itself.  This random graph model is referred to as uniform attachment.
We will prove the following geometric approximation result (convergence was shown without rate in 
\cite{brst01}), which is weaker than the result of \cite{prr10} but has a slightly simpler proof.
 
\begin{theorem}\label{thm1716} If $W$ is the in-degree of a node chosen uniformly
at random from the random graph $G_n$ generated according to uniform attachment and $N\sim\Geo(1/2)$, then
\bes
    \dtv(W, N)\leq \frac{2(\log(n)+1)}{n}.
\ee
\end{theorem}

\begin{proof}
Let $X_i$ have a Bernoulli
distribution, independent of all else, with parameter $\mu_i := (n-i+1)^{-1}$,
and let $N$ be an independent random variable that is uniform on the integers
$1,2,\ldots, n$.  If we imagine that node $n+1-N$ is the randomly selected node,
then it's easy to see that we can write $W:=\sum_{i=1}^{N} X_i.$

Next, let us prove that $W^{e} := \sum_{i=1}^{N-1} X_i$ has the discrete
equilibrium distribution w.r.t.\ $W.$ First note
that we have for bounded $f$ and every $m$,
\be
    \mu_m \IE \D f\bbbklr{\,\sum_{i=1}^{m-1} X_i}
    =  \IE\bbbkle{f\bbbklr{ \,\sum_{i=1}^m X_i}
        -   f\bbbklr{\,\sum_{i=1}^{m-1} X_i}},
\ee
where we use
\be
  \IE f(X_m)-f(0)  = \IE X_m \, \IE \D f(0)
\ee
and thus the fact that we can write $X_m^{e} \equiv 0$. Note also that for
any bounded function $g$ with
$g(0)=0$ we have
\be
    \IE\bbbklg{\frac{g(N)}{\mu_N} - \frac{g(N-1)}{\mu_N}}
    =  \IE g(N).
\ee
We now assume that $f(0) = 0$. Hence, using the above two facts and
independence between $N$ and the sequence $X_1,X_2,\dots$, we have
\bes
   \IE W \IE \D f(W^{e_0})=\IE f(W).
\ee
Since $W-W^{e}=X_N,$ we have $\IE\abs{W^e-W} = \frac{1}{n}\sum_{i=1}^n (n-i+1)^{-1}$
and the result follows upon applying Theorem~\ref{thm177}.
\end{proof}

\section{Concentration Inequalities}\label{con}

The techniques we have developed for estimating expectations of characterizing operators (e.g. exchangeable pairs)
can also be used to prove concentration inequalities (or large deviations).  By concentration inequalities, we mean estimates of
$\IP(W\geq t)$ and $\IP(W\leq -t)$, for $t>0$ and some centered random variable $W$.  Of course our previous
work was concerned with such estimates, but here we are after the rate that these quantities
tend to zero as $t$ tends to infinity - in the tails of the distribution.
Distributional error terms are maximized in the body of the distribution and so typically do not
provide information about the tails.  

The study of concentration inequalities have a long history and have also found 
recent use in machine learning and analysis of algorithms - see \cite{bbl04} and
references therein for a flavor of the modern considerations of these types of problems.
Our results will hinge on the following fundamental observation. 

\begin{proposition}\label{prop211}
If $W$ is random variable and there is a $\delta>0$ such that $\IE[e^{\theta W}]<\infty$ for all $\theta\in (-\delta, \delta)$, then for all 
$t>0$ and $0<\theta<\delta$,
\ba
\IP(W\geq t)\leq \frac{ \IE[e^{\theta W}]}{e^{\theta t}} \mbox{\, and \,} \IP(W\leq-t)\leq \frac{ \IE[e^{-\theta W}]}{e^{\theta t}}.
\ee
\end{proposition}
\begin{proof}
Using first that $e^x$ is an increasing function, and then Markov's inequality,
\ba
\IP(W>t) =  \IP\left(e^{\theta W}>e^{\theta t}\right)\leq \frac{ \IE[e^{\theta W}]}{e^{\theta t}},
\ee
which proves the first assertion.  The second assertion follows similarly.
\end{proof}

Before discussing the use of Proposition \ref{prop211} in Stein's method, we first work out a couple of easy examples.

\begin{example}[Normal distribution]\label{ex211}
Let $Z$ have the standard normal distribution and recall
that for $t>0$ we have the Mills ratio bound 
\ba
\IP(Z\geq t)\leq \frac{e^{-t^2/2}}{t\sqrt{2\pi}}.
\ee
A simple calculation implies $\IE[e^{\theta Z}]=e^{\theta^2/2}$ for all $\theta \in \IR$,
so that for $\theta, t>0$ Proposition \ref{prop211} implies
\ba
\IP(Z\geq t)\leq  e^{\theta^2/2-\theta t},
\ee
and choosing the minimizer $\theta=t$ yields 
\ba
\IP(Z\geq t)\leq  e^{ -t^2/2},
\ee
which implies that this is the best behavior we can hope for using Proposition \ref{prop211} in examples
where the random variable is approximately normal (such as sums of independent variables).
\end{example}

\begin{example}[Poisson distribution]\label{ex212}
Let $Z$ have the Poisson distribution with mean $\lambda$.  
A simple calculation implies $\IE[e^{\theta Z}]=\exp\{\lambda(e^\theta-1)\}$ for all $\theta \in \IR$,
so that for $\theta, t>0$ Proposition \ref{prop211} implies
\ba
\IP(Z-\lambda\geq t)\leq  \exp\{\lambda(e^\theta-1)-\theta(t+\lambda)\},
\ee
and choosing the minimizer $\theta=\log(1+t/\lambda)$ yields 
\ba
\IP(Z-\lambda \geq t)\leq  \exp\left\{ -t\left(\log\left(1+\frac{t}{\lambda}\right)-1\right) -\lambda\log\left(1+\frac{t}{\lambda}\right) \right\},
\ee
which is of smaller order than $e^{-ct}$ for $t$ large and fixed $c>0$, but of bigger order than $e^{-t\log(t)}$.
This is the best behavior we can hope for using Proposition \ref{prop211} in examples
where the random variable is approximately Poisson (such as sums of independent indicators, each with a small probability of being one). 
\end{example}

How does Proposition \ref{prop211} help us use the techniques from Stein's method to obtain concentration inequalities?
If $W$ is random variable and there is a $\delta>0$ such that $\IE[e^{\theta W}]<\infty$ for all $\theta\in (-\delta, \delta)$,
then we can define $m(\theta)=\IE[e^{\theta W}]$ for $0<\theta<\delta$, and we also have that
$m'(\theta)=\IE[W e^{\theta W}]$.  Thus $m'(\theta)$ is of the form $\IE[Wf(W)]$, where $f(W)=e^{\theta W}$
so that we can 
use the techniques that we developed to bound the characterizing operator
for the normal and Poisson distribution to obtain a differential inequality for $m(\theta)$.  
Such an inequality will lead 
to bounds on $m(\theta)$ so that we can apply Proposition \ref{prop211} to obtain bounds on the tail probabilities
of $W$.  This observation was first made in \cite{cha05}.

\subsection{Concentration using exchangeable pairs}
Our first formulation using the couplings of Sections \ref{norm} and \ref{poi} for concentration inequalities uses exchangeable pairs.  We follow
the development of \cite{cha07a}.
\begin{theorem}\label{thm211a}
Let $(W, W')$ an $a$-Stein pair with $\var(W)=\sigma^2<\infty$.  If $\IE[e^{\theta W}\abs{W'-W}]<\infty$
for all $\theta\in \IR$ and for some sigma-algebra $\%F\supseteq\sigma(W)$ there are non-negative constants $B$ and $C$ such that 
\ban
\frac{\IE[(W'-W)^2|\%F]}{2a}\leq BW + C, \label{122g}
\ee
then for all $t>0$,
\ba
\IP(W\geq t)\leq \exp\left\{\frac{-t^2}{2C+2Bt}\right\} \mbox{\, and \,} \IP(W\leq-t)\leq \exp\left\{\frac{-t^2}{2C}\right\}.
\ee
\end{theorem}
\begin{remark}
The reason the left tail has a better bound stems from condition \eq{122g} which implies that
$BW+C\geq0$.  Thus, the condition essentially forces the centered variable $W$ to be bounded from below
whereas there is no such requirement for large positive values.  As can be understood from the proof of the theorem,
conditions other than \eq{122g} may be substituted to yield different bounds; see \cite{chde10}.
\end{remark}
Before proving the theorem, we apply it in a simple example.
\begin{example}[Sum of independent variables]
Let $X_1, \ldots, X_n$ be independent random variables with $\mu_i:=\IE[X_i]$, $\sigma_i^2:=\var(X_i)<\infty$
and define $W=\sum_i X_i-\mu_i$.
Let $X_1', \ldots, X_n'$ be an independent copy of the $X_i$ and for $I$ independent of all else and uniform on $\{1, \ldots, n\}$,
let $W'=W-X_I+X_I'$ so that as usual, $(W, W')$ is a $1/n$-Stein pair.  We consider two special cases of this setup.
\begin{enumerate}
\item For $i=1, \ldots, n$, assume $\abs{X_i-\mu_i}\leq C_i$.  Then clearly the moment generating function condition of
Theorem \ref{thm211a} is satisfied and we also have
\ba
\IE[(W'-W)^2 | (X_j)_{j\geq 1}]&=\frac{1}{n}\sum_{i=1}^n\IE[(X_i'-X_i)^2 | (X_j)_{j\geq 1}] \\
	&=\frac{1}{n}\sum_{i=1}^n\IE[(X_i'-\mu_i)^2]+\frac{1}{n}\sum_{i=1}^n (X_i-\mu_i)^2  \\
	&\leq \frac{1}{n}\sum_{i=1}^n(\sigma_i^2+C_i^2),
\ee
so that we can apply Theorem \ref{thm211a} with $B=0$ and $2C=\sum_{i=1}^n(\sigma_i^2+C_i^2)$.  We have shown
that for $t>0$,
\ba
\IP\left(\abs{W-\IE[W]}\geq t\right)\leq 2\exp\left\{-\frac{t^2}{\sum_{i=1}^n(\sigma_i^2+C_i^2)}\right\},
\ee
which is some version of Hoeffding's inequality \cite{hoe63}.
\item For $i=1, \ldots, n$, assume  $0\leq X_i \leq 1$. Then the moment generating function condition of the theorem is satisfied and
\ba
\IE[(W'-W)^2 | (X_j)_{j\geq 1}]&=\frac{1}{n}\sum_{i=1}^n\IE[(X_i'-X_i)^2 | (X_j)_{j\geq 1}] \\
	&=\frac{1}{n}\sum_{i=1}^n\IE[(X_i')^2]-2\mu_iX_i+X_i^2  \\
	&\leq \frac{1}{n}\sum_{i=1}^n\mu_i +X_i= \frac{1}{n}(2\mu+W),
\ee
where $\mu:=\IE[W]$ and we have used that $-2\mu_iX_i\leq 0$ and $X_i^2\leq X_i$.
We can now apply Theorem \ref{thm211a} with $B=1/2$ and $C=\mu$ to find that
for $t>0$,
\ban
\IP\left(\abs{W-\mu}\geq t\right)\leq 2\exp\left\{-\frac{t^2}{2\mu+t}\right\}. \label{21198}
\ee
Note that if $\mu$ is constant in $n$, then \eq{21198} is of the order $e^{-t}$ for large $t$, which according to Example \ref{ex212} is
similar to the order of Poisson tails.  However, if $\mu$ and $\sigma^2:=\var(W)$ are both going to infinity at the same rate,
then \eq{21198} implies
\ba
\IP\left(\frac{\abs{W-\mu}}{\sigma}\geq t\right)\leq 2\exp\left\{-\frac{t^2}{2\frac{\mu}{\sigma^2}+\frac{t}{\sigma}}\right\}, 
\ee
so that for $\sigma^2$ large, the tails are of order $e^{-ct^2}$, which according to Example \ref{ex211} is
similar to the order of Gaussian tails.
\end{enumerate} 
\end{example}

\begin{proof}[Proof of Theorem \ref{thm211a}]
Let $m(\theta)=\IE[e^{\theta W}]$ and note that $m'(\theta)=\IE[We^{\theta W}]$.  Since $(W,W')$ is
an $a$-Stein pair, we can use \eq{52} to find that for all $f$ such that $\IE\abs{Wf(W)}<\infty$,
\ba
\IE[(W'-W)(f(W')-f(W))]=2a \IE[Wf(W)].
\ee
In particular,
\ban
m'(\theta)=\frac{\IE[(W'-W)(e^{\theta W'}-e^{\theta W})]}{2a}, \label{2211a}
\ee
and in order to bound this term, we use the convexity of the exponential function to obtain for $x>y$
\ban
\frac{e^x-e^y}{x-y} = \int_0^1 \exp\{tx+(1-t)y\} dt \leq \int_0^1 te^x+(1-t)e^y dt = \frac{e^x+e^y}{2}. \label{2212a}
\ee
Combining \eq{2211a} and \eq{2212a}, we find that for all $\theta\in \IR$, 
\ban
\abs{m'(\theta)}&\leq \abs{\theta}\frac{\IE[(W'-W)^2(e^{\theta W'}+e^{\theta W})]}{4a} \notag \\
		&=\abs{\theta}\frac{\IE[(W'-W)^2 e^{\theta W}]}{2a} \notag\\
		&\leq \abs{\theta}\,\IE[(BW +C) e^{\theta W}] \notag\\
		&\leq B \abs{\theta} m'(\theta)+C \abs{\theta} m(\theta), \label{1222h}
\ee
where the equality is by exchangeability and the penultimate inequality follow from the hypothesis \eq{122g}.
Now, since $m$ is convex and $m'(0)=0$, we find that $m'(\theta)/\theta >0$ for $\theta \not=0$.  We now break the proof into 
two cases, corresponding to the positive and negative tails of the distribution of $W$.
\begin{enumerate}
\item[$\theta>0$.] In this case, our calculation above implies that for $0<\theta<1/B$,
\ba
\frac{d}{d\theta}\log(m(\theta))=\frac{m'(\theta)}{m(\theta)}\leq \frac{C\theta}{1-B\theta},
\ee
which yields that 
\ba
\log(m(\theta))\leq \int_0^\theta \frac{C u}{1-B u} du\leq  \frac{C \theta^2}{2(1-B\theta)},
\ee
and from this point we easily find 
\ba
m(\theta)\leq \exp\left\{\frac{C\theta^2}{2(1-B\theta)}\right\}.
\ee
According to Proposition \ref{prop211} we now have for $t>0$ and $0<\theta<1/B$,
\ba
\IP(W\geq t)\leq  \exp\left\{\frac{C\theta^2}{2(1-B\theta)}-\theta t\right\},
\ee
and choosing $\theta=t/(C+Bt)$ proves the first assertion of the theorem.
\item[$\theta<0$.]  In this case, since $m'(\theta)<0$, \eq{1222h} is bounded above by
$-C \theta m(\theta)$ which implies
\ba
C \theta \leq \frac{d}{d\theta} \log(m(\theta)) < 0.
\ee
From this equation, some minor consideration shows that
\ba
\log(m(\theta))\leq \frac{C \theta^2}{2}
\ee
According to Proposition \ref{prop211} we now have for $t>0$ and $\theta<0$,
\ba
\IP(W\leq -t)\leq  \exp\left\{\frac{C\theta^2}{2}+\theta t\right\},
\ee
and choosing $\theta=-t/C$ proves the second assertion of the theorem.
\end{enumerate}
\end{proof}
\begin{example}[Hoeffding's combinatorial CLT]
Let $(a_{i j})_{1\leq i,j \leq n}$ be an array of real numbers 
and let $\sigma$ be a uniformly chosen
random permutation of $\{1, \ldots, n\}$. Let
\ba
W=\sum_{i=1}^na_{i \sigma_j}-\frac{1}{n}\sum_{i,j}a_{i j},
\ee
and define $\sigma'=\sigma \tau$, where $\tau$ is a
uniformly chosen transposition and
\ba
W'=\sum_{i=1}^na_{i \sigma'_j}-\frac{1}{n}\sum_{i,j}a_{i j}.
\ee
It is a straightforward exercise to show that 
that $\IE[W]=0$ and $(W, W')$ is a $2/(n-1)$-Stein pair so that
it may be possible to apply Theorem \ref{thm211a}.  In fact, we have the following result.
\begin{proposition}
If $W$ is defined as above with $0\leq a_{i j}\leq 1$, then for all $t>0$, 
\ba
\IP(\abs{W}\geq t)\leq 2\exp\left\{\frac{-t^2}{\frac{4}{n}\sum_{i,j}a_{i,j}+2t}\right\}.
\ee
\end{proposition}
\begin{proof}
Let $(W, W')$ be the $2/(n-1)$-Stein pair as defined in the remarks preceding the statement of the proposition.  We now have
\ba
\IE[(W'-W)^2|\sigma]&=\frac{1}{n(n-1)}\sum_{i,j}\left(a_{i \sigma_i}+a_{j \sigma_j}-a_{i \sigma_j}-a_{j \sigma_i}\right)^2 \\
	&\leq \frac{2}{n(n-1)}\sum_{i,j}\left(a_{i \sigma_i}+a_{j \sigma_j}+a_{i \sigma_j}+a_{j \sigma_i}\right) \\
	&\leq \frac{4}{n-1}W+\frac{8}{n(n-1)}\sum_{i,j}a_{i j},
\ee
so that we can apply Theorem \ref{thm211a} with $B=1$ and $C=(2/n)\sum_{i,j} a_{i j}$, to prove the result.
\end{proof}
\end{example}

\subsection{Application: Magnetization in the Curie-Weiss model}

Let $\beta>0$, $h\in \IR$ and for  $\sigma
\in \{ -1,1 \}^n$ define the Gibbs measure 
\ban
\IP( \sigma ) = Z^{-1} \exp \left\{ \frac{\beta}{n} \sum_{i < j}
  \sigma_i \sigma_j + \beta h \sum_i \sigma_i \right\}, \label{2199}
\ee
where $Z$ is the appropriate normalizing constant (the so-called ``partition function" of statistical physics).

We think of $\sigma$ as a configuration of ``spins" ($\pm 1$) on a system with $n$ sites.  The spin of a site
depends on those at all other sites since for all $i\not=j$, each of the terms $\sigma_i \sigma_j$ appears in the first sum.
Thus, the most likely configurations are those that have many of the spins the same spin ($+1$ if $h>0$ and $-1$ if $h<0$).
This probability model is referred to as the Curie-Weiss model and a quantity of interest is
$m=\frac{1}{n}\sum_{i=1}^n \sigma_i$, the ``magnetization" of the system.  We will show the following result found in \cite{cha07a}.

\begin{proposition} \label{prop2112y}
If $m = \frac{1}{n} \sum_{i=1}^n \sigma_i$, then for all $\beta >
0$, $h \in \mathbb{R}$, and $t \geq 0$,
\[
\IP \left( |m- \tanh (\beta m + \beta h ) | \geq \frac{\beta}{n} +
  \frac{t}{\sqrt{n}} \right) \leq 2 \exp\left\{ - \frac{t^2}{4(1+ \beta)} \right\},
\]
where $\tanh(x):=(e^x-e^{-x})/(e^x+e^{-x})$.
\end{proposition}

In order to the prove the proposition, we need a more general result than that of Theorem \ref{thm211a}; 
the proofs of the two results are very similar.

\begin{theorem}\label{thm211b}
Let $(X, X' )$ an exchangeable pair of random
elements on a Polish space. Let $F$ an antisymmetric function and
define 
\[
f(X) := \IE [ F(X, X') |
X ].
\]
If $\IE [ e^{\theta f(X)} \abs{F(X,X')}] < \infty$ for all $\theta \in
\mathbb{R}$ and there are constants $B, C \geq 0$ such that  
\[
\frac{1}{2} \IE \left[ \left|(f(X)-f(X'))F(X,X') \right| \big|  X \right] \leq Bf(X) + C,
\]
then for all $t>0$,
\ba
\IP( f(X) > t) \leq \exp \left\{ \frac{-t^2}{2C+2Bt}
  \right\} \mbox{\, and \,} \IP( f(X) \leq -t) \leq \exp \left\{ \frac{-t^2}{2C}  \right\}.
\ee
\end{theorem} 
In order to recover Theorem \ref{thm211a} from the result above, 
if 
$(W,W')$
is an $a$-Stein pair,
then we can take $F(W,W') =(W-W')/a$, so that $f(W):=\IE [ F(W,W') | W] = W$.

\begin{proof}[Proof of Proposition \ref{prop2112y}]
In the notation of Theorem \ref{thm211b}, we will set $X = \sigma$ and
$X' = \sigma'$, where $\sigma$ is chosen according to the Gibbs measure
given by \eq{2199} and $\sigma'$ is a step from $\sigma$ in the
following reversible Markov chain: at each step of the chain a site
from the $n$ possible sites is chosen uniformly at random and then the
spin at that site is resampled according to the Gibbs measure \eq{2199} conditional on the value of the spins
at all other sites. This chain is most commonly known as the
\emph{Gibbs sampler}. 
We will take $F(\sigma, \sigma') = \sum_{i=1}^n (\sigma_i - \sigma_i')$
so that we will use Theorem \ref{thm211b} to study $f(\sigma) = \IE[ F(\sigma,\sigma')|\s ]$.

The first thing we need to do is compute the transition probabilities
for the Gibbs sampler chain. Suppose the chosen site is site $i$, then
\[
\begin{split}
\IP(\s'_i = 1 | (\s_j)_{j \neq i} ) & = \frac{ \IP ( \s_i=1,(\s_j)_{j \neq
    i} )}{\IP ( (\s_j)_{j \neq i} )} ,\\
\IP(\s'_i = -1 | (\s_j)_{j \neq i} ) & = \frac{ \IP ( \s_i=-1,(\s_j)_{j \neq
    i} )}{\IP ( (\s_j)_{j \neq i} )} . \\
\end{split}
\]

Note that $\IP ( (\s_j)_{j \neq i} ) = \IP ( \s_i = 1, (\s_j)_{j \neq i}
) + \IP ( \s_i = -1 , (\s_j)_{j \neq i} )$, and that 
\[
\begin{split}
\IP(\s'_i = 1 , (\s_j)_{j \neq i} ) & =  Z^{-1} \exp \left\{
  \frac{\beta}{n} \left( \sum_{k < j, j \neq i} \sigma_j \sigma_k
    + \sum_{j \neq i} \s_j \right) + \beta h \sum_{j \neq i} \sigma_j
  + \beta h \right\}    ,\\
\IP(\s'_i = -1 , (\s_j)_{j \neq i} ) & =  Z^{-1} \exp \left\{
  \frac{\beta}{n} \left( \sum_{k < j, j \neq i} \sigma_j \sigma_k
    - \sum_{j \neq i} \s_j \right) + \beta h \sum_{j \neq i} \sigma_j
  - \beta h \right\}   . \\
\end{split}
\]

Thus
\[
\begin{split}
\IP(\s'_i = 1 | (\s_j)_{j \neq i} ) & = \frac{ \exp 
\left\{ \frac{\beta}{n} \sum_{j \neq i} \s_j + \beta h \right\} }
{ \exp\left\{ \frac{\beta}{n} \sum_{j \neq i} \s_j + \beta h \right\} 
+ \exp\left\{ -\frac{\beta}{n} \sum_{j \neq i} \s_j - \beta h \right\} } ,\\
\IP(\s'_i = -1 | (\s_j)_{j \neq i} ) & = \frac{ \exp 
\left\{ -\frac{\beta}{n} \sum_{j \neq i} \s_j - \beta h \right\} }
{ \exp\left\{ \frac{\beta}{n} \sum_{j \neq i} \s_j + \beta h \right\} 
+ \exp\left\{ -\frac{\beta}{n} \sum_{j \neq i} \s_j - \beta h \right\} } ,
\\
\end{split}
\]
and hence
\begin{equation}
\label{eq1}
\IE[F(\s,\s')|\s] = \frac{1}{n} \sum_{i=1}^n \s_i - \frac{1}{n}
\sum_{i=1}^n \tanh \left( \frac{\beta}{n} \sum_{j \neq i} \s_j + \beta
  h \right) ,
\end{equation}
where the summation over $i$ and the factor of $1/n $ is due
to the fact that the resampled site is chosen uniformly at random
(note also that for $j\neq i$, $\IE
[ \s_j-\s_j'|\sigma \text{ and chose site $i$}] = 0$).

We will give a concentration inequality for \eqref{eq1} using 
Theorem \ref{thm211b}, and then
show that the difference between \eq{eq1} and the quantity of
interest in the proposition is almost surely bounded by a small quantity,
which will prove the result.

If we denote $m_i := \frac{1}{n} \sum_{j \neq i} \sigma_j$, then
\[
f(\s) := \IE[F(\s,\s')|\s] = 
m - \frac{1}{n} \sum_{i=1}^n \tanh \{ \beta m_i + \beta h \},
\]
and we need to check the conditions of Theorem \ref{thm211b}
for $f(\s)$.
The
condition involving the moment generating function is
obvious since all quantities involved are finite, so we 
only need to find constants $B, C>0$ such that  
\begin{equation}
\label{eq2}
\frac{1}{2} \IE \left[ \left|
\left(f(\s)-f(\s')\right)F(\s,\s') \right| \big| \s \right] \leq Bf(\s) + C.
\end{equation}
Since $F(\sigma, \sigma')$ is the difference of the sum of the spins
in one step of the Gibbs sampler and only one spin can change 
in a step of the chain, we have
$|F(\s,\s')|\leq 2$.

Also, if we denote $m' := \frac{1}{n} \sum_{i=1}^n \s'$, then
using that  $| \tanh(x) - \tanh(y) | \leq |x-y|$ (which essentially follows
from the inequality \eq{2212a}: $2|e^x-e^y| \leq
  |x-y| (e^x + e^y) $), we find
\ba
|f(\s)-f(\s')| \leq |m-m'| + \frac{\beta}{n} \sum_{i=1}^n |m_i - m_i'| \leq \frac{2(1+
  \beta)}{n} ,
\ee

Hence, \eqref{eq2} is satisfied with $B=0$ and $C =
\frac{2(1+\beta)}{n}$ and Theorem \ref{thm211b} now yields
\[
\IP \left( \left|m- \frac{1}{n} \sum_{i=1}^n \tanh (\beta m_i + \beta h ) \right| > 
  \frac{t}{\sqrt{n}} \right) \leq 2 \exp\left\{ - \frac{t^2}{4(1+ \beta)} \right\} .
\]

To complete the proof we note that
\[
\bigg| \frac{1}{n} \sum_{i=1}^n \left[ \tanh (\beta m_i + \beta h ) - \tanh
(\beta m + \beta h ) \right] \bigg| \leq \frac{1}{n} \sum_{i=1}^n |
\beta m_i - \beta m | \leq \frac{\beta}{n} ,
\]
and thus an application of  the triangle inequality yields the bound in the
proposition.
\end{proof}

\subsection{Concentration using size-bias couplings}

As previously mentioned, the key step in the proof of Theorem \ref{thm211a} was to rewrite
$m'(\theta):=\IE[We^{\theta W}]$ using exchangeable pairs in order to get a differential inequality
for $m(\theta)$.  We can follow this same program, but with a size-bias coupling in place of 
the exchangeable pair.  We follow the development of \cite{ghgo11}.

\begin{theorem}\label{thm211c}
Let $X\geq 0$ with $\IE [X] = \mu$ and $0<\var(X) = \sigma^2<\infty$ and let $X^s$ be a size-biased coupling of $X$ such that $|X-X^s|\leq C<\infty$.
\begin{enumerate}
\item If $X^s\geq X$, then 
\ba
\IP\left(\frac{X-\mu}{\sigma} \leq -t\right) \leq \exp\left\{\frac{-t^2}{2 \left(\frac{C\mu}{\sigma^2}\right)}\right\}.
\ee
\item If $m(\theta) = \IE[e^{\theta X}] <\infty$ for $\theta = 2/C$, then
\ba
\IP\left(\frac{X-\mu}{\sigma} \geq t\right) \leq  \exp\left\{ \frac{-t^2}{2\left(\frac{C\mu}{\sigma^2} + \frac{C}{2\sigma} t\right)}\right\}.
\ee
\end{enumerate}
\end{theorem}

\begin{proof}
To prove the first item,
let $\theta\leq 0$ so that $m(\theta):=\IE[e^{\theta X}] <\infty$ since $X\geq 0$.  As in the proof of Theorem \ref{thm211a}, 
we will need the inequality \eq{2212a}: for all $x,y\in \IR$,
\ba
\left|\frac{e^x-e^y}{x-y}\right| \leq \frac{e^x+e^y}{2}.
\ee
Using this fact and that $X^s\geq X$, we find
\ban
\IE [e^{\theta X} -  e^{\theta X^s}] \leq \frac{C|\theta|}{2} \left( \IE [e^{\theta X}] + \IE [e^{\theta X^s}]\right)
 \leq C|\theta| \IE [e^{\theta X}]. \label{2122c}
\ee

The definition of the size-bias distribution implies that $m'(\theta)=\mu\IE[e^{\theta X^s}]$ so that \eq{2122c} yields
the differential inequality $m'(\theta) \geq \mu(1+C\theta)m(\theta)$, or put otherwise
\ban
\frac{d}{d\theta} \left[\log(m(\theta))-\mu\theta\right] \geq \mu C\theta. \label{2122n}
\ee
Setting $\tilde m(\theta) = \log(m(\theta))-\mu\theta$, \eq{2122n} implies
$\tilde m(\theta)\leq \mu C\theta^2/2$, and it follows that 
\[ \IE\left[ \exp\left\{\theta\left(\frac{X-\mu}{\sigma}\right)\right\} \right]= m\left(\frac{\theta}{\sigma}\right) \exp\left\{-\frac{\mu\theta}{\sigma}\right\}
= \exp\left\{\tilde m\left(\frac{\theta}{\sigma}\right)\right\} \leq \exp\left(\frac{\mu C\theta^2}{2\sigma^2}\right).\]
We can now apply Proposition \ref{prop211} to find that 
\ban
\IP\left(\frac{X-\mu}{\sigma} < -t\right) \leq \exp\left(\frac{\mu C \theta^2}{2\sigma^2} + \theta t\right). \label{2134a}
\ee
The right hand side of this \eq{2134a} is minimized at $\theta  = -\sigma^2 t/\mu C$, and substituting this value into \eq{2134a} yields
the first item of the theorem.

For the second assertion of the theorem, suppose that $0\leq \theta < 2/C$.  A calculation similar to \eq{2122c} above shows that 
\[\frac{m'(\theta)}{\mu} -m(\theta) \leq \frac{C\theta}{2} \left( \frac{m'(\theta)}{\mu} +m(\theta)\right),\]
so that we can write
\[ m'(\theta) \leq \frac{\mu \left(1+\frac{C\theta}{2}\right)}{1-\frac{C\theta}{2}} m(\theta).\]
Again defining $\tilde m(\theta) = \log(m(\theta))-\mu\theta$, we have $\tilde m'(\theta) \leq C\mu\theta/(1-\frac{C\theta}{2})$ so that
\[\tilde m\left(\frac{\theta}{\sigma}\right) \leq  \frac{ C\mu\theta^2}{\sigma^2\left(2-\frac{C\theta}{\sigma}\right)} \quad\textrm{for } \ 0\leq \theta < \min\{2/C,2\sigma/C\}.\]
We can now apply Proposition \ref{prop211} to find that 
\ban
\IP\left(\frac{X-\mu}{\sigma} \geq t\right) \leq \exp\left(\frac{\mu C \theta^2}{\sigma^2\left(2-\frac{C\theta}{\sigma}\right)} - \theta t\right). \label{2134b}
\ee
The right hand side of this \eq{2134b} is minimized at
\[\theta = t\left(\frac{C\mu}{\sigma^2}+\frac{Ct}{2\sigma}\right)^{-1},\]
and substituting this value into \eq{2134b} yields
the second item of the theorem.
\end{proof}

Theorem \ref{thm211c} can be applied in many of the examples we have discussed in the context of
the size-bias transform for normal and Poisson approximation and others.  We content ourselves with 
a short example and refer to \cite{ghgo11a} for many more applications.

\begin{example} 
Let $Y_1,\dots, Y_n$ be i.i.d.\! $Be(p)$, fix $k\geq 1$, and define $X_i = \prod_{j=i}^{i+k-1} Y_j$ with the bounds being modular.  Further define $X=\sum_{i=1}^n X_i$ and define $X^s$ by sampling $X_1,\dots, X_n$ and forcing $X_1=1$.  Conditional on this, the rest of the $Y_i$ are as before.  If 
\[ X^{(k)}_i = \begin{cases} 1 & \textrm{if there exists a head run of length $k$ at $i$ after forcing} \\ 0 & \textrm{otherwise},\end{cases}\]
then $X^s = 1+\sum_{j=2}^n X^{(k)}_j$.  Note that $X^s\geq X$ and $|X^s-X|\leq 2k-1$.  In this case Theorem \ref{thm211c} implies
\[ \IP\left(\left|\frac{X-\mu}{\sigma}\right| \geq t\right) \leq 2 \exp\left( \frac{-t^2}{2(2k-1)\left(\frac{\mu}{\sigma^2} + \frac{t}{2\sigma} \right)}\right),\]
where $\mu = np^k$ and $\sigma^2 = \mu\left(1-p^k+\sum_{i=1}^{k-1} (p^i-p^k)\right)$.
\end{example}

\subsubsection*{Acknowledgments}
The author thanks the students who sat in on the course that this document is based.  Special
thanks go to 
those students who contributed notes, parts of which may appear in some form above:
Josh Abramson, Miklos Racz, Douglas Rizzolo, and Rachel Wang.

\bibliographystyle{abbrv}
\bibliography{Ref.2011.03}

\end{document}